\documentclass[a4paper,10pt]{article}
\usepackage{amsmath,amssymb,amsfonts,amsthm}
\usepackage{graphicx,color,subfigure}
\usepackage{enumerate,multirow}
\usepackage{fullpage}
\usepackage[]{hyperref}
\usepackage[normalem]{ulem}

\newcommand\RR{\ensuremath{\mathbb{R}}}

\newcommand\ZZ{\ensuremath{\mathbb{Z}}}
\newcommand\NN{\ensuremath{\mathbb{N}}}
\newcommand\eps{\ensuremath{\varepsilon}}

\newcommand\Tor{\mathsf{T}}
\newcommand\Hex{\mathsf{H}}
\newcommand\Str{\mathsf{S}}
\newcommand\Rec{\mathsf{R}}

\newcommand\bH{b^{\Hex}}
\newcommand\bS{b^{\Str}}
\newcommand\bC{b^{\mathsf c}}

\newtheorem{thm}{Theorem}[section]
\newtheorem{prop}[thm]{Proposition}

\newtheorem{lem}[thm]{Lemma}
\newtheorem{defin}[thm]{Definition}
\newtheorem{rem}[thm]{Remark}
\newtheorem{conj}[thm]{Conjecture}

\makeatletter

\@addtoreset{equation}{section}
\makeatother

\newenvironment{proofof}[1]{\begin{trivlist}
      \item[]\hspace{0cm}{\it Proof of #1.}
      \hspace{0cm}}{\hfill $\square$
      \end{trivlist}}

\definecolor{gr}{rgb}   {0.,   0.69,   0.23 }
\definecolor{cy}{rgb}   {0.,   0.57,   0.67 }
\definecolor{mg}{rgb}   {0.85,  0.,    0.85}
\definecolor{marron}{rgb}  {0.6, 0.40, 0.1} 
\definecolor{or}{rgb}   {0.8,  0.4,   0.}

\title{Spectral minimal partitions for a family of tori}
\author{Virginie Bonnaillie-No{\"e}l\footnote{Virginie Bonnaillie-No\"el, D\'epartement de Math\'ematiques et Applications (DMA UMR 8553), PSL Research University, CNRS, ENS Paris, 45 rue d'Ulm, F-75230 Paris Cedex 05, France, \texttt{virginie.bonnaillie@ens.fr}} 
\ and Corentin L\'ena\footnote{ Corentin L{\'e}na, Dipartimento di Matematica \emph{Giuseppe Peano}, Universit\`a Degli studi di Torino, Via Carlo Alberto, 10, 10123 Torino (TO), Italia, 
\texttt{clena@unito.it}}}

\begin{document}

\maketitle

\begin{abstract}
We study partitions of the two-dimensional flat torus $\left(\RR/\ZZ\right)\times\left(\RR/b\ZZ\right)\,$ into $k$ domains, with $b$ a real parameter in $(0,1]$ and $k$ an integer. We look for partitions which minimize the energy, defined as the largest  first eigenvalue of the Dirichlet Laplacian on the domains of the partition. We are in particular interested in the way these minimal partitions change when $b$ is varied. 
We present here an improvement, when $k$ is odd, of the results on transition values of $b$ established by B.~Helffer and T.~Hoffmann-Ostenhof (2014)  in \cite{HelHof14} and state a conjecture on those transition values. We establish an improved upper bound of the minimal energy by explicitly constructing hexagonal tilings of the torus. 
These tilings are close to the partitions obtained from a systematic numerical study based on an optimization algorithm adapted from B.~Bourdin, D.~Bucur, and {\'E}.~Oudet (2009) in \cite{BouBucOud09}. These numerical results also support our conjecture concerning the transition values and give better estimates near those transition values.


\end{abstract}

\section{Introduction}
\subsection{Motivation}
Optimal partition problems are a field of shape optimization which has recently generated much interest, see e.g. \cite{BucButHen98,BucBut05,CafLin07,HelHofTer09}. They are connected to nodal sets of Laplacian eigenfunctions (see e.g. \cite{HelHofTer09}) and to steady states for competition-diffusion systems of partial differential equations (see e.g. \cite{ConTerVer05c}). This paper focus on a particular type of problem, studied by B.~Helffer, T.~ Hoffmann-Ostenhof, and S.~Terracini in \cite{HelHofTer09}.

Let us describe the general setting. In the following, $M$ is a compact, two-dimensional, Riemannian manifold without boundary, and $k$ an integer, $k\ge 2\,$. 
All along this paper, we consider $k$-partitions of $M$ in the following sense. 
\begin{defin}
\label{definPartition}
We call \emph{$k$-partition} (or simply \emph{partition}) a finite family $\mathcal{D}=(D_i)_{1\le i\le k}$ of open, connected, and mutually disjoint subsets of $M$ (called \emph{domains} of the partition). 
This partition is \emph{strong} if $M=\bigcup_{i=1}^k \overline{D_i}\,$. In that case, we can define the \emph{boundary} of $\mathcal{D}$ as $N(\mathcal{D})=\bigcup_{i=1}^k \partial D_i\,$. We then say that $\mathcal{D}$ is \emph{regular} if $N(\mathcal{D})$ is locally a regular curve, except at a finite number of singular points, where a finite number of half-curves meet with equal angles (the 'equal angle meeting property'). \\
We denote by $\mathfrak{P}_k$ the set of all $k$-partitions. 
\end{defin}

For any $k$-partition $\mathcal{D}\in \mathfrak{P}_k\,$, we define its \emph{energy} by
\begin{equation}\label{eq.nrj}
\Lambda_{k}(\mathcal{D})=\max_{1\le i \le k} \lambda_1(D_i)\,,
\end{equation}
where, for any open set $\omega \subset M\,$, we denote by $(\lambda_j(\omega))_{j \ge 1}$ the eigenvalues of the Laplacian on $\omega\,$, with Dirichlet boundary conditions, arranged in non-decreasing order and counted with multiplicities. 
The optimization problem we consider here is to study
\begin{equation}\label{eq.optim} 
\mathfrak{L}_k(M)=\inf_{\mathcal{D} \in \mathfrak{P}_k}\Lambda_k(\mathcal{D})\,.
\end{equation}
We say that a partition $\mathcal{D}^*\in \mathfrak{P}_k$ is \emph{minimal} when $\Lambda_k(\mathcal{D}^*)=\mathfrak{L}_k(M)\,$. The existence and regularity of minimal partitions have been established in \cite{BucButHen98,ConTerVer05b,CafLin07,HelHofTer09}, and we have 
\begin{thm}
\label{thmMinPartReg}
Let $\mathcal{D}=(D_i)_{1\le i\le k}$ be a  minimal $k$-partition. Up to zero capacity sets, $\mathcal{D}$ is strong and regular. 
Furthermore, $\mathcal{D}$ is \emph{equispectral}, that is to say $\Lambda_k(\mathcal{D})=\lambda_1(D_j)$ for any $1\leq j\leq k\,$.
\end{thm}

In the subsequent paper \cite{HelHof14}, B.~Helffer and T.~Hoffmann-Ostenhof considered the case where $M$ is the two-dimensional flat torus $\Tor(1,b)=\left(\mathbb{R}/\mathbb{Z}\right)\times\left(\mathbb{R}/b\mathbb{Z}\right)\,$, with $b$ a parameter in $(0,1]\,$. We also consider this situation in the present paper. In particular, we are concerned with finding numerically the minimal partitions of $\Tor(1,b)\,$, for a range of values of $b$\,. Minimal partition problems on surfaces have already been investigated numerically by several authors. The paper \cite{BouBucOud09} treats the case of the torus $T(1,1)$\,, and was the inspiration for our numerical study, while \cite{ElliottRanner2014Surfaces} considers several surfaces embedded in $\RR^3$\,, including the sphere. A more complete study of the sphere, using a different method, is presented in \cite{Bogosel2015MFS}. In all these papers, the energy to be minimized is the sum of the eigenvalues rather than the maximum.

\subsection{Nodal partitions}

We will exploit the connection, shown in \cite{HelHofTer09}, between minimal partitions and nodal domains of eigenfunctions of the Laplacian. Let us recall some relevant definitions and results. With an eigenfunction $u$ of the Laplacian on $M\,$, we associate the \emph{nodal domains} which are the connected components, denoted by $D_{i}\,$, of $M \setminus N(u)$ with $N(u)=\{x \in M\,;\, u(x)=0\}\,$. The number $\mu(u)$ of connected components is finite.
According to classical results on the regularity of the nodal set, the family $(D_i)_{1 \le i \le \mu(u)}$ of the nodal domains is a strong regular $\mu(u)$-partition in the sense of Definition \ref{definPartition}.  We call it the \emph{nodal partition} of $u\,$. Let us note that, for a nodal partition, an even number of half-curves meet at each singular point in the boundary. This is in contrast with the situation for minimal partitions, where there is no constraint on this number. \\
A famous result, proved by R.~Courant \cite{CouHil53}, states that if $u$ is an eigenfunction associated with $\lambda_j(M)\,$,  $\mu(u) \le j\,$. 
Following \cite{HelHofTer09}, we introduce a new definition for the case of equality.
\begin{defin}
An eigenfunction $u$ of the Laplacian, associated with the eigenvalue $\lambda\,$, is said to be \emph{Courant-sharp} if $\mu(u)=\min\{\ell\,;\,\lambda_{\ell}(M)=\lambda\}\,$.
\end{defin}
\medskip
We can now state a result of \cite{HelHofTer09} that links minimal and nodal partitions.

\begin{thm}
 \label{thmNodalMinimal}
 We have $\lambda_k(M) \le\mathfrak{L}_k(M)\,$, and therefore the nodal partition of a Courant-sharp eigenfunction is minimal. Conversely, if the nodal partition of some eigenfunction is minimal, this eigenfunction is Courant-sharp. Finally, if $\mathfrak{L}_k(M)=\lambda_k(M)\,$, all minimal $k$-partitions are nodal.
\end{thm}

\begin{rem} Theorem \ref{thmNodalMinimal} implies that $\mathfrak{L}_k(M)=\lambda_k(M)$ if there exists a nodal minimal $k$-partition, and $\mathfrak{L}_k(M)>\lambda_k(M)$ otherwise.
\end{rem}
 Let us note that eigenfunctions associated with $\lambda_2(M)\,$ always have two nodal domains. Therefore, according to Theorem \ref{thmNodalMinimal}, the notion of minimal $2$-partition coincides with the notion of nodal partition of a second eigenfunction. The problem of finding minimal $k$-partitions becomes interesting for $k\ge 3\,$, which we will assume in the rest of the paper.

\subsection{Summary of the results}

To use the nodal partitions, let us give explicitly the eigenvalues of the Laplacian on the general torus
\begin{equation*}
	\Tor(a,b)= \left(\RR/a\ZZ\right)\times \left(\RR/b\ZZ\right)\,,\qquad \mbox{ with }\quad 0<b \leq a\,.
\end{equation*}
The Laplacian on $\Tor(a,b)$ is unitarily equivalent to the Laplacian on the rectangle $\Rec(a,b)=(0,a)\times(0,b)$ with periodic boundary conditions, and its spectrum can be computed by separation of variables.
\begin{prop}
\label{propNodalTorus}
The eigenvalues of the Laplacian on $\Tor(a,b)$ are
\[\lambda_{m,n}(a,b)=4\pi^2\left(\frac{m^2}{a^2}+\frac{n^2}{b^2}\right)\,,\quad\mbox{with}\quad m,\, n \in \NN_{0}\,.\]
\end{prop}
Let us recall that, according to \cite{Len15Torus}, the only non-constant Courant-sharp eigenfunctions for the torus $\Tor(1,1)$ are associated with $\lambda_2(\Tor(1,1))=4\pi^2\,$. According to Theorem \ref{thmNodalMinimal}, this implies that, as soon as $k \ge 3\,$, a  minimal $k$-partition of $\Tor(1,1)$ is not nodal and we have to find new candidates.\\ 
Let us introduce a particular partition of $\Tor(a,b)$ into vertical strips.
 \begin{defin}\label{definDk}
  We denote by $\mathcal{D}_{k}(a,b)$ the $k$-partition of $\Tor(a,b)$ with domains
  \begin{equation*}
 D_i=\left(\frac{i-1}{k}a,\frac{i}{k}a\right)\times\left(0,b\right)\,,\qquad \mbox{ for }i=1,\dots, k\,.
\end{equation*}
\end{defin} 

We have $\Lambda_{k}(\mathcal{D}_{k}(a,b))=k^2\pi^2/a^2\,$, and any partition obtained from $\mathcal{D}_k(a,b)$ by a translation has the same energy. Let us note that when $k$ is even, $\mathcal{D}_k(a,b)$ is the nodal partition of the eigenfunction $(x,y)\mapsto \sin(k\pi x/a)\,$, associated with the eigenvalue $\lambda_{k/2,0}(a,b)$.

Let us now focus on the torus $\Tor(1,b)\,$, with $b\in (0,1]\,$. Following \cite{HelHof14}, we want to know for which values of $b$ the partition $\mathcal{D}_k(1,b)$ is minimal.
Then we define the \emph{transition value} $b_k$ by
\begin{equation}\label{eq.bk}
 b_k=\sup\{b>0~;~\mathcal{D}_{k}(1,b)\mbox{ is a minimal $k$-partition of } \Tor(1,b) \}\,.
 \end{equation}
This notion is well adapted since $\mathcal{D}_{k}(1,b)$ is a minimal $k$-partition of $\Tor(1,b)$, for any $b\in(0,b_{k}]$, as it will be established in Proposition~\ref{prop.bk}.\\

When $k$ is even, a direct application of  Theorem \ref{thmNodalMinimal} gives us the transition value (this result and its detailed proof can be found in \cite{HelHof14}).
\begin{prop}\label{propPartEven}
If $k$ is even, then $b_{k}=2/k\,$. Furthermore, if $b < 2/k\,$, $\mathcal{D}_k(1,b)$ is, up to a translation, the only minimal $k$-partition of $\Tor(1,b)\,$.  
\end{prop}
Let us note that if $k$ is even and $b=b_{k}\,$, $\mathcal{D}_k(1,b)$ is no longer the only minimal partition of $\Tor(1,b)$ up to a translation, due to the multiplicity of the eigenvalue $\lambda_{k}(\Tor(1,b))$\,. We will see this in more detail in Section \ref{subsecTransVal}.\\

When $k$ is odd, \cite[Theorem 1.1]{HelHof14} proves that $b_{k}\geq 1/k\,$. 
Before stating our improvement of this estimate, let us introduce some notation. For $b\in (0,1]\,$, we consider the infinite strip $\Str_b=\RR\times \left(0,b\right)\,$ and we define
\begin{equation}\label{eq.optimibkFK}
	\bS_k =\sup\left\{b\in(0,1]\,;\, j(b)>k^2\pi^2\right\}\,,\qquad \mbox{ with }\qquad
	j(b)=\inf_{\Omega\subset \Str_b, |\Omega|\le b}\lambda_1(\Omega)\,.
\end{equation}
\begin{thm}
\label{thmbkOdd}
If $k$ is odd, then $b_{k}\geq \bS_{k}>1/k$.
\end{thm}

We are also interested in obtaining upper bounds of $\mathfrak{L}_k(\Tor(1,b))$ for $b \in (0,1]\,$. Since $\Lambda_k(\mathcal{D}_k(1,b))= k^2\pi^2\,$, we always have $\mathfrak{L}_k(\Tor(1,b))\le k^2\pi^2\,$. For some values of the parameters, we construct hexagonal partitions which give us an improved upper bound.

\begin{thm}
\label{thmExkPart}
For $k\in\{3\,,\,4\,,\,5\}\,$, there exists $\bH_k \in (0,1)$ such that, for any $b \in (\bH_k,1]\,$, there exists a tiling of $\Tor(1,b)$ by $k$ hexagons that satisfies the equal angle meeting property. We denote by $\Hex_k(b)$ the corresponding tiling domain, and we have
\[\mathfrak{L}_k(\Tor(1,b))\le \min\left(k^2\pi^2, \lambda_1(\Hex_k(b))\right)
,\quad \forall b\in(\bH_{k},1] \,.\]
More explicitly, we can choose 
\[\bH_3=\frac{\sqrt{11}-\sqrt{3}}{4}\simeq 0.396\,,\quad \bH_4=\frac{1}{2\sqrt{3}}\simeq 0.289< b_{4}=\frac 12\,,\quad\mbox{and }\quad \bH_5=\frac{\sqrt{291}-5\sqrt{3}}{36}\simeq 0.233\,.\]
\end{thm}

The paper is organized as follows. In Section  \ref{secTransitions}, we analyse the transition values and  give a proof of Theorem \ref{thmbkOdd}. We also set out some arguments to conjecture the transition values and for such $b$, we present some candidates to be minimal partitions, which are obtained from eigenfunctions of the Laplacian on a covering of the torus. In Section \ref{secNumerics}, we describe our numerical method, which is based on the work of B.~Bourdin, D.~Bucur, and {\'E}.~Oudet (2009) in \cite{BouBucOud09}. For a fixed $ k\in\{3\,,\,4\,,\,5\}$, we compute candidates to be a minimal $k$-partition of $\Tor(1,b)\,$. These computations suggest the existence of hexagonal tilings of a specific type. In Section \ref{secTilings}, we construct explicitly  these tilings and compute their energy. This improves the previously known bounds of the minimal energy. Near the conjectured transition values and near $b=1\,$, the numerical simulations give us better candidates.

\section{Transitions between different types of minimal partitions}
\label{secTransitions}

\subsection{Transition value}
\label{subsecTransDef}

Let us first show that the notion of transition value $b_{k}$ introduced in \eqref{eq.bk} is well behaved.
\begin{prop}\label{prop.bk}
For any $b\in(0,b_{k}]\,$, $\mathcal{D}_{k}(1,b)$ is a minimal $k$-partition of $\Tor(1,b)\,$.
\end{prop}
The proof is a direct consequence of the following properties of $b \mapsto \mathfrak{L}_k(\Tor(1,b))\,$:

\begin{prop} \label{propMonotonicity} The function $b\mapsto \mathfrak{L}_k(\Tor(1,b))\,$, defined on $(0,1]\,$, is continuous and  non-increasing.
\end{prop}

\begin{proof}
 Let us pick $b$ and $b'$ in $(0,1]\,$, with $b'<b\,$. We define a mapping $F$ from $\Tor(1,b')$ to $\Tor(1,b)$ by
 \[\begin{array}{cccc}
       F: & \Tor(1,b')& \rightarrow & \Tor(1,b)\\
            & (x,y) & \mapsto & \left(x,\frac{b}{b'}y\right)\,.\\
  \end{array}\]
If $\omega$ is an open set in $\Tor(1,b')\,$, $F(\omega)$ is an open set in $\Tor(1,b)\, $. By a direct estimate of the Rayleigh quotients, using the change of variable defined by $F\,$, we obtain 
\begin{equation}
\label{eqIneqLambda1}
\lambda_1(F(\omega))\le\lambda_1(\omega)\le \frac{b^2}{b'^2}\lambda_1(F(\omega))\,.
\end{equation}
Let us now consider a minimal partition $\mathcal{D}'=(D_i')_{1\le i \le k}$ of $\Tor(1,b')\,$. We define the partition $F(\mathcal{D}')=(F(D_i'))_{1\le i \le k}\,$ of $\Tor(1,b)\,$. According to the inequality on the left in \eqref{eqIneqLambda1}, we have
\begin{equation*}
\mathfrak{L}_k(\Tor(1,b)) \le \Lambda_k(F(\mathcal{D}')) \le \Lambda_k(\mathcal{D}') = \mathfrak{L}_k(\Tor(1,b'))\,.
\end{equation*}
Since $b$ and $b'$ are arbitrary, this establishes monotonicity. Let us on the other hand consider a minimal partition $\mathcal{D}=(D_i)_{1\le i \le k}$ of $\Tor(1,b)\,$, and define the partition $F^{-1}(\mathcal{D})=(F^{-1}(D_i))_{1\le i \le k}$ of $\Tor(1,b')\,$. According to the inequality on the right in \eqref{eqIneqLambda1}, we have
\begin{equation*}
\mathfrak{L}_k(\Tor(1,b')) \le \Lambda_k(F^{-1}(\mathcal{D})) \le \frac{b^2}{b'^2}\Lambda_k(\mathcal{D}) = \frac{b^2}{b'^2}\mathfrak{L}_k(\Tor(1,b))\,.
\end{equation*} 
Since $b$ and $b'$ are arbitrary, this establishes continuity.
\end{proof}

\subsection{Proof of Theorem \ref{thmbkOdd}}
\label{subsecProofOdd}

We break down the proof into several lemmas. The structure of the argument follows closely \cite{HelHof14}, with two main changes. First, we have imposed the constraint $|\omega| \le b$ in the auxiliary optimization problem \eqref{eq.optimibkFK}, in addition to the inclusion constraint $\omega \subset \Str_b\,$. This allows us to exclude the existence of domains homeomorphic to a disk for a larger range of values of $b\,$, as discussed in Lemmas \ref{lembFK} and \ref{lemNoDisk}. We have also used the pair-symmetric structure of the lifted partition to obtain a better lower bound of the energy, as seen in Lemmas \ref{lemCourantSym} and \ref{lemLower}.

The definition of the optimization problem \eqref{eq.optimibkFK} implies that $b\mapsto j(b)$ is non-increasing with respect to $b$\,. According to the Faber-Krahn inequality, the disk is a minimizer for sufficiently large $b\,$. For a small $b\,$, the following lower bound, deduced from the Poincar{\'e} inequality on $\Str_b\,$, gives more information (see \cite{HelHof14}):
\[j(b)\ge \frac{\pi^2}{b^2}\,.\]
It can be shown, using a concentration-compactness result for shapes proved by D. Bucur (see \cite{Buc00}), that there exists a quasi-open minimizer $\Omega^*$ for Problem \eqref{eq.optimibkFK}, which satisfies $|\Omega^*|=b\,$. This allows us to obtain an estimate of $\bS_k\,$. 
\begin{lem} 
\label{lembFK}
We have
\[\frac{1}{k} <\bS_k <\frac{1}{\sqrt{k^2-1}}\,.\]
\end{lem}

\begin{proof} 
An elementary argument using monotonicity, similar to the proof of Proposition \ref {propMonotonicity}, shows that the function $b\mapsto j(b)$ is continuous. Furthermore, in the case $b=1/k\,$, we have seen that there exists a quasi-open set $\Omega^*\subset \Str_b$ such that $\lambda_1(\Omega^*)=j(1/k)$ and $|\Omega^*|=1/k\,$. This condition on the area implies in particular that $j(1/k)=\lambda_1(\Omega^*)>k^2\pi^2\,$, by strict monotonicity of the first Dirichlet-Laplacian eigenvalue. We conclude that $\bS_k>1/k\,$. \\
To obtain the upper bound, let us consider the rectangle
\[\Rec(1,b)=\left(0,1\right)\times \left(0,b\right)\,.\]
We have $\Rec(1,b)\subset \Str_b$ and $|\Rec(1,b)|=b\,$. If $b=1/\sqrt{k^2-1}\,$, we have $\lambda_1(\Rec(1,b))=k^2\pi^2\,$. Furthermore, we know that $\Rec(1,b)$ is not minimal, since the normal derivative of a first eigenfunction is not constant on the free boundary. We conclude that $\bS_k<1/\sqrt{k^2-1}\,$.
\end{proof}

We now prove a topological property of minimal partitions.
\begin{lem} 
\label{lemNoDisk}
If $b<\bS_k\,$, a minimal partition of $\Tor(1,b)$ has no domain homeomorphic to a disk.
\end{lem}

\begin{proof} 
Let us assume that some minimal partition $\mathcal{D}$ has one such domain, that we denote by $D$. We consider the following covering of $\Tor(1,b)$ by the plane $\RR^2\,$:
	\[\begin{array}{cccc}
			\Pi_{\infty}:& \RR^2 & \rightarrow & \Tor(1,b)\\
					&(x,y)& \mapsto    & (x \mbox{ mod } 1, y \mbox{ mod } b)\,.\\
	  \end{array}
	  \]
Let $D_0$ be one of the connected components of $\Pi_{\infty}^{-1}(D)\,$.  It is homeomorphic to a disk, and $|D_0|= |D|\le b\,$. Furthermore, since $\Tor(1,b)$ is of width $b\,$, for any $x_0 \in \RR\,$, the total length of the vertical slice at $x_0\,$, that is to say of the set $\{y\,;\, (x_0,y)\in D_0\}\subset \RR\,$, is smaller than $b\,$. Let us call $D_0^{S}$ the  Steiner symmetrization of $D_0$ with respect to the line $\{(x,y)\,;\,y=1/2\}\,$. It has the same area as $D_0\,$, and it is contained in $\Str_b\,$, according to the geometrical property mentioned above. Since Steiner symmetrization does not increase the first eigenvalue, we obtain
\[j(b)\le \lambda_1(D_0^S)\le \lambda_1(D_0)=\lambda_1(D)=\Lambda_k(\mathcal{D})\,.\]
Since $b<\bS_k\,$, $j(b)>k^2\pi^2=\Lambda_k\left(\mathcal{D}_k(1,b)\right)\,$, contradicting the minimality of $\mathcal{D}\,$.  
\end{proof}	  

 Following \cite{HelHof14}, we consider the torus 
 \begin{equation*}					
   \Tor(2,2b)=\left(\RR/2\ZZ\right)\times\left(\RR/2b\ZZ\right)
 \end{equation*}
 equipped with the natural projection map $\Pi_4: (x,y)\mapsto (x \mbox{ mod }1,y\mbox{ mod }b)$ from $\Tor(2,2b)$ to $\Tor(1,b)\,$. Since every point of $\Tor(1,b)$ has four antecedents by $\Pi_4\,$, $\Tor(2,2b)$ can be seen as a four-sheeted covering of $\Tor(1,b)\,$. The pull-back $\Pi_4^{-1}(D)$ of a connected open set $D$ in $\Tor(1,b)$ is an open set in $\Tor(2,2b)$ having at most four connected components. With any $k$-partition $\mathcal{D}=(D_i)_{1\le i\le k}\,$ of $\Tor(1,b)$,  we can therefore associate the partition whose domains are all the connected components of all the sets $\Pi_4^{-1}(D_i)\,$, for $i\in \{1,\,\dots,\,k\}\,$. We call it the partition \emph{lifted from $\mathcal{D}$} and denote it by $\Pi_4^{-1}(\mathcal{D})\,$. It is a regular $\ell$-partition of $\Tor(2,2b)\,$, with $k\le\ell\le 4k\,$, and it has the same energy as $\mathcal{D}\,$.
 
 Let us now define the following mapping on $\Tor(2,2b)\,$: 
 \[\begin{array}{cccc}
 \sigma: & \Tor(2,2b) & \to     & \Tor(2,2b)\\
 & (x,y)   & \mapsto & (x+1 \mod 2,y)\,.\\
 \end{array}\]
 We have $\Pi_4(\sigma((x,y)))=\Pi_4((x,y))$ for all $(x,y) \in \Tor(2,2b)\,$. We say that $u\in L^2(\Tor(2,2b))$ is \emph{antisymmetric} if $u\circ \sigma=-u$ and we denote by $\mathcal{A}_{\sigma}$ the space of antisymmetric functions. 
 
 Let us note that for any function $u$ on $\Tor(2,2b)\,$, $-\Delta (u\circ \sigma)= \left(-\Delta u\right)\circ \sigma\,$, so that $\mathcal{A}_{\sigma}$ is stable under the action of $-\Delta\,$. We write  $H^{a}_{\sigma}$ for the Friedrichs extension of the differential operator $-\Delta$ acting on  $C^{\infty}(\Tor(2,2b))\cap \mathcal{A}_{\sigma}\,$. The operator $H^{a}_{\sigma}$ is self-adjoint, with domain $H^2(\Tor(2,2b))\cap \mathcal{A}_{\sigma}\,$, and has compact resolvent. We denote by $(\lambda_k^{\sigma,a})_{k \ge 1}$ the sequence of its eigenvalues, which we will call antisymmetric, arranged in non-decreasing order and counted with multiplicities.
 
 Following \cite{HelHofTer10a}, we consider partitions of a specific type. For any positive integer $\ell\,$, we say that a $\ell$-partition $\mathcal{D}=(D_i)_{1\le i\le \ell}$ of $\Tor(2,2b)$ is \emph{pair-symmetric} if, for any $i\in\{1,\dots,\ell\}\,$, $\sigma(D_i)=D_j$ with $j \neq i\,$. Let us note that a pair-symmetric partition has an even number of domains, and that the nodal domains of an antisymmetric eigenfunction form a pair-symmetric partition.
 
 \begin{lem} 
 	\label{lemCourantSym}
 	Let  $\mathcal{D}$ be a pair-symmetric and equispectral $2k$-partition of $\Tor(2,2b)\,$. We have
 	\[\lambda_{k}^{\sigma,a}\le\Lambda_{2k}(\mathcal{D})\,.\]
 \end{lem}
 
 The proof is an application of the min-max characterization of eigenvalues, with test functions taken in the space $\mathcal{A}_{\sigma}\,$, as allowed by the definition of pair-symmetric partitions. We do not give the details here, and we refer instead the reader to \cite[Proposition 6.3]{HelHofTer10a}, where a similar result is discussed in the context of a double covering of the sphere. We now apply Lemma \ref{lemCourantSym} to get a lower bound on the energy of a partition.
 
 \begin{lem}
 	\label{lemLower}
 	For $b\le 1/{\sqrt{k^2-1}}\,$, if $\mathcal{D}$ is a $k$-partition of $\Tor(1,b)$ with no domain homeomorphic to a disk, then we have $k^2\pi^2\le\Lambda_k(\mathcal{D})\,$.
\end{lem}
 
\begin{proof} 
A topological analysis of $N(\mathcal{D})\,$, using the hypothesis that no domain is homeomorphic to a disk, shows that $\Pi_4^{-1}(D_i)$ has two  distinct connected components for each $i\in\{1,\,,\dots,\,k\}\,$, which are exchanged by the map $\sigma\,$. We refer the reader to \cite{HelHof14} for the details. Consequently, $\Pi_4^{-1}(\mathcal{D})$ is a pair-symmetric  and equispectral $2k$-partition. According to Lemma \ref{lemCourantSym}, we obtain 
\[\lambda_k^{\sigma,a}\le\Lambda_{2k}(\Pi_4^{-1}(\mathcal{D}))=\Lambda_{k}(\mathcal{D})\,.\] 
A direct computation, using the facts that $k$ is odd and that $b\le 1/\sqrt{k^2-1}\,$, shows that $\lambda_k^{\sigma,a}=k^2\pi^2\,$. 
\end{proof}
 
We can now complete the proof of Theorem \ref{thmbkOdd}. Let $ b<\bS_k$ and let $\mathcal{D}$ be a minimal $k$-partition of $\Tor(1,b)\,$. According to Lemma \ref{lemNoDisk}, no domain of $\mathcal{D}$ is homeomorphic to a disk. From Lemmas \ref{lembFK} and \ref{lemLower}, we obtain
\[\Lambda_{k}(\mathcal{D}_k(1,b))=k^2\pi^2\le \Lambda_k(\mathcal{D})=\mathfrak{L}_k\left(\Tor(1,b)\right).\]
This implies that $\mathcal{D}_k(1,b)$ is minimal. 

\subsection{Conjectures on the transition values}
\label{subsecTransVal}

Instead of the four-sheeted covering $\Pi_4: \Tor(2,2b) \rightarrow \Tor(1,b)\,$, we now consider the two-sheeted covering $\Pi_2: \Tor(2,b) \rightarrow \Tor(1,b)\,$, equipped with the map $\sigma: (x,y)\mapsto (x+1\mbox{ mod }2,y)\,$. In the same manner as before, we can consider for this covering lifted partitions, antisymmetric functions, antisymmetric eigenvalues, and pair-symmetric partitions. In the rest of this section, these terms will be understood with respect to the covering $\Pi_2: \Tor(2,b) \rightarrow \Tor(1,b)\,$. Lemma \ref{lemCourantSym} also  holds in that case. We have the following conditional result.

\begin{prop}
\label{propbkLowerCond} 
If $\mathcal{D}$ is a $k$-partition of $\Tor(1,b)$ such that $\Pi_2^{-1}(\mathcal{D})$ is a $2k$-partition of $\Tor(2,b)$, then 
\[\lambda_k^{\sigma,a}\le \Lambda_{2k}(\Pi^{-1}(\mathcal{D}))=\Lambda_{k}(\mathcal{D})\,.\]
\end{prop}

\begin{proof} Since we have assumed that $\Pi_2^{-1}(\mathcal{D})$ is a $2k$-partition, the pullback  $\Pi_2^{-1}(D_i)$ of a domain of $\mathcal{D}$ has  two connected components, and the map $\sigma$ exchanges them. This implies that $\Pi_2^{-1}(\mathcal{D})$ is a pair-symmetric partition, and the result follows from Lemma \ref{lemCourantSym}.
\end{proof}

A direct computation shows that if $k$ is odd and if $b\le 2/\sqrt{k^2-1}\,$, $\lambda_k^{\sigma,a}=k^2\pi^2\,$. If we were able to prove that a minimal $k$-partition of $\Tor(1,b)$ can be lifted to a $2k$-partition of $\Tor(2,b)$ when $b<2/\sqrt{k^2-1}\,$, we would obtain $b_k\ge 2/\sqrt{k^2-1}\,$. However, this is not obvious, even assuming that the boundary set of the partition does not contain any singular point (see \cite[Section 5]{HelHof14} for a discussion of this problem).  We have on the other hand the following result.
\begin{prop}
\label{propbkUpper} If $k\ge 3$ is odd, we have $b_k\le 2/{\sqrt{k^2-1}}\,$.
\end{prop}
To prove Proposition \ref{propbkUpper}, we use the following result, whose proof is outlined in \cite[Proposition 6.3]{HelHofTer10a} in the case of a double covering of the sphere. It consists in reproducing the arguments in the proof of \cite[Theorem 1.17]{HelHofTer09}, while preserving the antisymmetry. 
\begin{lem}
	\label{lemASNodalMinPart}
If $\mathcal{D}$ is a nodal $2k$-partition associated with an antisymmetric eigenvalue, and if  
$\mathcal{D}$ has minimal energy among pair-symmetric partitions, then		 $\Lambda_{2k}(\mathcal{D})=\lambda_k^{\sigma,a}$. 
\end{lem}
\begin{proofof}{Proposition \ref{propbkUpper}} We have that
$\mathcal{D}_{2k}(2,b)$ is the partition lifted from $\mathcal{D}_k(1,b)$, and  is also the nodal partition of the antisymmetric eigenfunction $(x,y)\mapsto \sin(k\pi x)\,$. Let us assume by contradiction that $b>2/\sqrt{k^2-1}$ and $\mathcal{D}_{k}(1,b)$ is minimal. This would imply that  $\mathcal{D}_{2k}(2,b)$ is minimal among pair-symmetric $2k$-partitions, and thus, according to Lemma \ref{lemASNodalMinPart},
that $k^2\pi^2=\Lambda_{2k}(\mathcal{D}_{2k}(2,b))=\lambda_k^{\sigma,a}\,$, whereas a direct computation shows that the condition $b>2/\sqrt{k^2-1}$ implies $\lambda_k^{\sigma,a}<k^2\pi^2\,$.
\end{proofof}
  
Propositions \ref{propbkLowerCond} and \ref{propbkUpper}, and the numerical results of Section \ref{secNumerics}, suggest the following conjecture.

\begin{conj}
\label{conjkOdd}
Let us denote $\bC_k={2}/{\sqrt{k^2-1}}\,$. 
If $k\ge 3$ is odd, we conjecture that $b_k=\bC_k\,$.
\end{conj}

Let us now try to analyse minimal partitions at the transition values. According to Proposition \ref{propPartEven}, $b_4=1/2\,$. The minimal $4$-partitions of $\Tor(1,1/2)$ are therefore nodal, associated with the eigenvalue $16\pi^2\,$. Since the eigenvalue $16\pi^2$ has multiplicity $4$, we obtain in this way minimal partitions which are not merely a translation of $\mathcal{D}_4(1,1/2)\,$. Figure \ref{figPartk4Nod} shows an example whose boundary contains singular points. 
\begin{figure}[!htbp]
\begin{center}
\includegraphics[height=2.2cm]{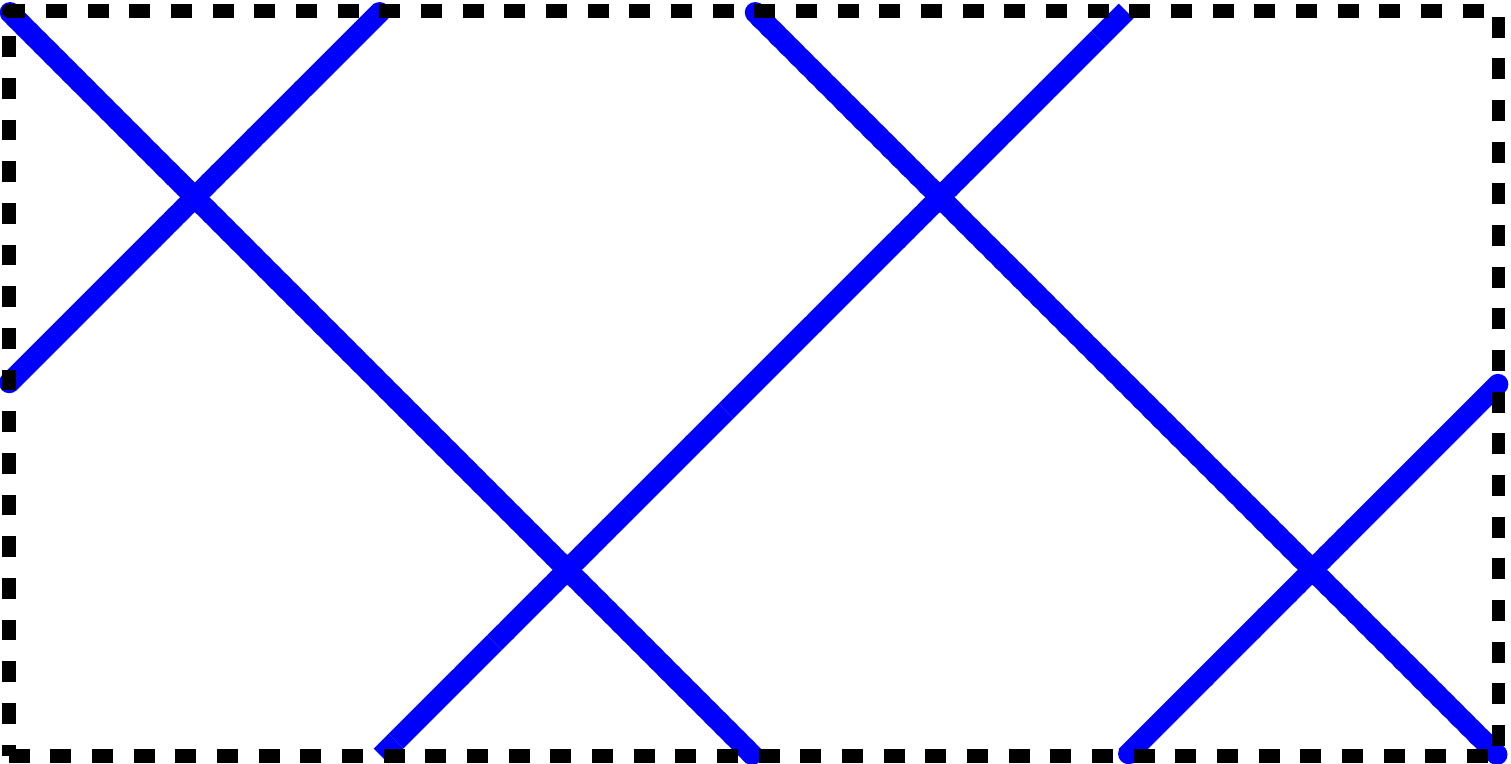}
\caption{A nodal $4$-partition of $\Tor(1,1/2)$ (associated with $\sin(4\pi x)+\sin\left(4\pi y\right)\,$). \label{figPartk4Nod}}
\end{center}
\end{figure}
We conjecture that this partition is a starting point for the apparition of non-nodal $4$-partitions of $\Tor(1,b)$ when $b=1/2+\varepsilon\,$, with $0<\varepsilon \ll 1 \,$. More precisely, we conjecture that each singular point of order four splits into two singular points of order three (see Figures \ref{fig4Part051}, \ref{fig4Part052}, and \ref{fig4Part053} in Section~\ref{subsec4part} for numerical simulations). A similar deformation mechanism was already suggested by the numerical simulations in \cite[Section 7]{BonHelHof09} and \cite[Sections 5 and 6]{BonLen14}, where the authors considered rectangles and sectors, rather than tori, and where a singular point appeared on the boundary of the domain.\\

In the case of an odd $k\,$, we are not able to give explicit examples of minimal $k$-partitions which are not translations of $\mathcal{D}_k(1,b)\,$. We can however construct candidates that would be minimal if Conjecture  \ref{conjkOdd} was true. For instance, for $k=3\,$, Conjecture \ref{conjkOdd} implies $b_{3}=1/\sqrt2$ and $\mathfrak{L}_3(\Tor(1,1/\sqrt{2}))=9\pi^2\,$, which means that any $3$-partition with energy $9\pi^2$ is minimal. We now look for antisymmetric eigenfunctions on $\Tor(2,1/\sqrt{2})\,$, associated with the eigenvalue $9\pi^2\,$, which have  $6$ nodal domains. After projecting the corresponding nodal partition on $\Tor(1,1/\sqrt{2})\,$, we obtain a $3$-partition with energy $9\pi^2\,$. Figure \ref{figNodal6Part} shows an example, in which the boundary contains singular points.
In the same way, Figure \ref{figPartk10Nod} shows how to obtain  a $5$-partition of $\Tor(1,1/\sqrt{6})$ by projection of a nodal $10$-partition of $\Tor(2,1/\sqrt{6})\,$, associated with the eigenvalue $25\pi^2\,$. The former partition is minimal provided Conjecture \ref{conjkOdd} is true. For $k=3$ and $k=5$\,, the partitions obtained numerically, for $b=2/\sqrt{k^2-1}+\eps\,$, seem close to these examples, with each singular point of order $4$ splitting into a pair of singular points of order $3$ (see Figures \ref{figk3fig2b072} and \ref{fig5PartSplit}).
\begin{figure}[!htbp]
\begin{center}
\subfigure[A nodal $6$-partition of $\Tor(2,1/\sqrt{2})$ (associated with $\cos(3\pi x)-\cos(\pi x)\,\cos(2\sqrt{2}\,\pi y)+\sin(\pi x)\,\sin(2\sqrt{2}\,\pi y)$).]{\includegraphics[height=2.6cm]{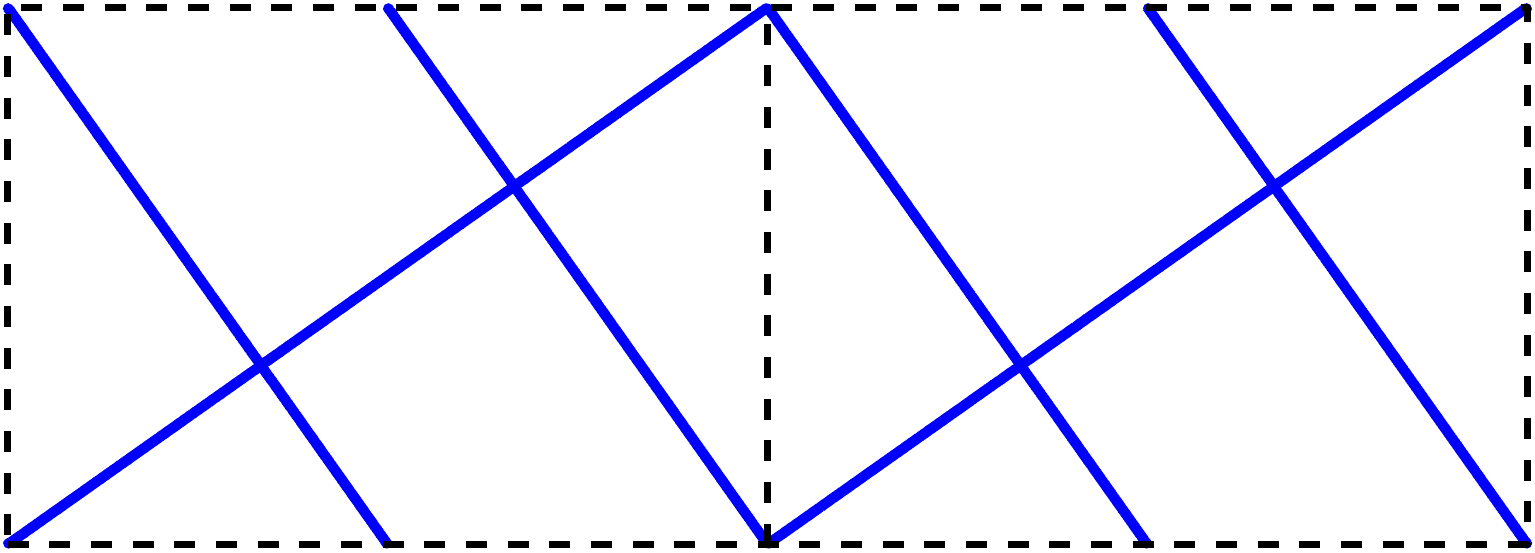}}
\hspace{2cm}
\subfigure[The $3$-partition of $\Tor(1,1/\sqrt{2})$ after projection.\label{figNodal6Partb}]{\includegraphics[height=2.6cm]{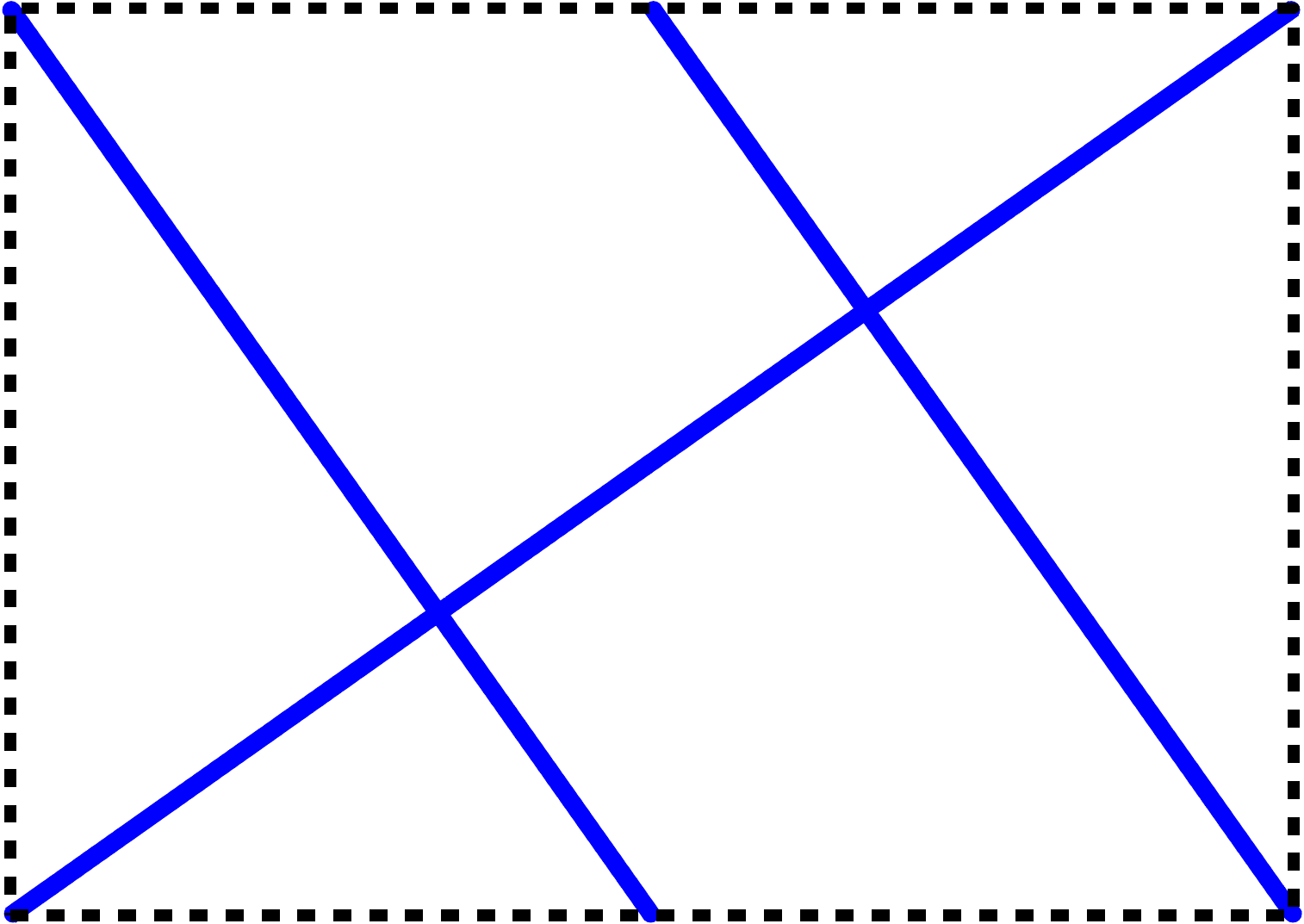}}
\caption{Construction of a $3$-partition of  $\Tor( 1,1/\sqrt{2})\,$.\label{figNodal6Part}}
\end{center}
\end{figure}
\begin{figure}[!htbp]
\begin{center}
\subfigure[A nodal $10$-partition of $\Tor(2,1/\sqrt{6})$ (associated with $\cos(5\pi x)+\sin(\pi x)\sin(2\pi\sqrt{6}y)-\cos(\pi x)\cos(2\pi\sqrt{6}y)$).]{\includegraphics[height=1.5cm]{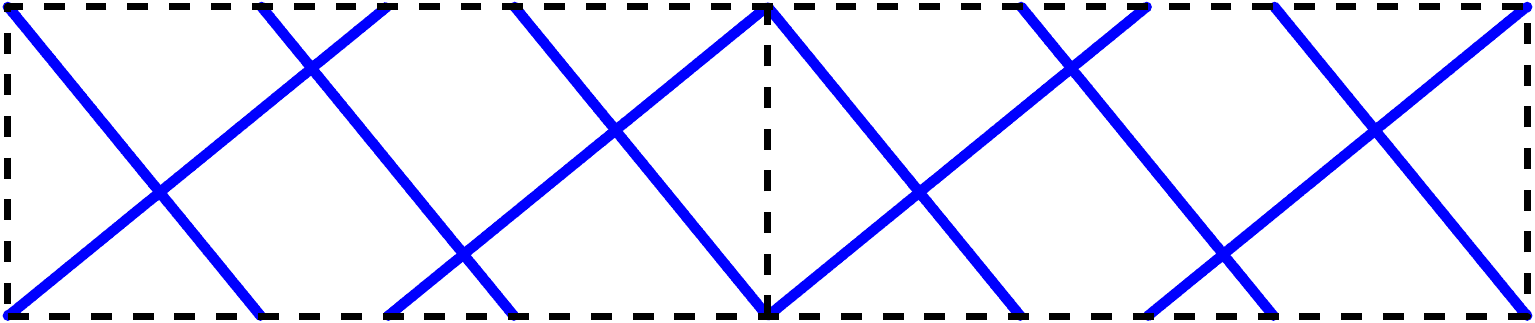}}
\hspace{2cm}
\subfigure[The $5$-partition of $\Tor(1,1/\sqrt{6})$ after projection.]{\includegraphics[height=1.5cm]{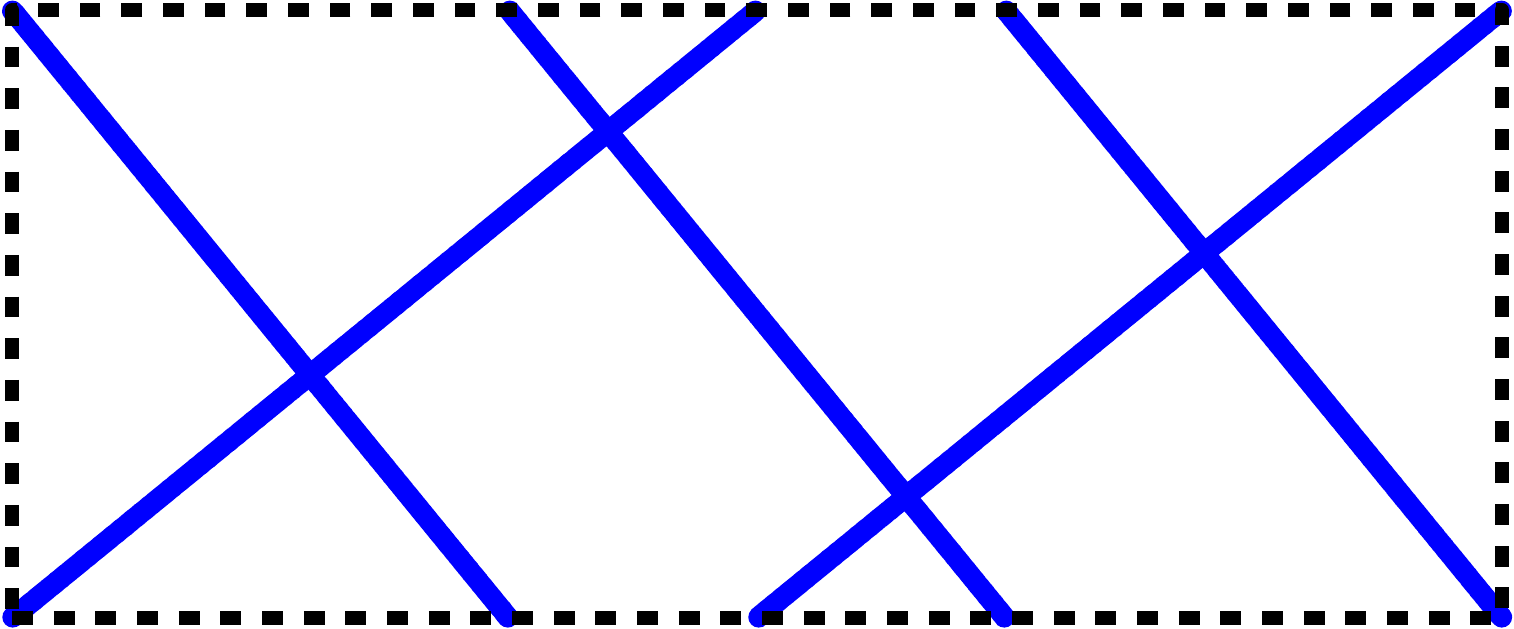}}
\caption{Construction of a $5$-partition of $\Tor(1,1/\sqrt6)\,$.\label{figPartk10Nod}}
\end{center}
\end{figure} 

\newpage
\section{Numerical study}
\label{secNumerics}

\subsection{Method}\label{sec.nummethod}
For our numerical investigations, we adapt the method introduced by B.~Bourdin, D.~Bucur, and {\'E}.~Oudet in \cite{BouBucOud09}. In order to apply it, we approach the energy \eqref{eq.nrj}, interpreted as an infinity norm of the first eigenvalues, by a $p$-norm.
\begin{defin}
\label{definPartpMin}
For any $1\le p < \infty\,$ and any $k$-partition $\mathcal{D}=(D_i)_{1\le i \le k}\,$, we define
 \begin{equation*}
 \Lambda_{k,p}(\mathcal{D})=\left(\frac{1}{k}\sum_{i=1}^{k}\lambda_1(D_i)^p\right)^{\frac{1}{p}}\,.
\end{equation*}
Then, we consider $\mathfrak{L}_{k,p}(M)=\inf\{\Lambda_{k,p}(\mathcal{D})\,;\, \mathcal{D}\in \mathfrak{P}_k\}\,$.
\end{defin}
In \cite{BouBucOud09}, the authors study the minimization of the sum of the first eigenvalues, which corresponds to the search for $\mathfrak{L}_{k,1}(M)\,$ in our notation. We extend the algorithm to cover the case $p\in [1,\infty)\,$ and look for the minimal energy $\mathfrak{L}_{k,p}(M)$ with $1<p<\infty$ large enough. This procedure is justified by the following result, proved in \cite{HelHofTer09}. 
\begin{prop}
 The minimal energy $\mathfrak{L}_{k,p}(M)$ is non-decreasing with respect to $p\,$, and 
 $$\lim_{p\to +\infty}\mathfrak{L}_{k,p}(M)=\mathfrak{L}_{k}(M)\,.$$
\end{prop}

To perform a numerical minimization, we mimic the method of \cite{BouBucOud09}: we replace the minimal $k$-partition problem by a relaxed version, where we look for $k$-tuples of functions $(f_1,\dots,f_k)$ which satisfy $\sum_{i=1}^{k}f_i=1$ and minimize a relaxed energy, depending on a small parameter $\eps>0\,$. After performing a finite difference discretization of this problem, we work with a matrix $\Phi$ of size $N\times k\,$. The integer $N$ is the number of points in the finite difference grid, and the entry $\Phi_{I,i}$ contains the (approximated) value of $f_i$ at the point indexed by $I$.  We optimize a discretized version of the energy using a gradient descent algorithm. To ensure a better convergence, we start from a random initial data on a coarse finite difference grid, and we make progressive refinements.\footnote{We thank {\'E}douard Oudet for giving us detailed  explanations on this point.} We refer the reader to \cite{BouBucOud09} for details on all these steps. In the end, we obtain a matrix $\Phi$ whose entries are either $0$ or $1\,$. We therefore have a discrete partition $(\widetilde D_{i})_{1\leq i\leq k}$ of the finite difference grid, where $\widetilde D_{i}$ contains the points $I$ such that $\Phi_{I,i}=1\,$.

To give an approximation of $\mathfrak{L}_{k,p}(\Tor(1,b))\,$, from the result of the numerical optimization, we have two further steps which are not included in \cite{BouBucOud09}:
\begin{itemize}
\item construct  a $k$-partition $\mathcal{D}=(D_{i})_{1\leq i\leq k}$ of $\Tor(1,b)$ from $(\widetilde D_{i})_{1\leq i\leq k}\,$;
\item compute the associated energy $\Lambda_{k}(\mathcal{D})\,$.
\end{itemize}

We want to construct a partition such that domains do not overlap and do not leave any part of $\Tor(1,b)$ uncovered. Let us show how this can be achieved on an example. Figure \ref{figPartNumConst} represents a discrete $3$-partition of a $7\times 6$ grid: the points of the grid are labeled by the domain to which they belong. We construct a boundary which separates points labeled by different integers,  To represent this boundary, we could use a grid whose vertices are the midpoints of the initial grid. However, since we work on the torus $\Tor(1,b)\,$, we can shift this grid by a half-step in the horizontal and vertical direction, and then deal with the initial grid. This gives us a strong partition of $\Tor(1,b)$ (see Definition \ref{definPartition}). Then, we compute an approximation of its energy, without relaxation, using a finite difference approximation of the Dirichlet-Laplacian. 
\begin{figure}[!htbp]
\begin{center}
 \includegraphics[width=4cm]{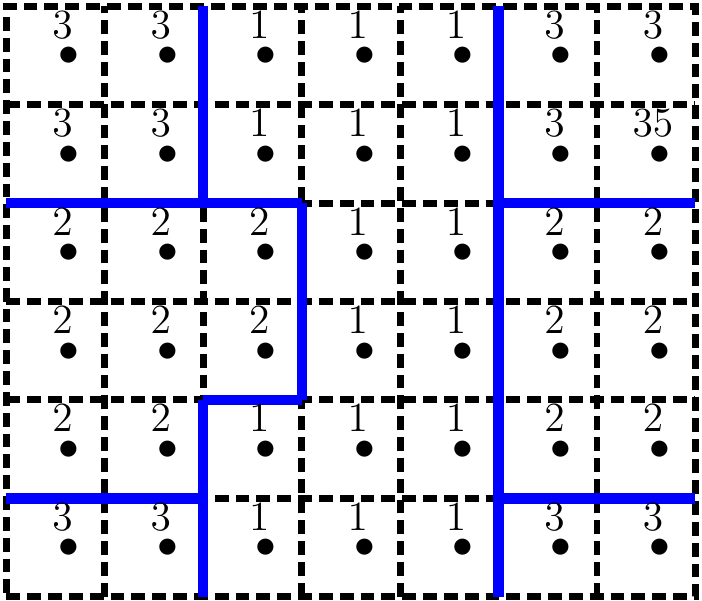}
 \caption{Partition obtained from the matrix $\Phi$.\label{figPartNumConst}}
\end{center}
\end{figure}

Let us point out that this optimization algorithm is not always successful. All the other parameters being equal, it  can converges rapidly to a good candidate for some initial data, whereas for others it terminates without reaching convergence.
To overcome this problem, we have made several simulations, starting from different initial data, and compared the resulting energies. We present the best candidates obtained in this way.

\subsection{Results}
\label{subsecNumRes}
\subsubsection{3-partitions of the torus $\Tor(1,b)$\label{subsec3part}}

We know from Proposition~\ref{thmbkOdd}, that $b_{3}>1/3\,$.
Conjecture~\ref{conjkOdd} suggests that $b_{3}$ equals $\bC_3=1/\sqrt{2}\simeq0.707\,$. If this is true, $ {\mathcal D}_3(1,b)$ should still  be minimal for $1/3<b \le 1/\sqrt 2\,$. To test this, we have implemented the method of Section~\ref{sec.nummethod} for $b\in \{j/100\,;\, j=34,\dots,100\}\,$. As was expected, the lowest energy in these cases is obtained for partitions of type $\mathcal D_{3}(1,b)$ as Figures \ref{figk3fig2b064}--\ref{figk3fig2b07} show.

\begin{figure}[h!]
\begin{center}
\subfigure[$b=0.64$\label{figk3fig2b064}]{\includegraphics[width=3.5cm]{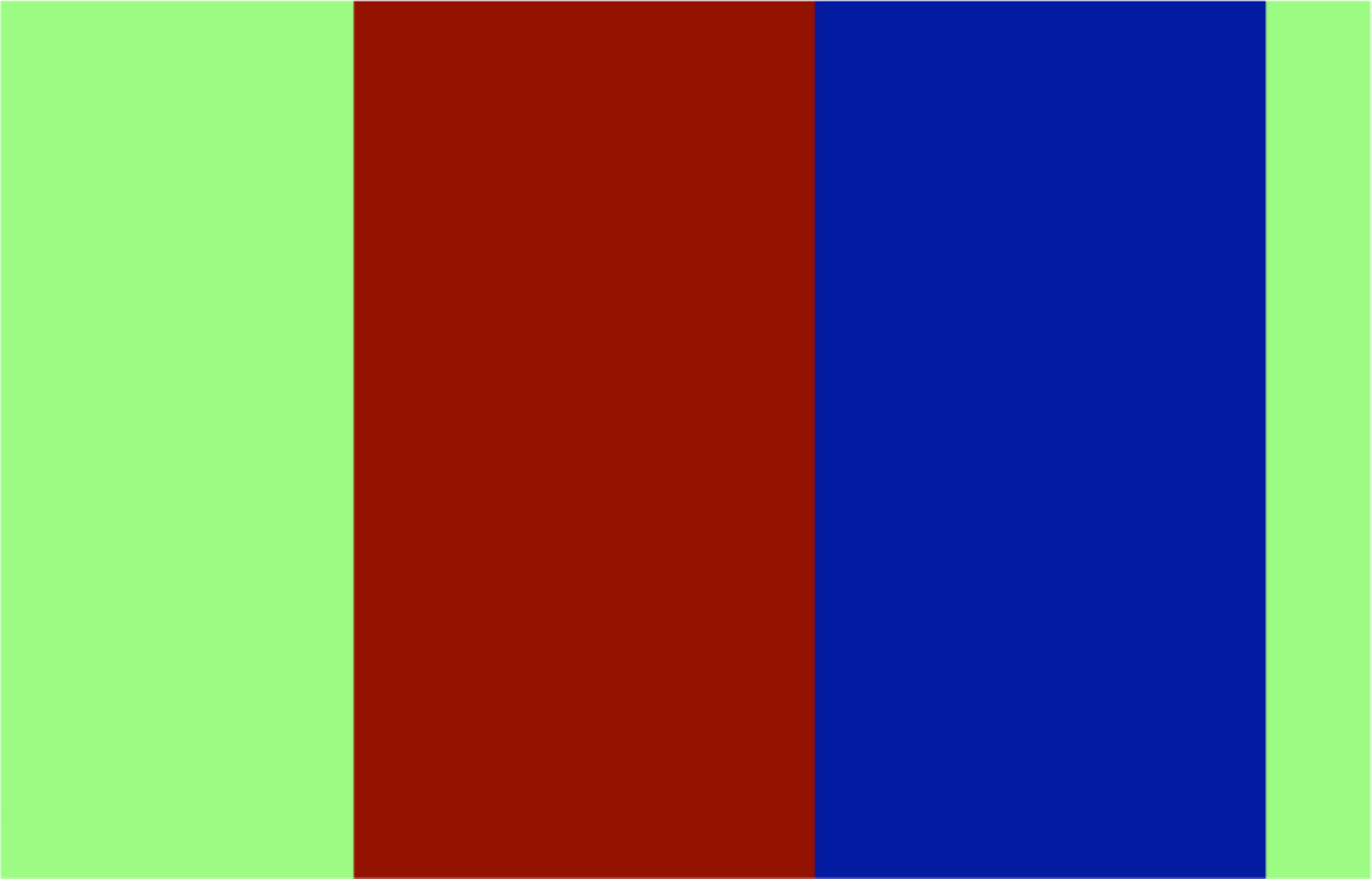}}\hfill
\subfigure[$b=0.7$\label{figk3fig2b07}]{\includegraphics[width=3.5cm]{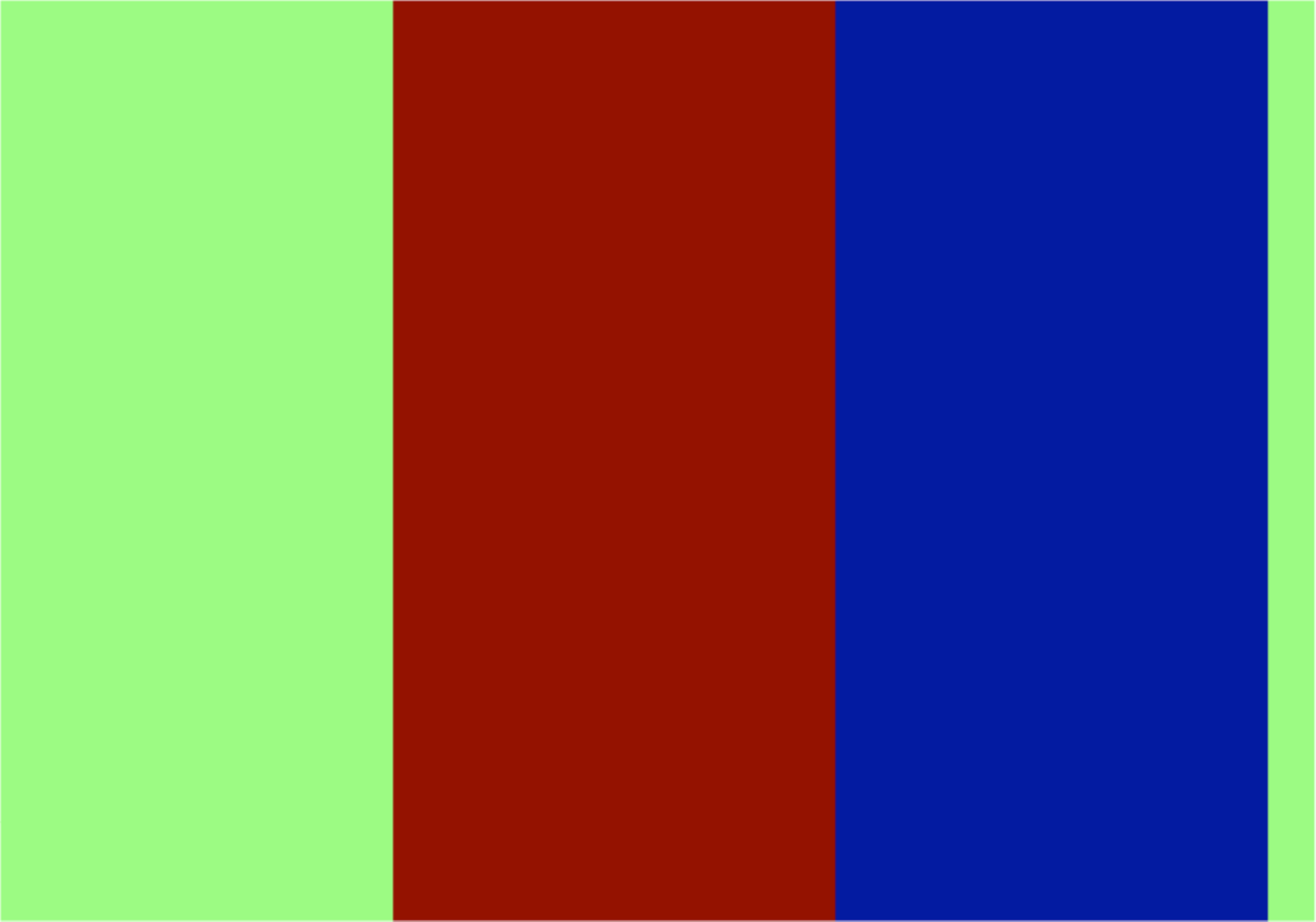}}\hfill
\subfigure[$b=0.71$\label{figk3fig2b071}]{\includegraphics[width=3.5cm]{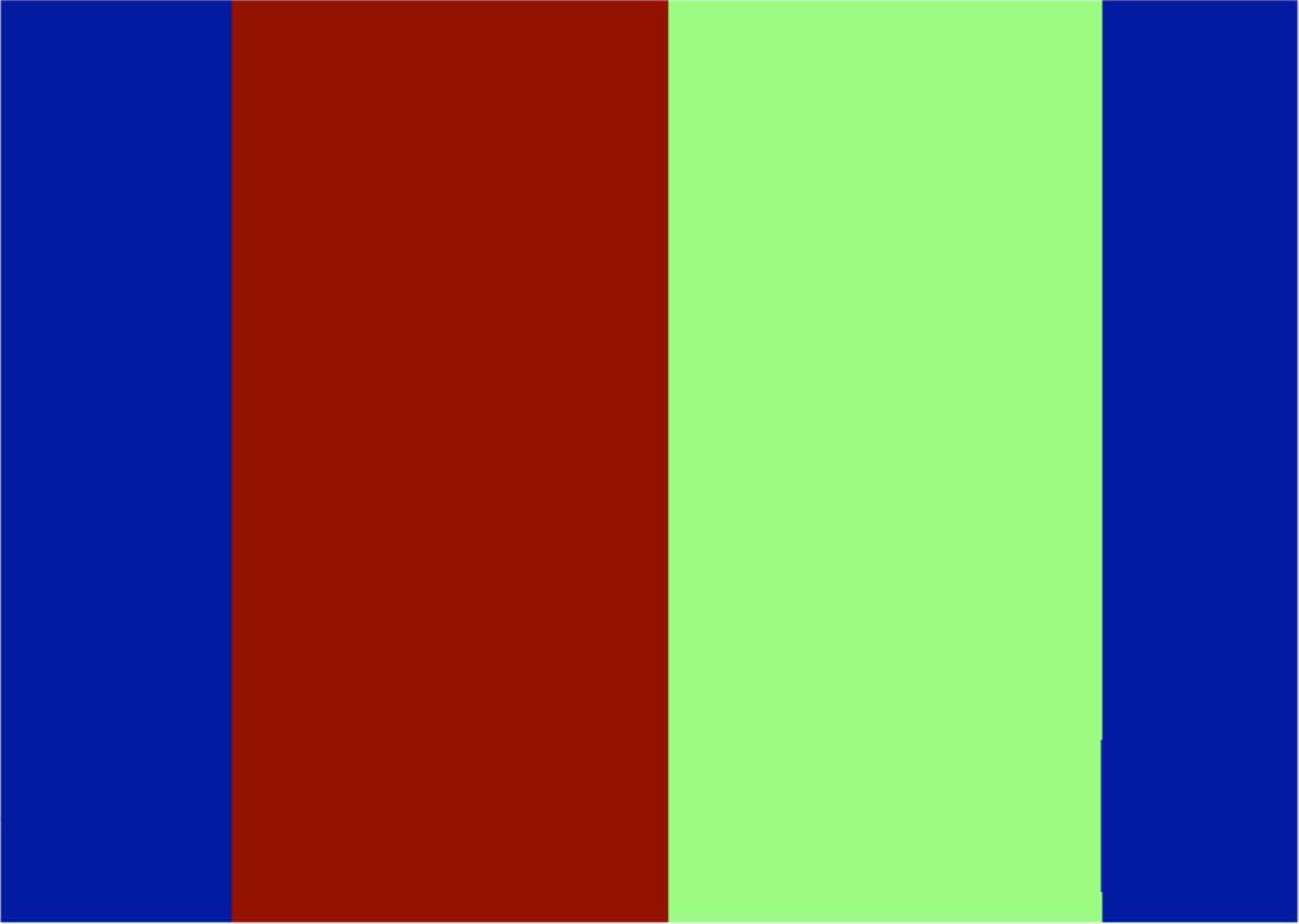}}\hfill
\subfigure[$b=0.72$\label{figk3fig2b072}]{\includegraphics[width=3.5cm]{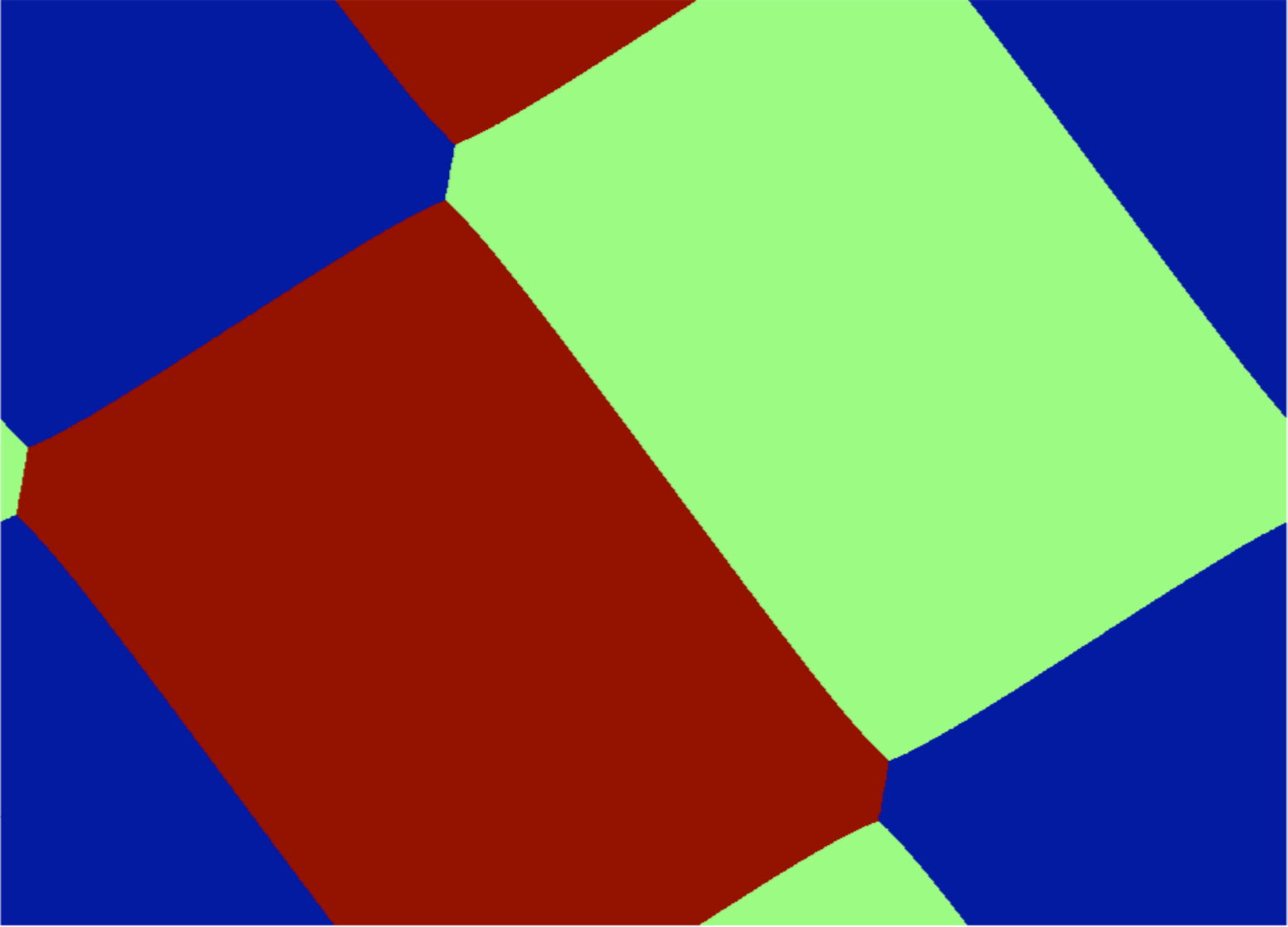}}\\
\subfigure[$b=0.73$\label{figk3fig2b073}]{\includegraphics[width=3.5cm]{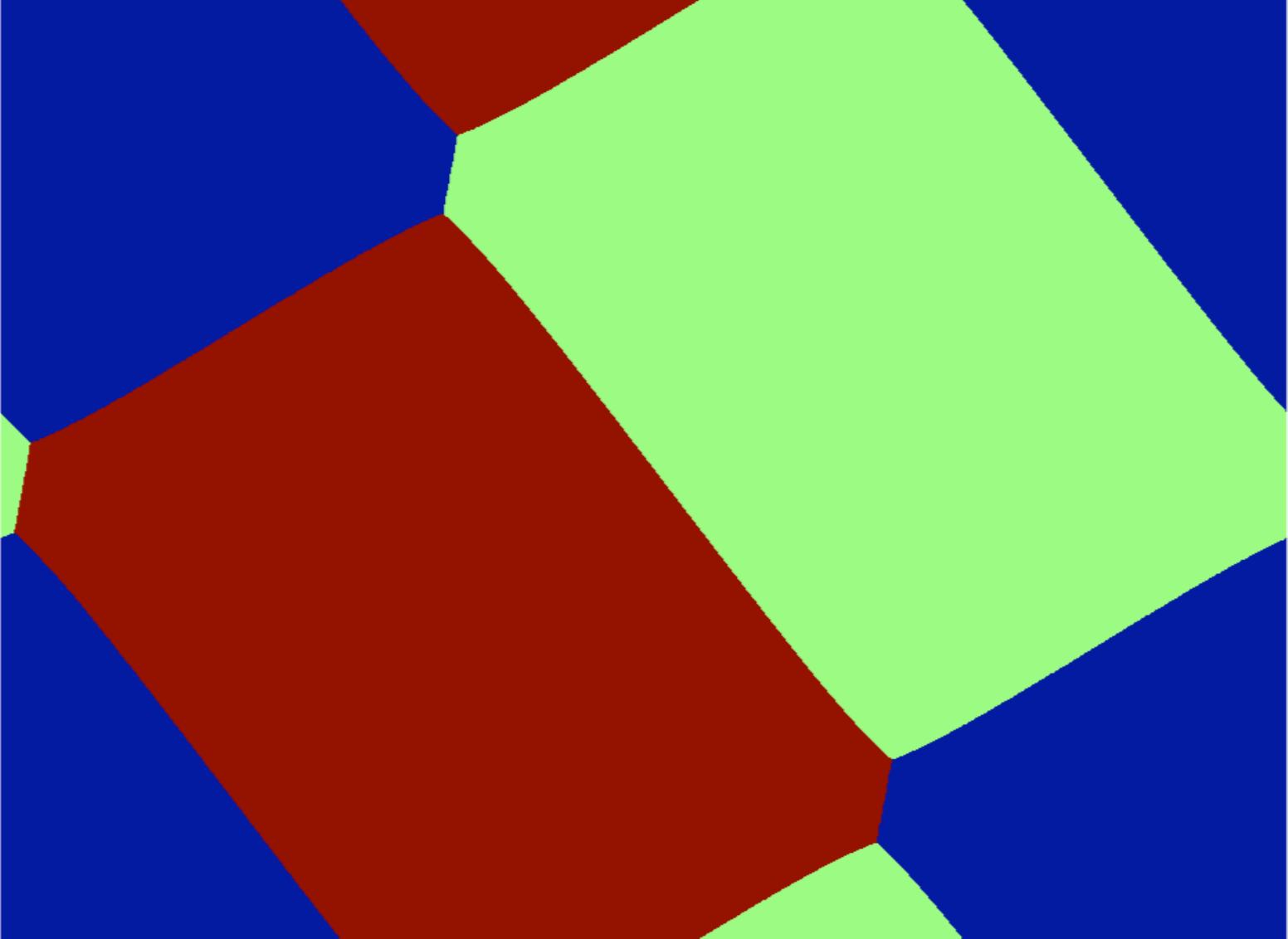}}\hfill
\subfigure[$b=0.8$\label{figk3fig2b08}]{\includegraphics[width=3.5cm]{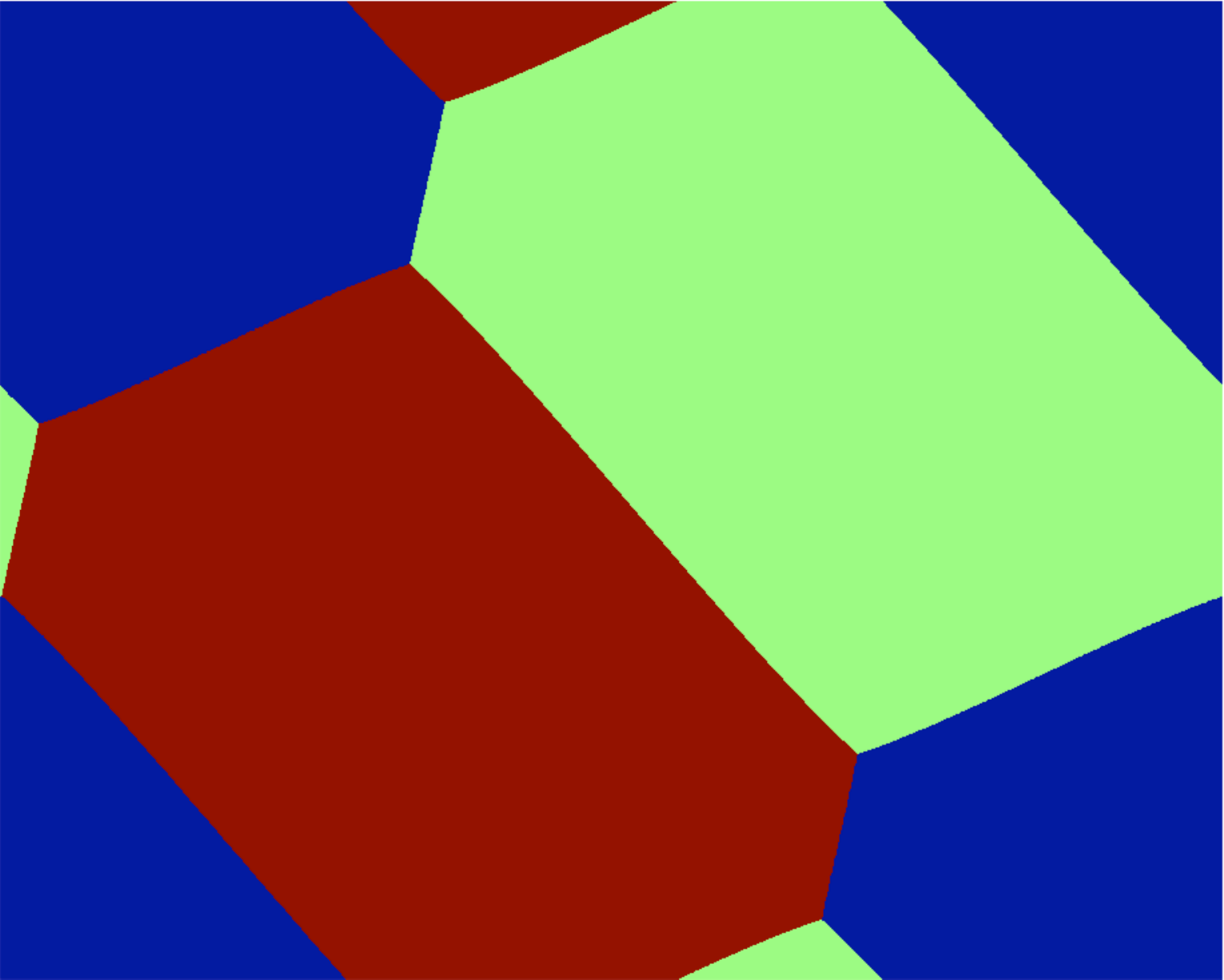}}\hfill
\subfigure[$b=0.9$\label{figk3fig2b09}]{\includegraphics[width=3.5cm]{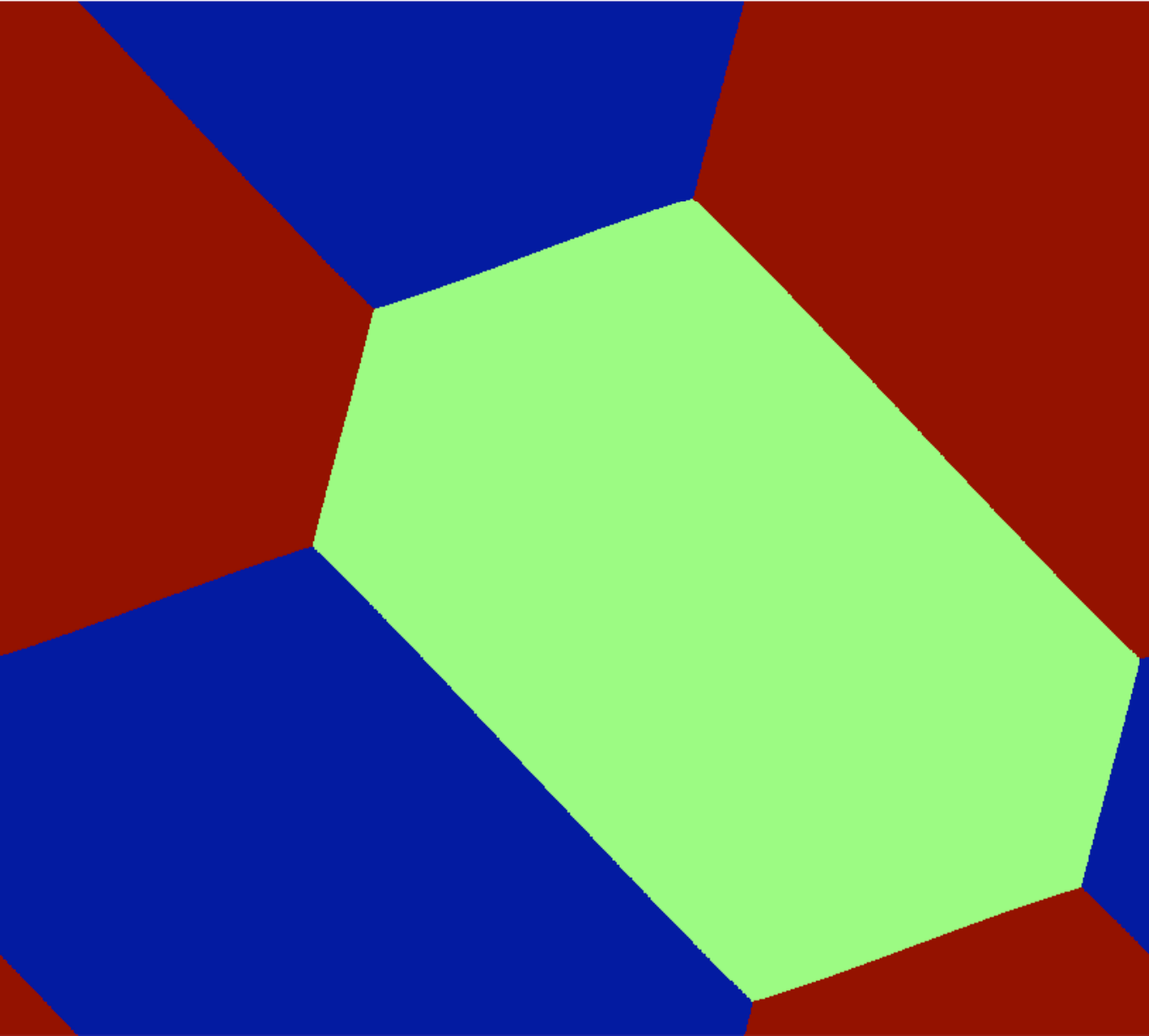}}\hfill
\subfigure[$b=1$\label{figk3fig2b1}]{\includegraphics[width=3.5cm]{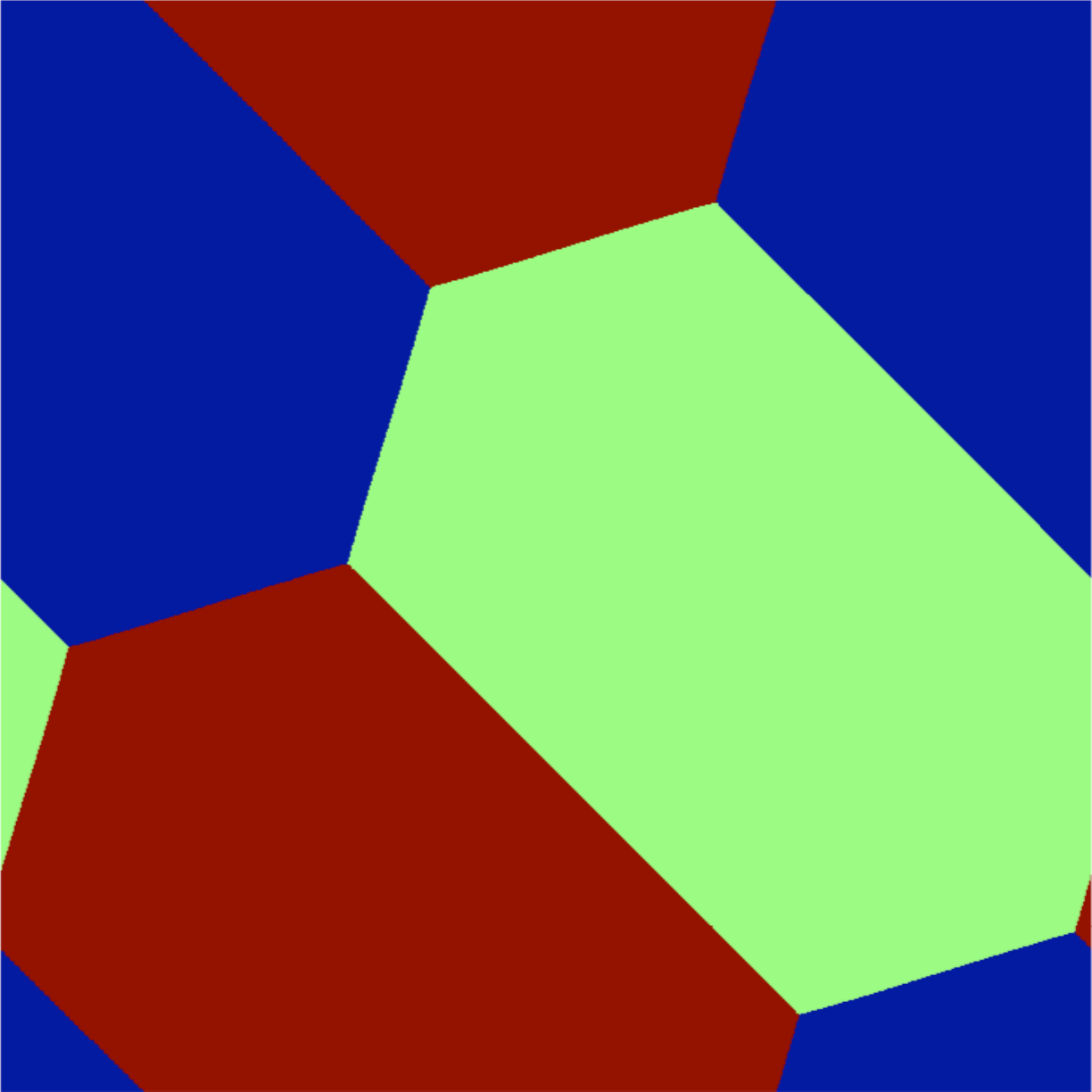}}
\caption{$3$-partitions for some values of $b\,$.\label{figk3fig2}}
\end{center}
\end{figure}

Let us now study what happens when $b$ is close to $\bC_3\,$.  Figure \ref{figk3fig2} shows our best result for different values of $b\,$. When $b$ is greater than $0.71\,$, the minimal partition seems to be a tiling of the torus by three isometric domains. These domains are roughly hexagonal, and close to the rectangles appearing in the partition of Figure \ref{figNodal6Part} when $b$ is close to $\bC_{3}\,$. For brevity, we will say in the following that a partition with hexagonal domains is \emph{hexagonal}.

For $b$ close to $\bC_3\,$, the final result of the optimization algorithm appears to be very sensitive to the initial data. As a consequence, the partitions of Figures \ref{figk3fig2b072} and \ref{figk3fig2b073} were not actually obtained by starting from random initial data. Rather, we ran the algorithm, starting from a random matrix, in the case $b=0.81\,$, where it produced an hexagonal partition similar to those of Figure \ref{figk3fig2}. We then used, as a starting point of the algorithm, the matrix obtained after two steps in the case $b=0.81\,$. Of course, we compared our final results with those of other runs starting from random initial data, and found that they always had a lower energy.  We used the same method  for $b$ close  $b_4$ when $k=4\,$, and for 
$b$ close to $\bC_5$ and $1$ when $k=5\,$.

For a larger $b\,$, up to $b=1\,$, the best candidates produced by the algorithm are still hexagonal partitions, as seen on Figures \ref{figk3fig2b08}--\ref{figk3fig2b1}. For each $b\,$, the energy of the best numerical candidate is an upper bound for $\mathfrak{L}_3(\Tor(1,b))\,$. This upper bound is plotted on Figure \ref{figfigk3.figVP} as a function of $b\,$, and compared with $9\pi^2=\Lambda_3(\mathcal{D}_3(1,b))\,$. We obtain a significant improvement for large $b\,$.  The third upper bound, represented by a solid line, will be discussed in Section \ref{secTilings}.

\begin{figure}[h!]\begin{center}
\includegraphics[height=6cm]{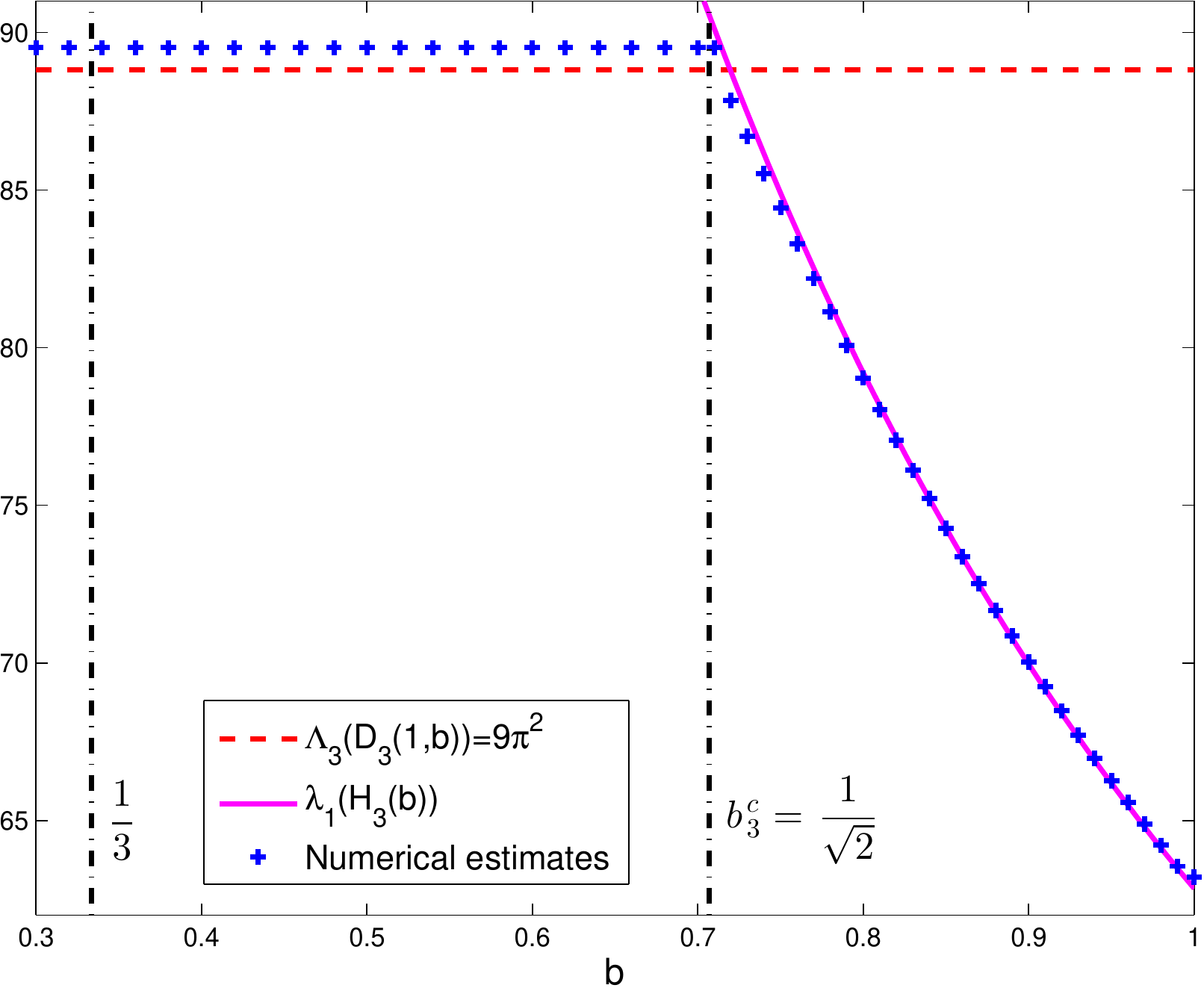} 
\caption{Upper bounds of $\mathfrak L_{3}(\Tor(1,b))$ for $b\in \{j/100\,;\, j=30\,,\,\dots\,,\,100\}\,$.\label{figfigk3.figVP}}
\end{center}
\end{figure}

\subsubsection{4-partitions of the torus $\Tor(1,b)$ \label{subsec4part}}
We know from Proposition~\ref{propPartEven} that $b_{4}=1/2\,$.  We are interested in the nature of minimal partitions for $b$ close to $b_4\,$. Figure \ref{fig4PartTrans} shows the best candidates obtained from the algorithm. 

\begin{figure}[h!] \begin{center}
\begin{minipage}{.73\linewidth}
\subfigure[$b=0.48$]{\includegraphics[width=3.5cm]{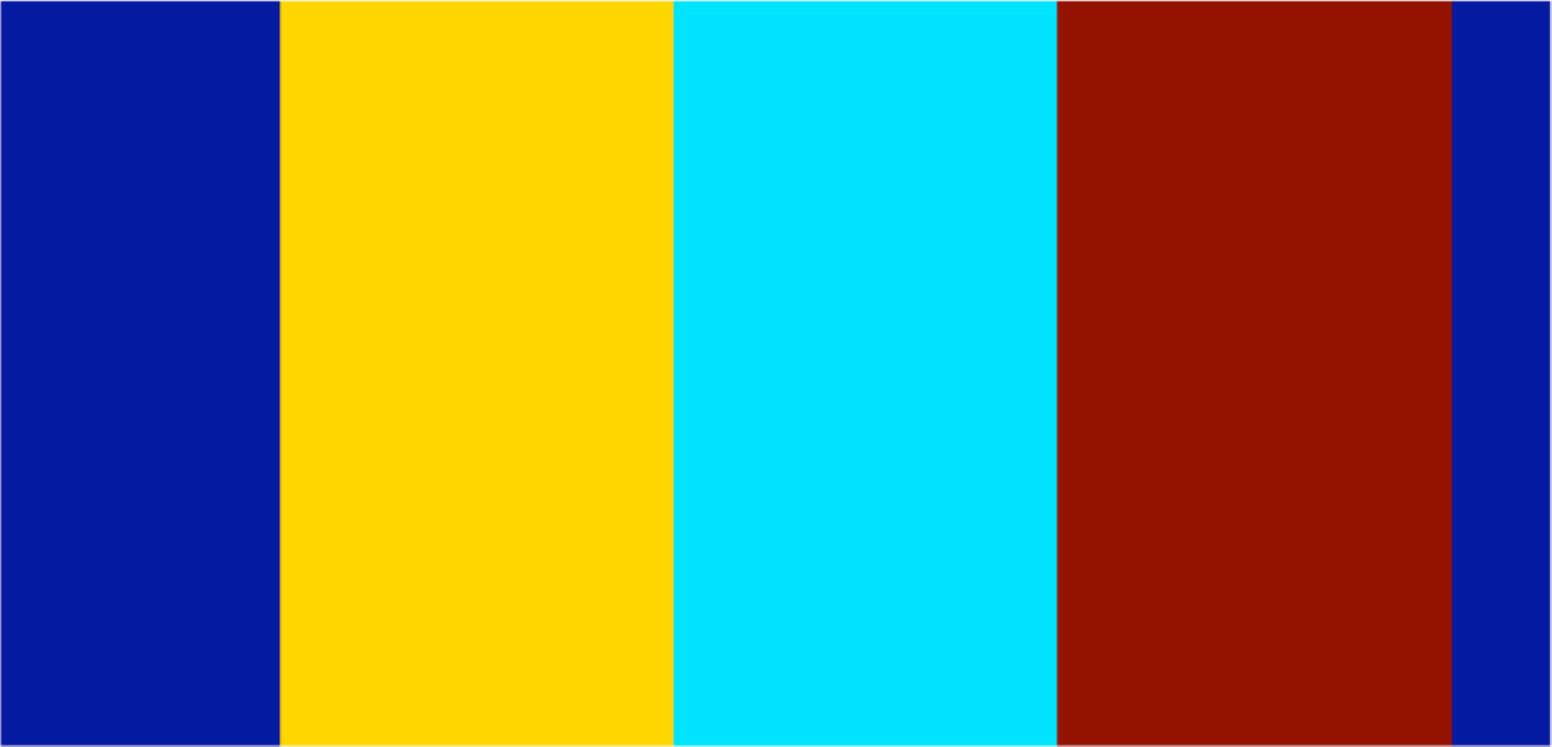}}\hfill
\subfigure[$b=0.49$]{\includegraphics[width=3.5cm]{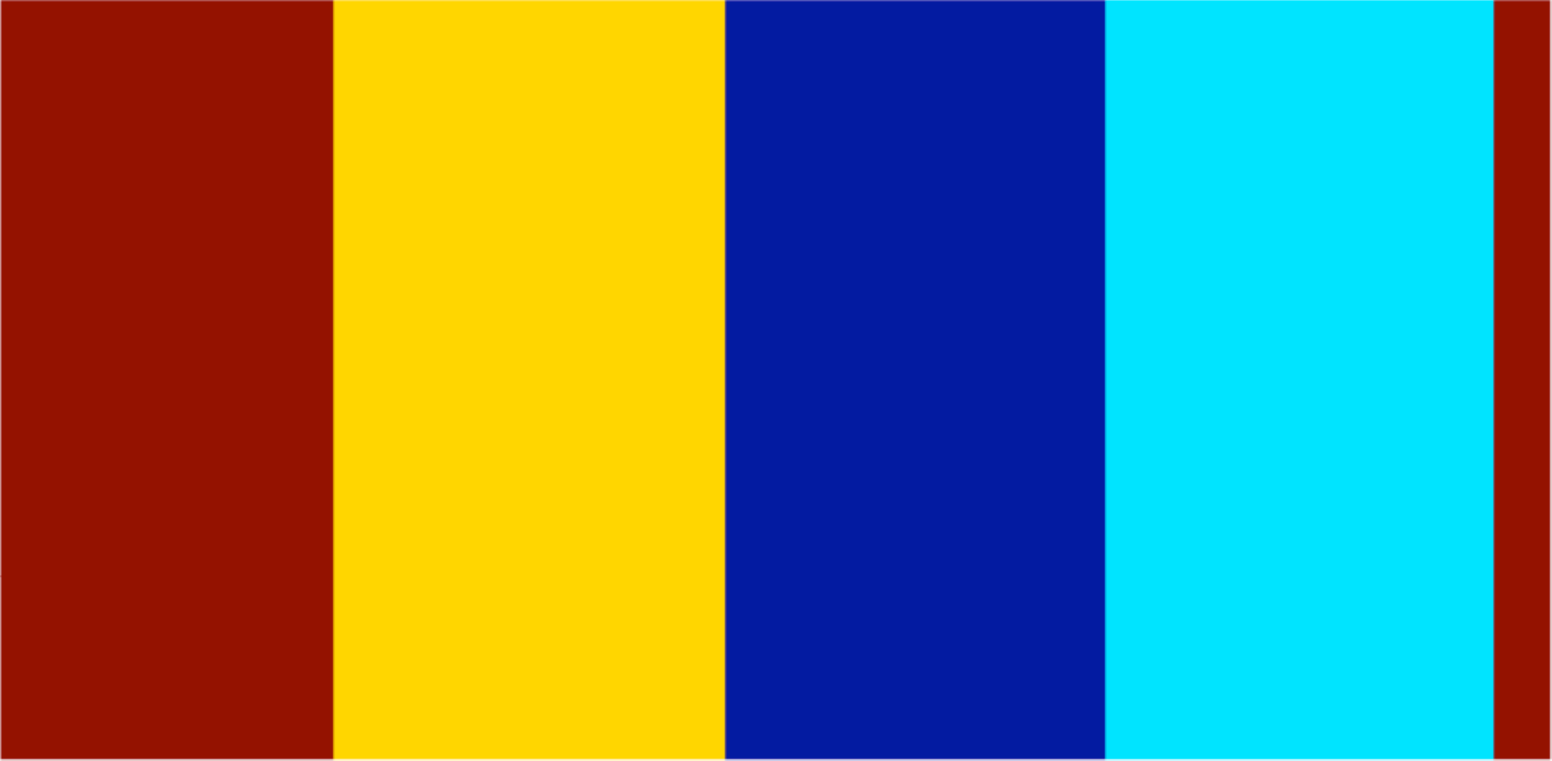}}\hfill
\subfigure[$b=0.50$]{\includegraphics[width=3.5cm]{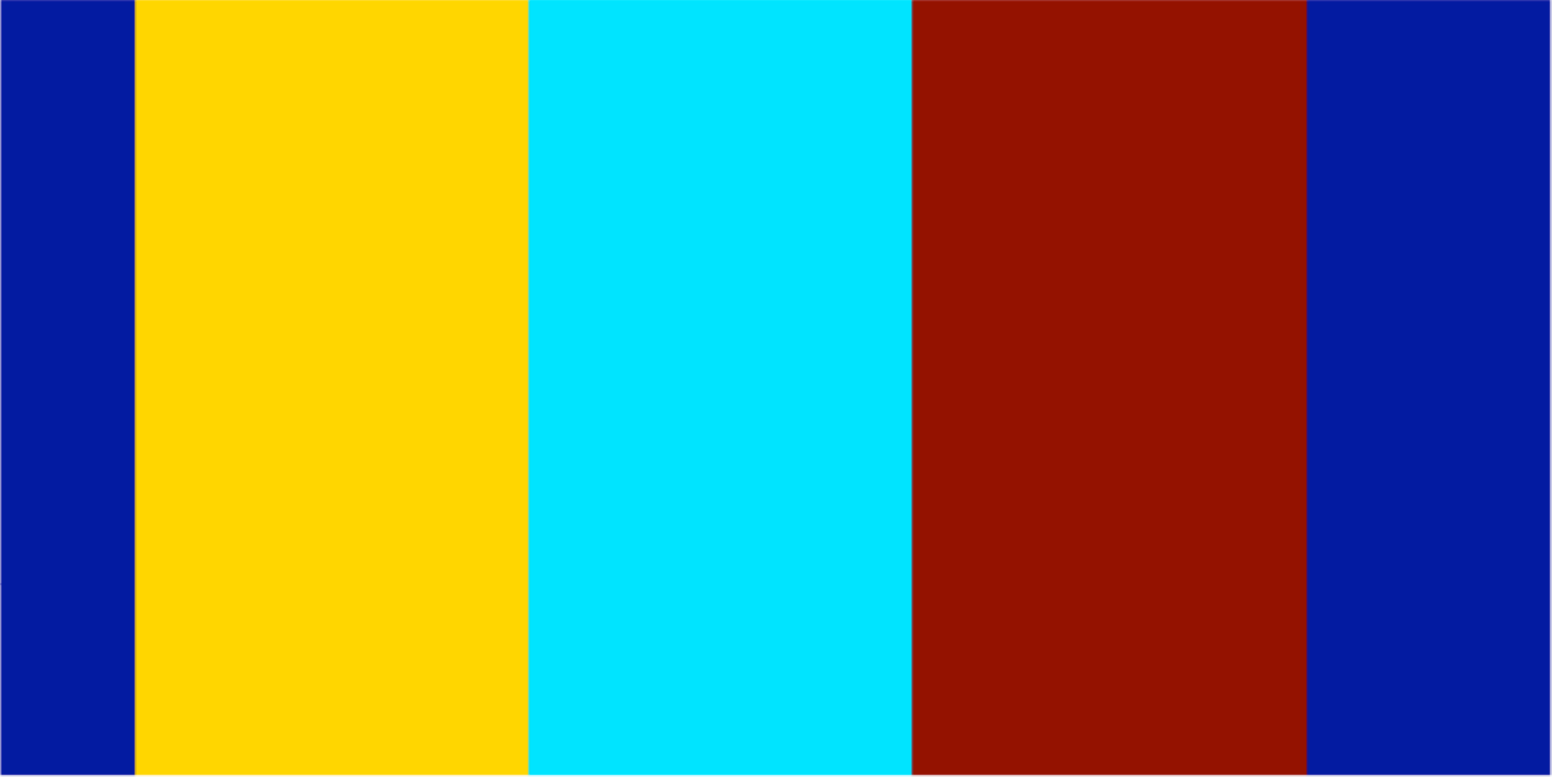}}\\
\subfigure[$b=0.51$\label{fig4Part051}]{\includegraphics[width=3.5cm]{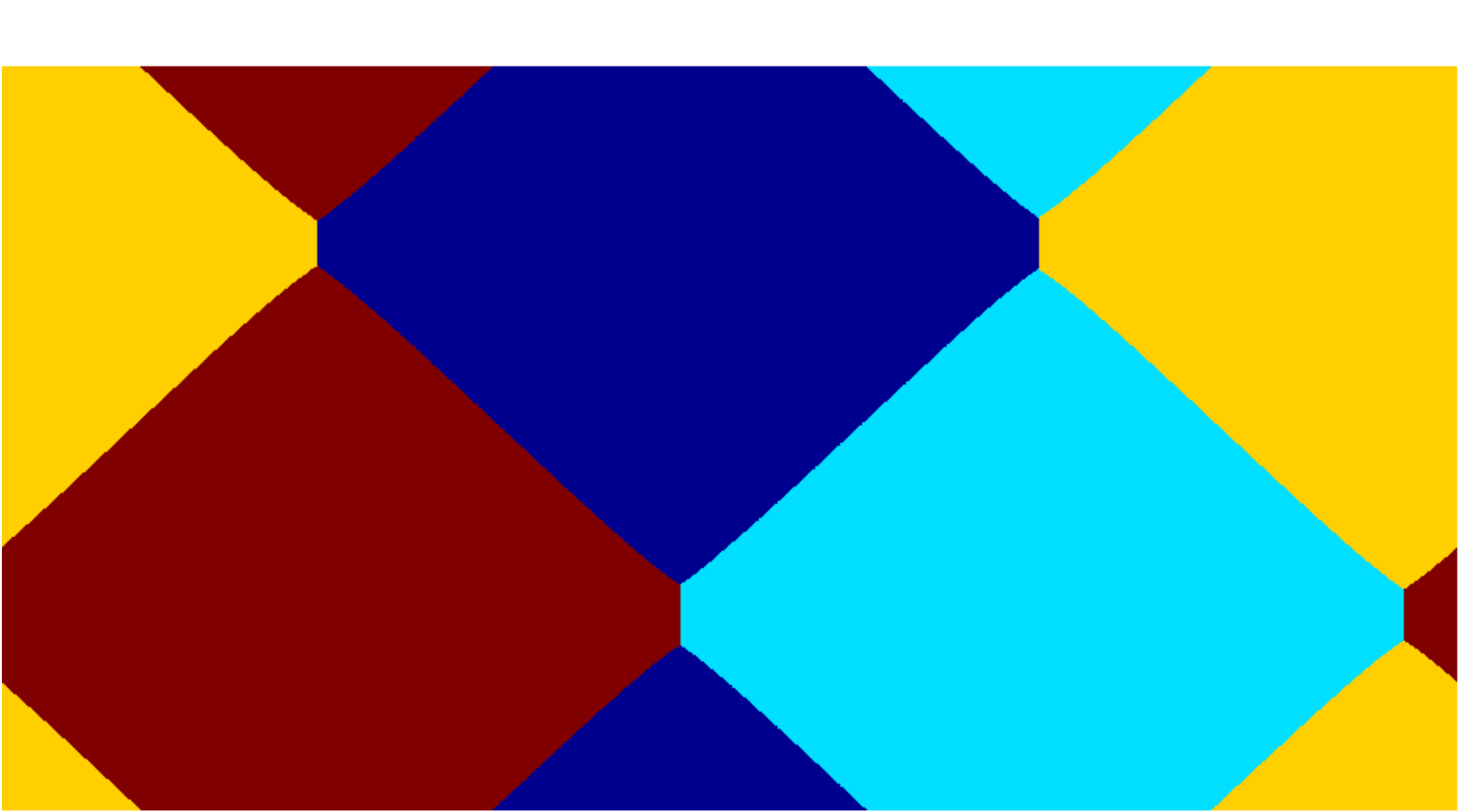}}\hfill
\subfigure[$b=0.52$\label{fig4Part052}]{\includegraphics[width=3.5cm]{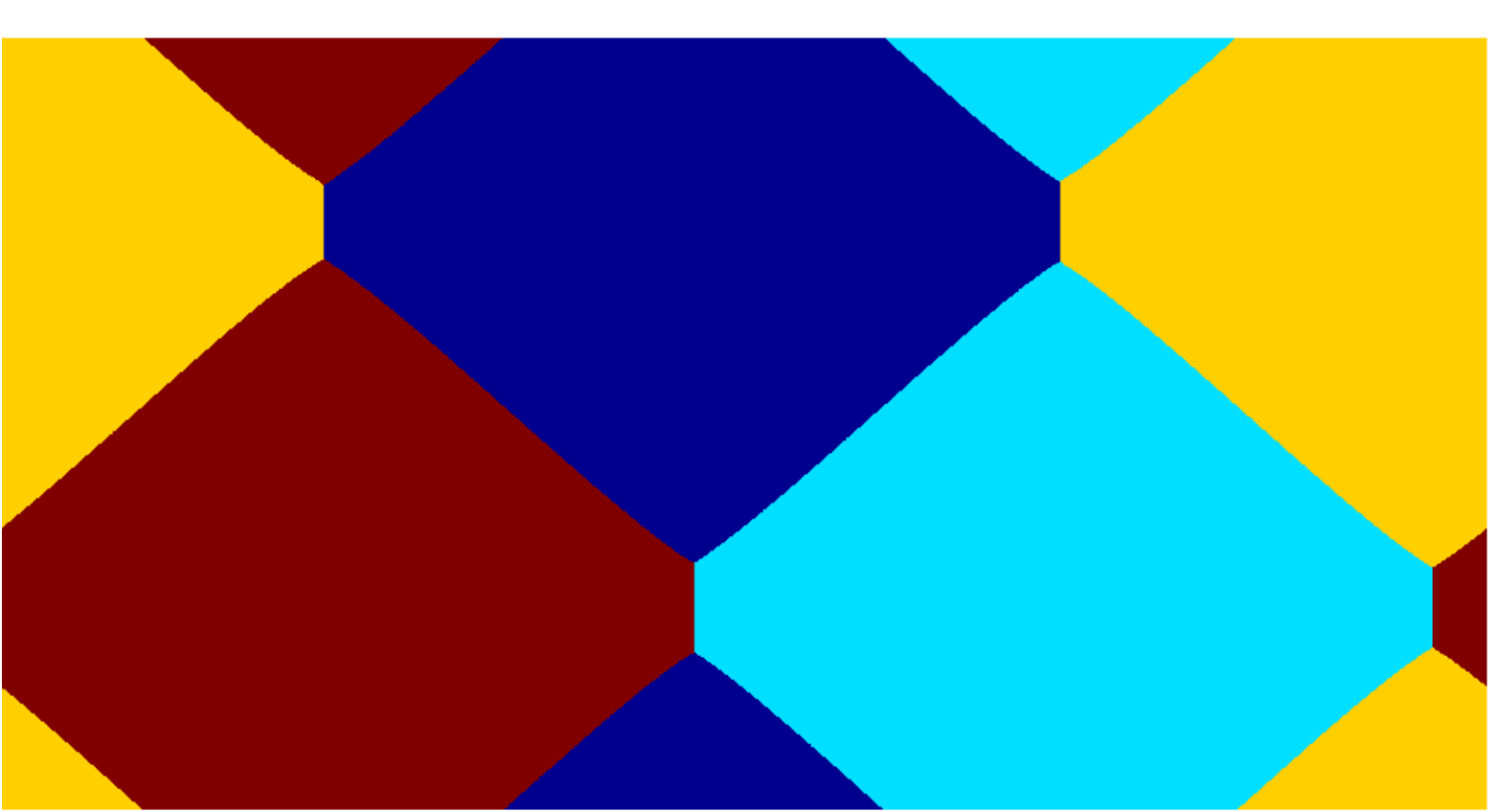}}\hfill
\subfigure[$b=0.53$\label{fig4Part053}]{\includegraphics[width=3.5cm]{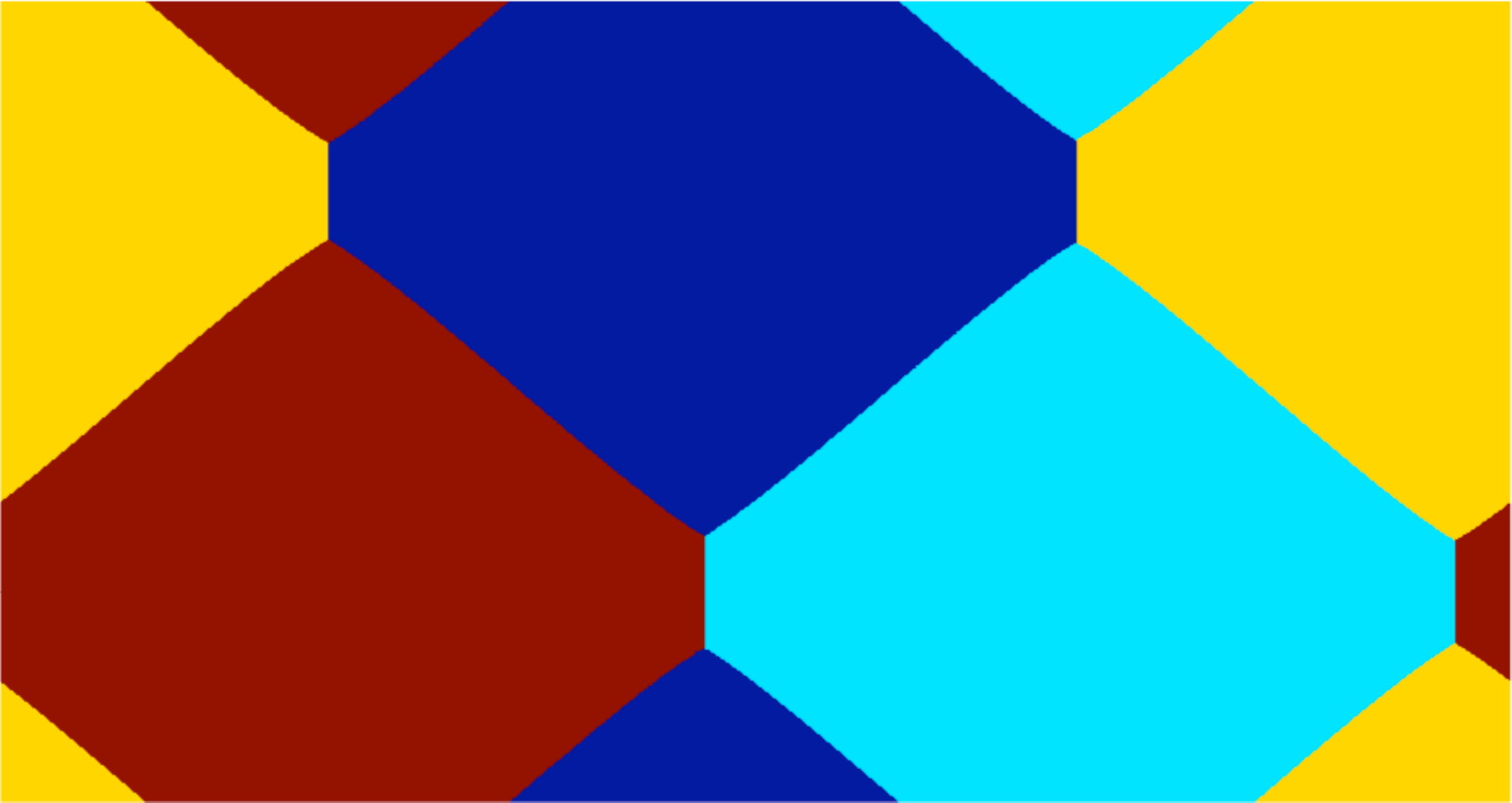}}
\end{minipage}
\hfill\hfill
\begin{minipage}{.25\linewidth}
\hfill\subfigure[$b=1$\label{fig4PartT11}]{\includegraphics[width=3.5cm]{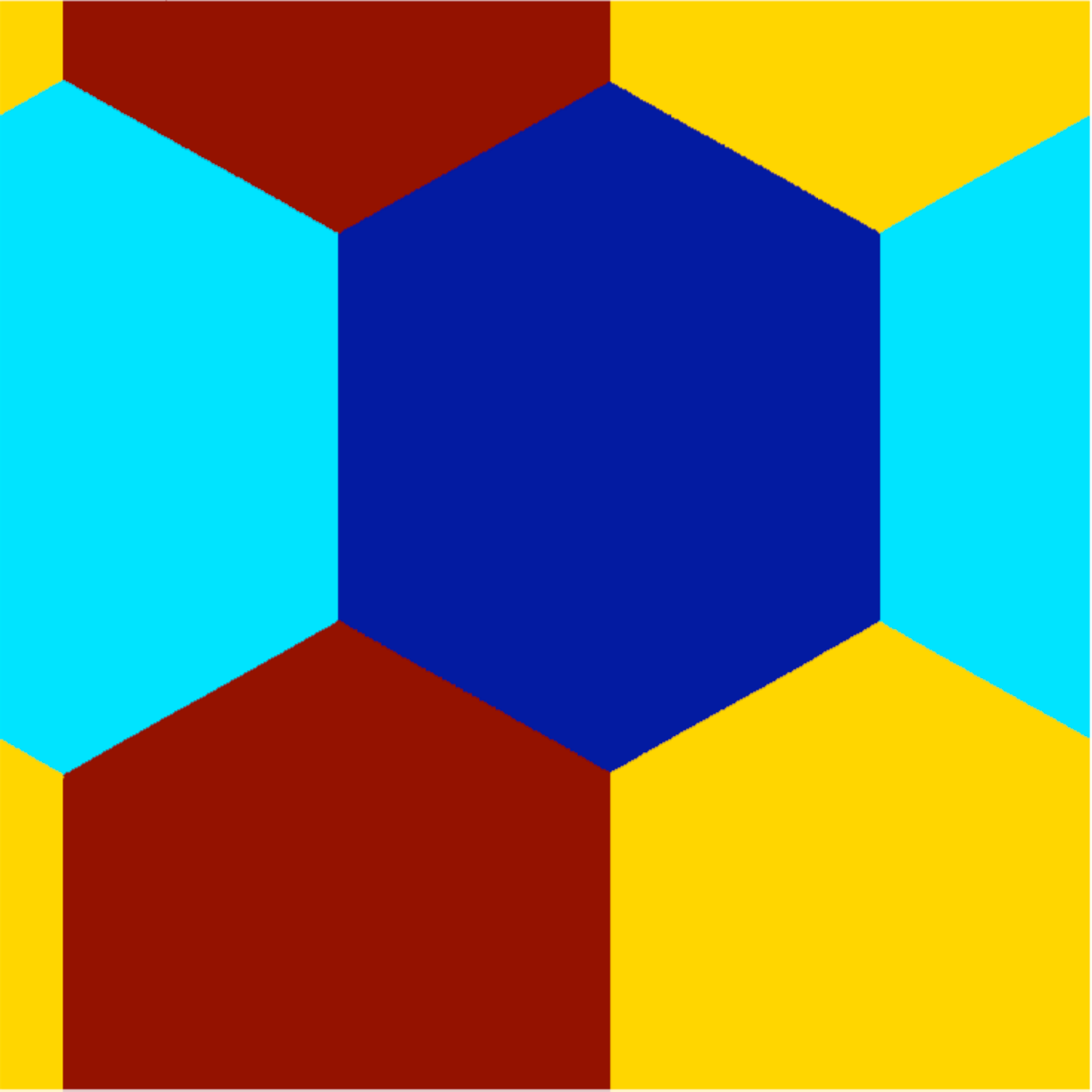}}
\end{minipage}
\caption{$4$-partition for $b\in\{j/100\,;\, j=48\,,\,\dots\,,\,53\}$ and $b=1\,$. \label{fig4PartTrans}}
\end{center} \end{figure}

As in the case $k=3\,$, we note the apparition of hexagonal partitions, shown on Figures \ref{fig4Part051}--\ref{fig4Part053}. The hexagonal domains of the partition shown on Figure \ref{fig4Part051} seem close to the square domains of the nodal $4$-partition shown on Figure \ref{figPartk4Nod}. This suggests that the partition of $\Tor(1,1/2)$ into four squares, shown on Figure \ref{figPartk4Nod}, is the starting point for the apparition of non-nodal $4$-partitions of $\Tor(1,b)\,$, when $b$ becomes greater than $1/2\,$.

For larger values of $b\,$, up to $b=1\,$, the minimal partitions are apparently still hexagonal. Figure \ref{fig4PartT11} shows for instance the best candidate for a minimal $4$-partition of $\Tor(1,1)\,$. The energy of the best candidates gives us an upper bound for $\mathfrak{L}_4(\Tor(1,b))\,$, represented on Figure \ref{figfigk4.figVP} as a function of $b\,$.

\begin{figure}[h!]\begin{center}
\includegraphics[height=6cm]{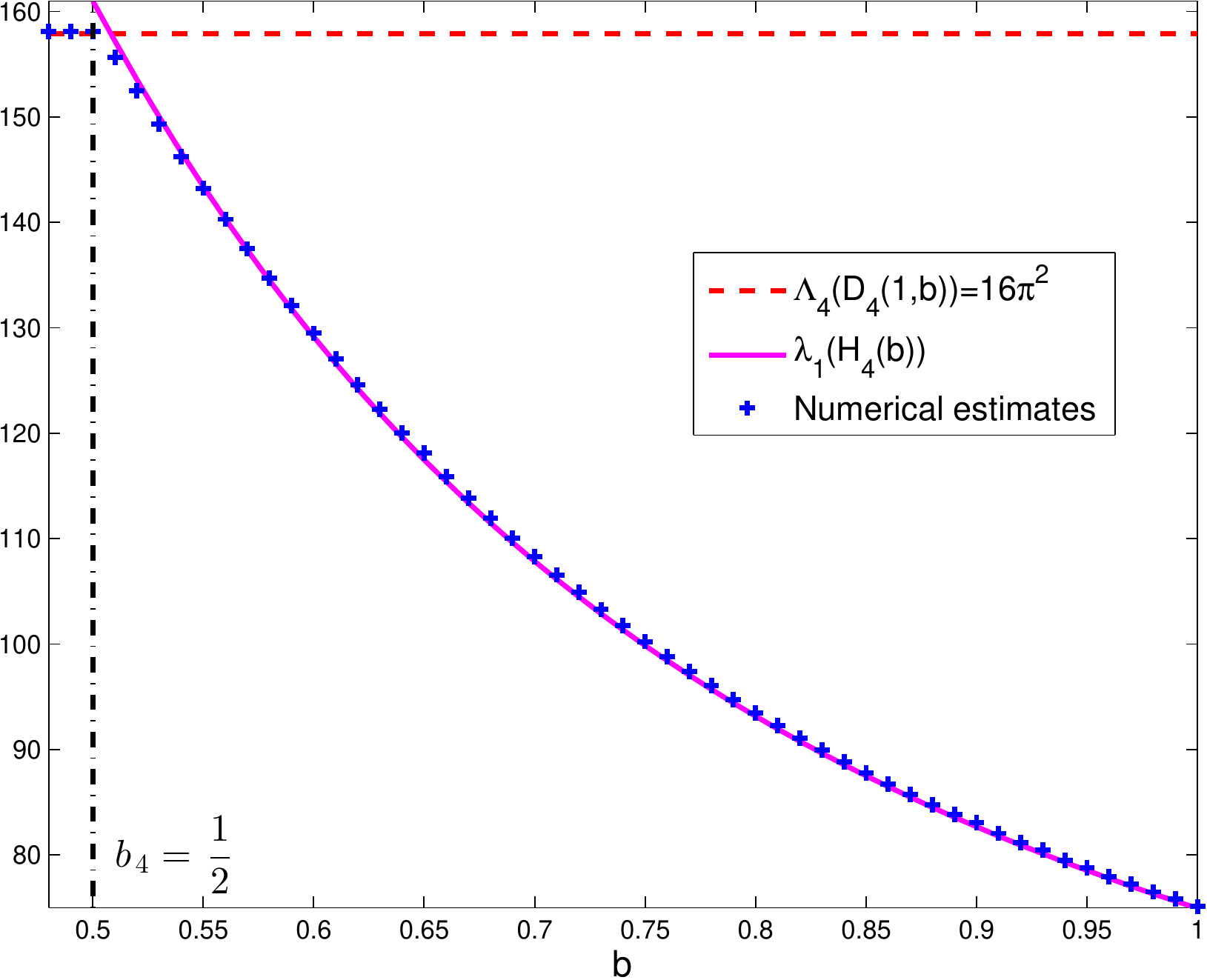}  
\caption{Upper bounds of $\mathfrak L_{4}(\Tor(1,b))$ for $b\in \{j/100\,;\, j=48\,,\,\dots\,,\,100\}\,$. \label{figfigk4.figVP}}
\end{center}\end{figure}

\subsubsection{5-partitions of the torus $\Tor(1,b)$\label{subsec5part}}
\begin{figure}[h!] \begin{center}
\subfigure[$b=0.40$]{\includegraphics[width=3.4cm]{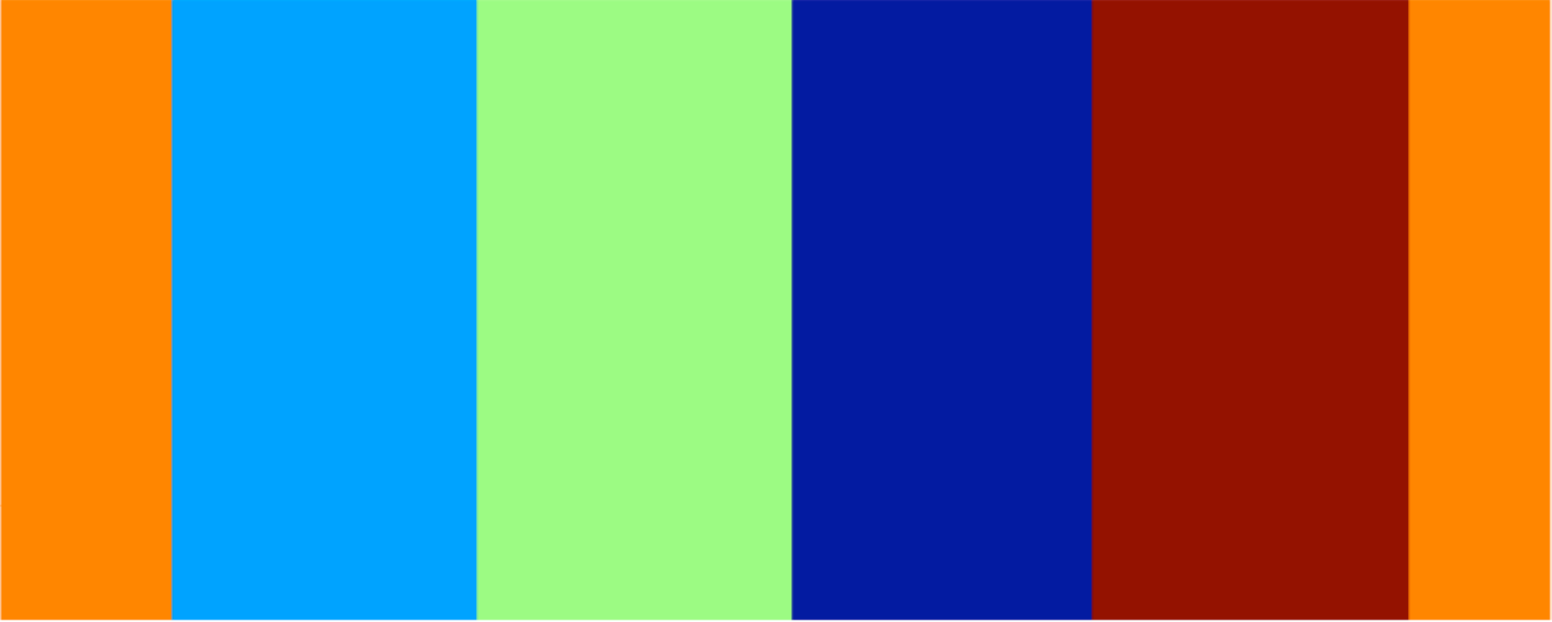} }\hfill
\subfigure[$b=0.41$]{\includegraphics[width=3.4cm]{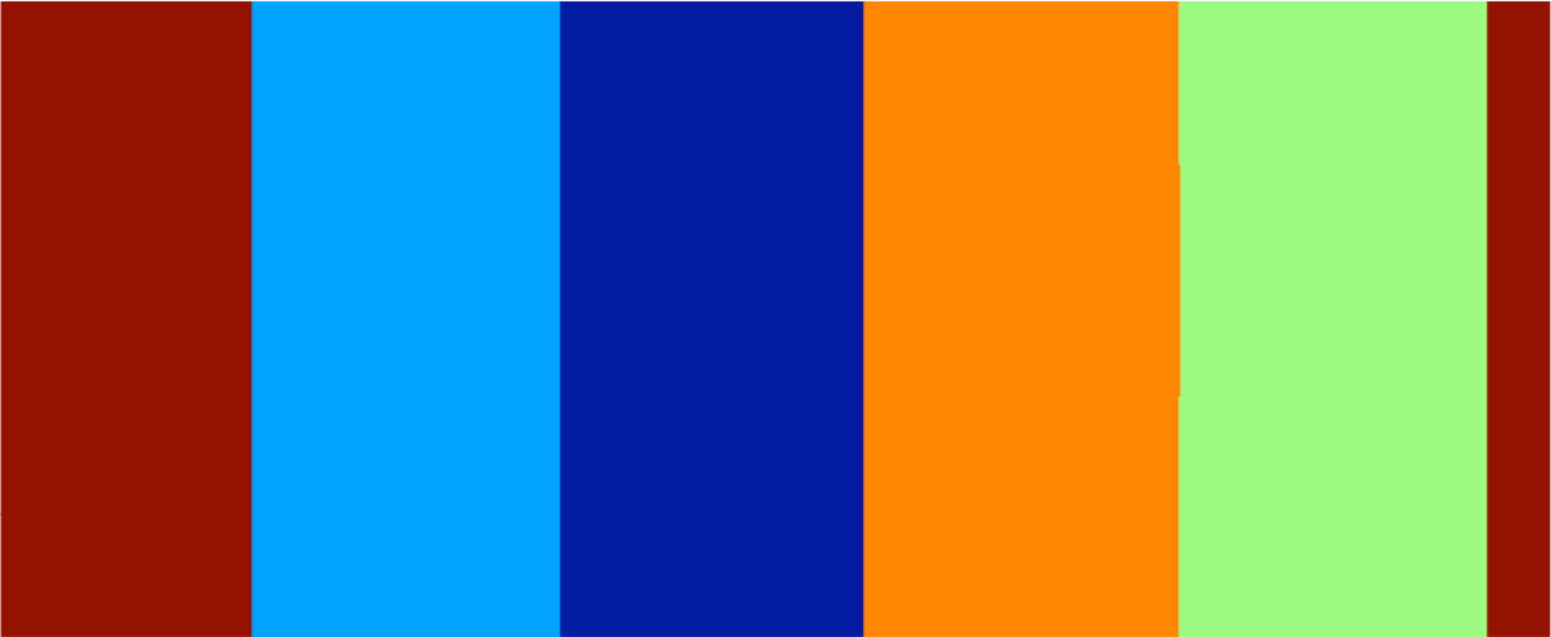} }\hfill
\subfigure[$b=0.42$\label{fig5PartSplit}]{\includegraphics[width=3.4cm]{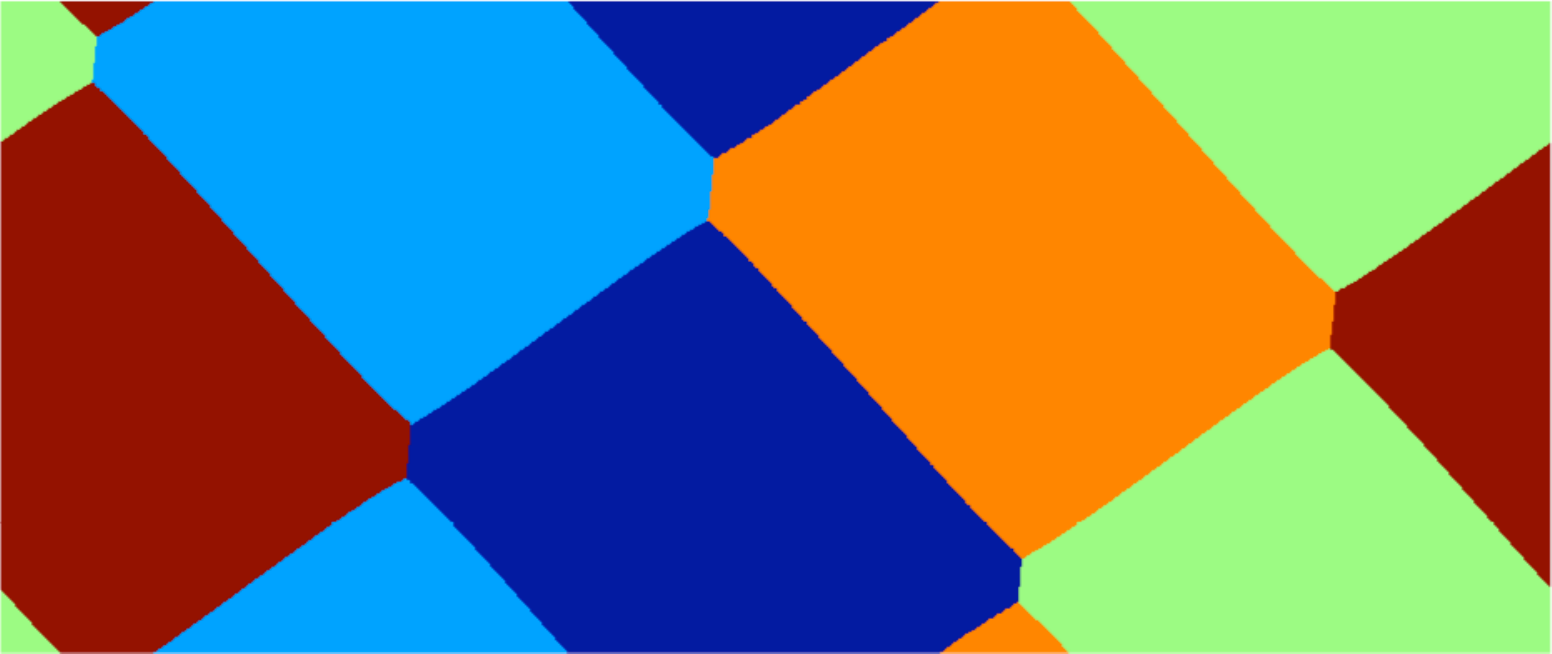} }\hfill
\subfigure[$b=0.43$]{\includegraphics[width=3.4cm]{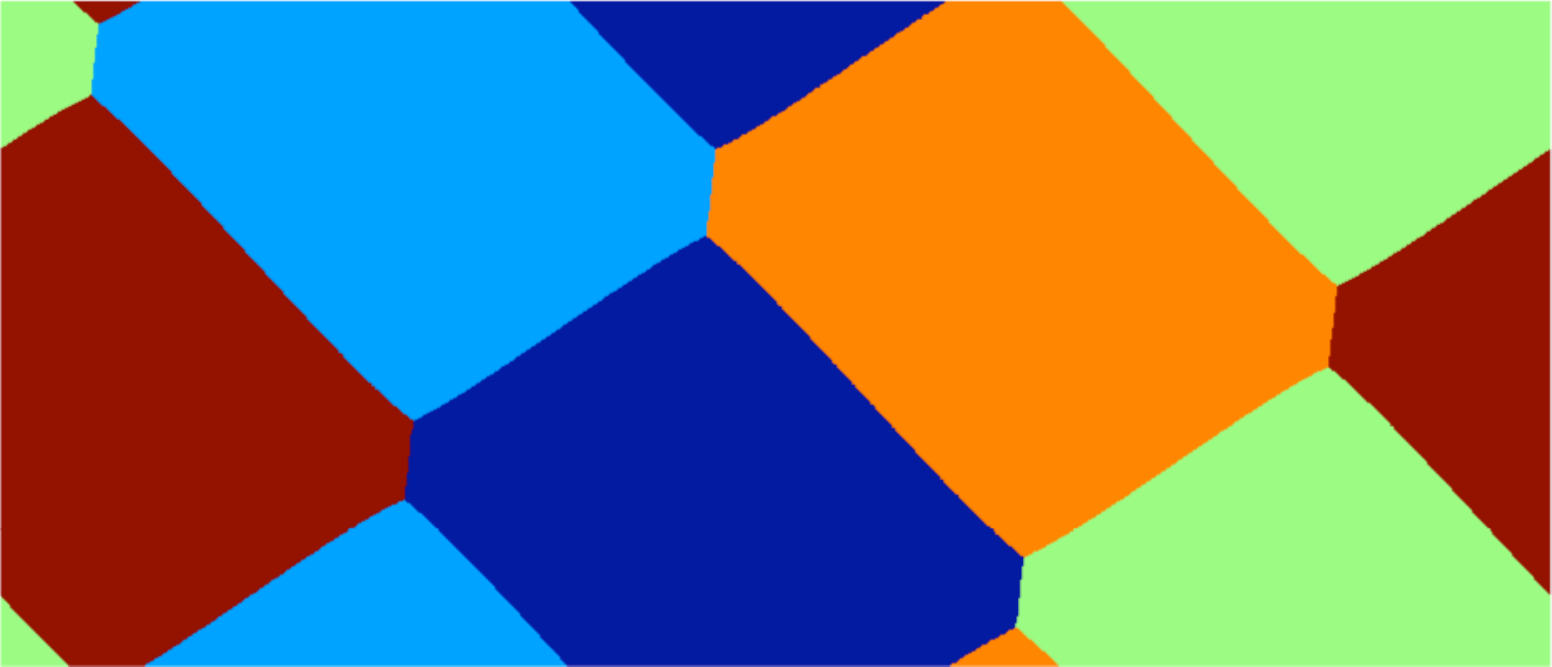} }\\
\subfigure[$b=0.44$]{\includegraphics[width=3.4cm]{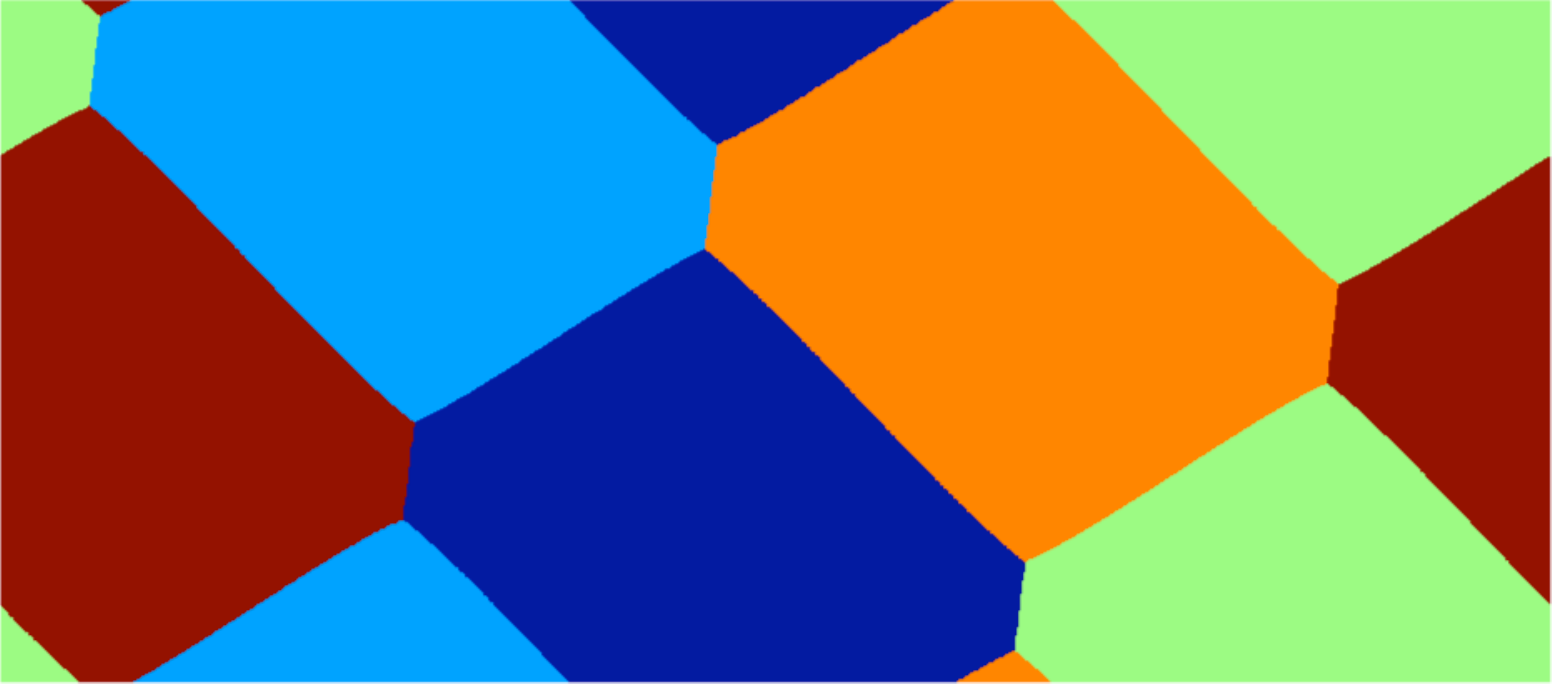} }\hfill
\subfigure[$b=0.45$]{\includegraphics[width=3.4cm]{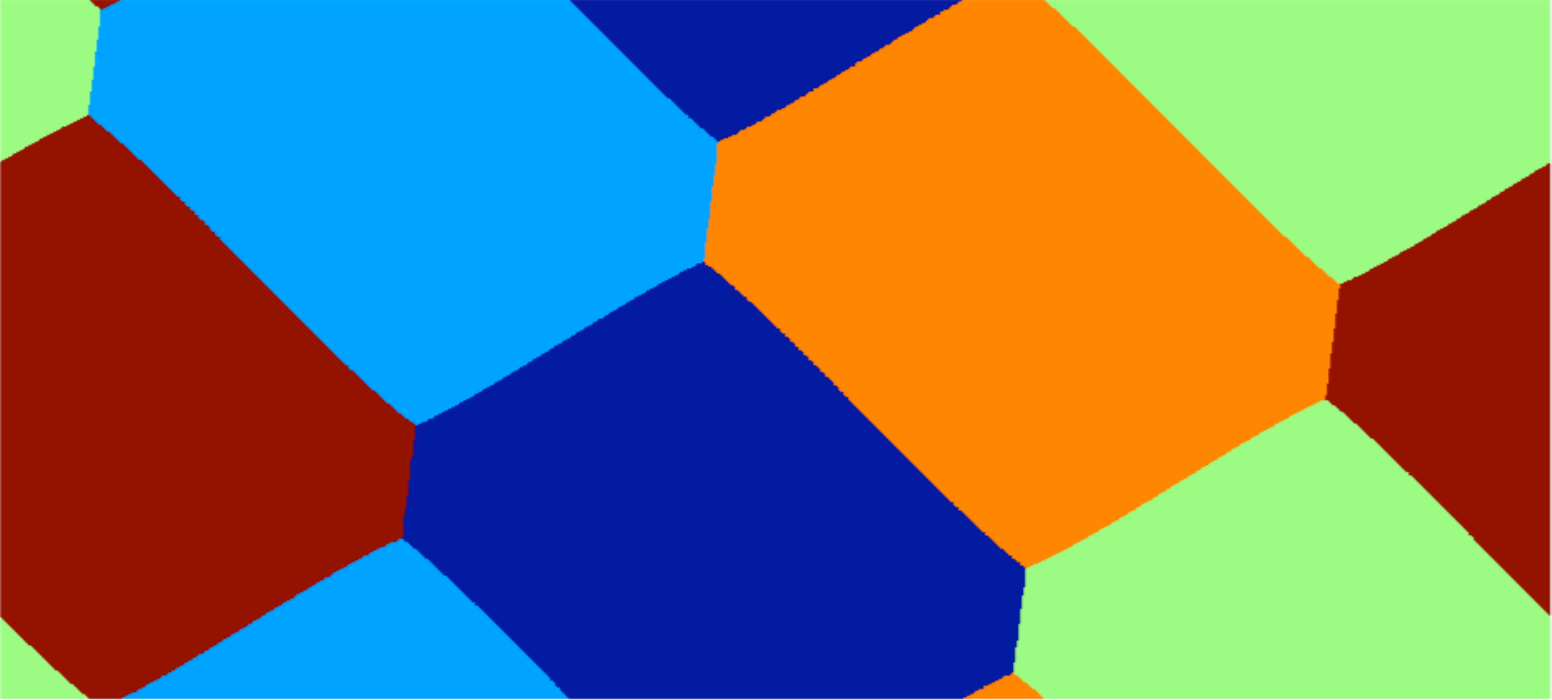} }\hfill
\subfigure[$b=0.5$]{\includegraphics[width=3.4cm]{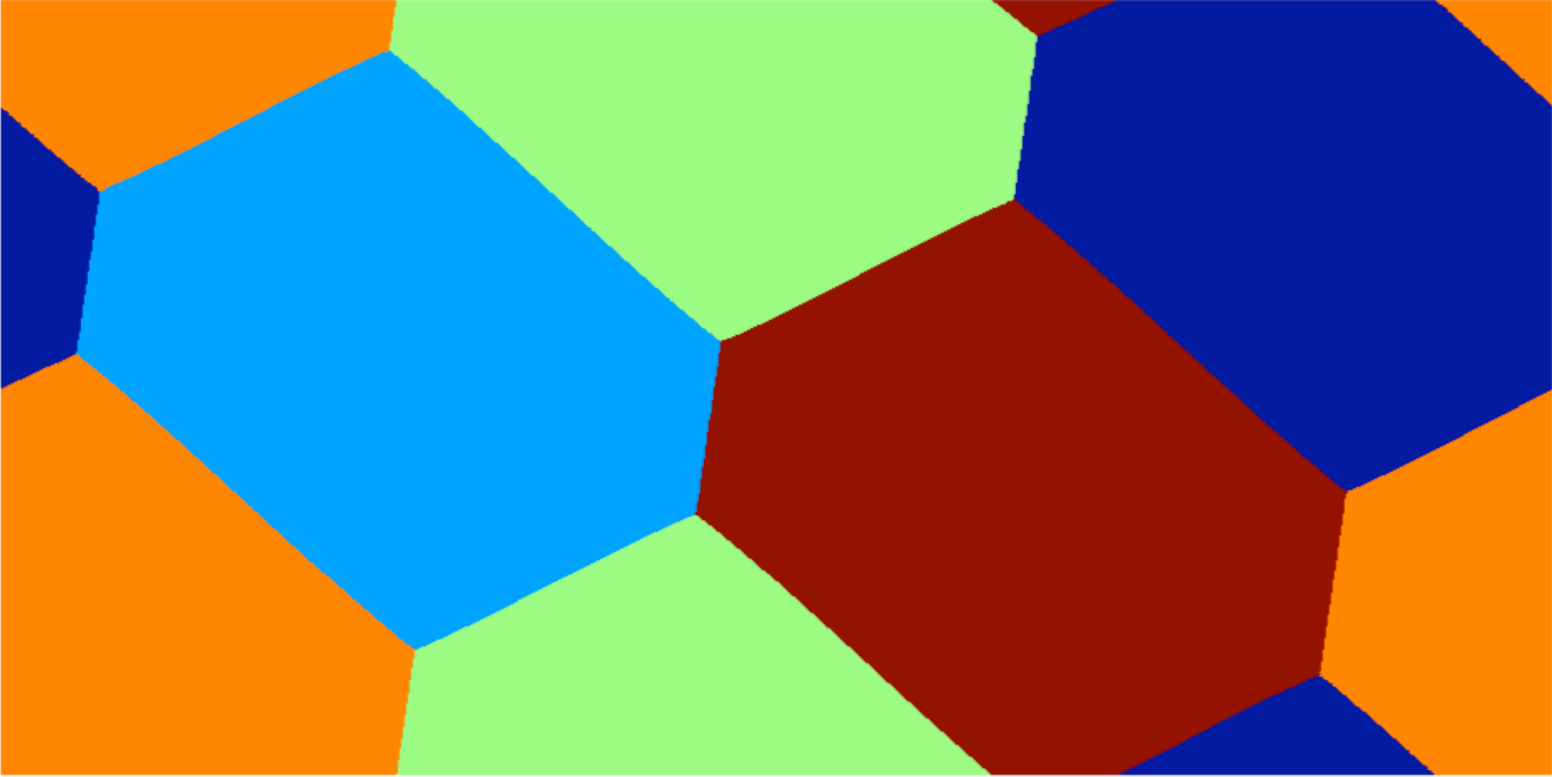} }\hfill
\subfigure[$b=0.7$]{\includegraphics[width=3.4cm]{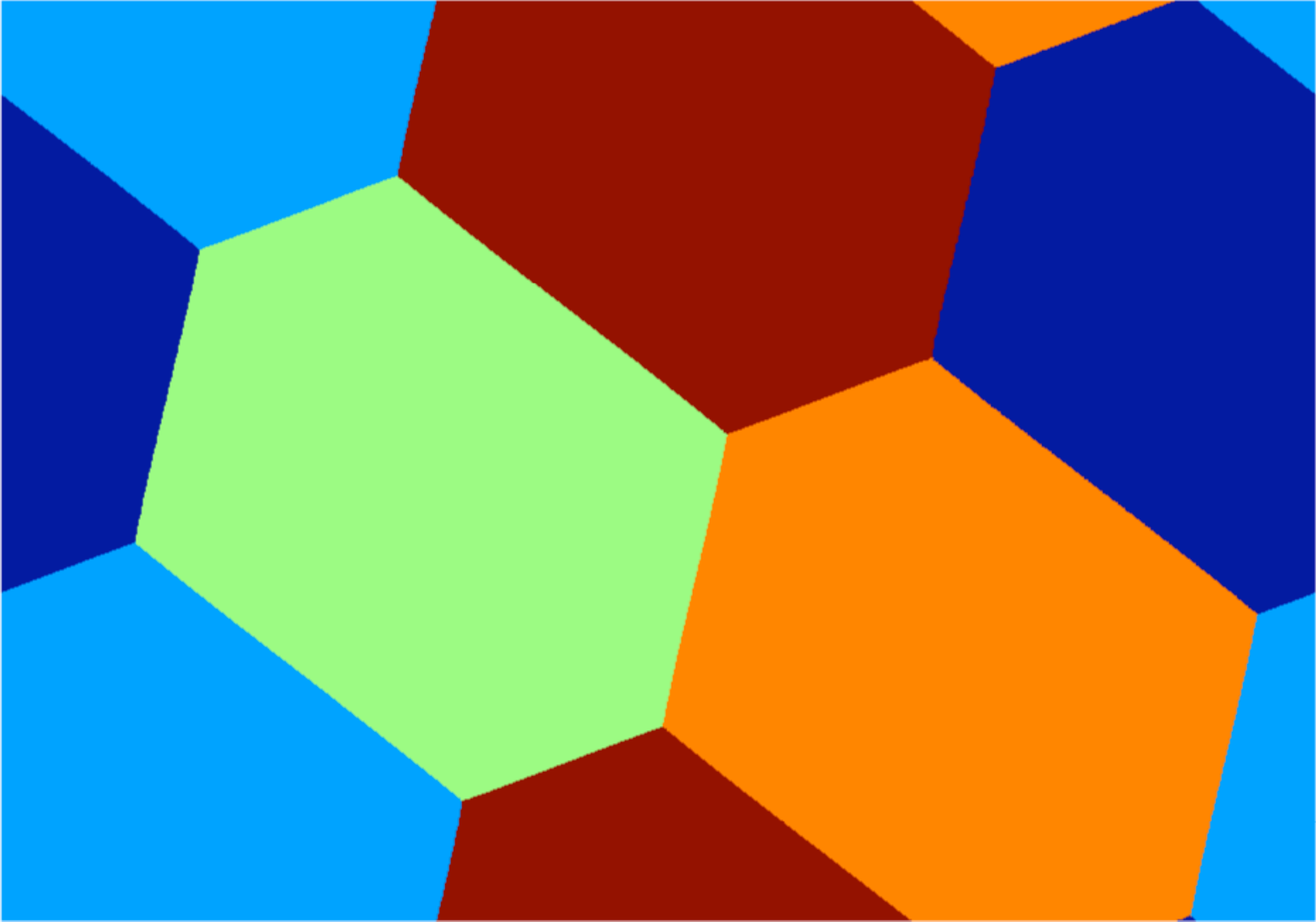} }\\
\subfigure[$b=0.9$]{\includegraphics[width=3.4cm]{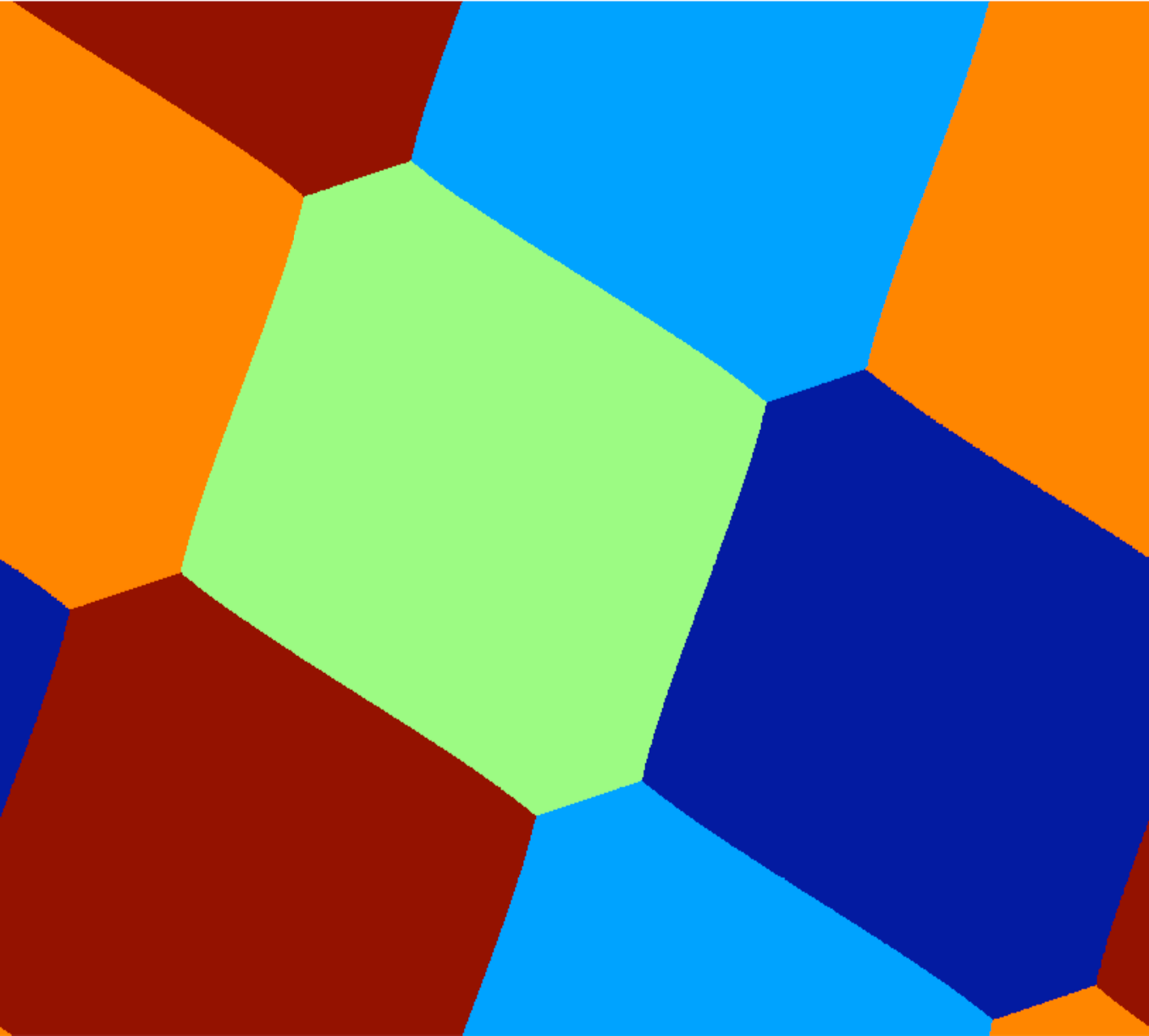} }\hfill
\subfigure[$b=0.98$\label{fig5PartTransCarre1}]{\includegraphics[width=3.4cm]{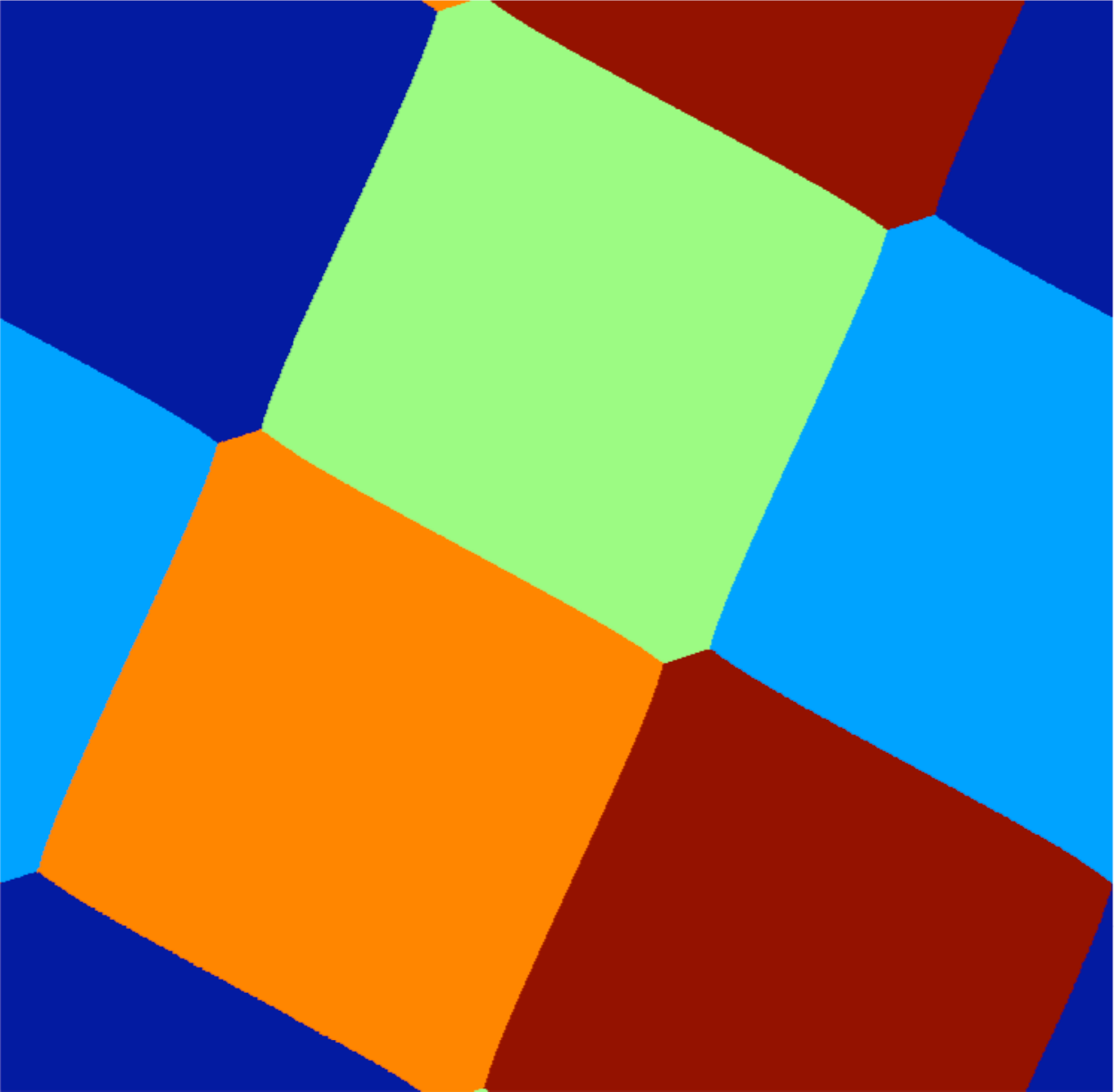}}\hfill
\subfigure[$b=0.99$\label{fig5PartTransCarre2}]{\includegraphics[width=3.4cm]{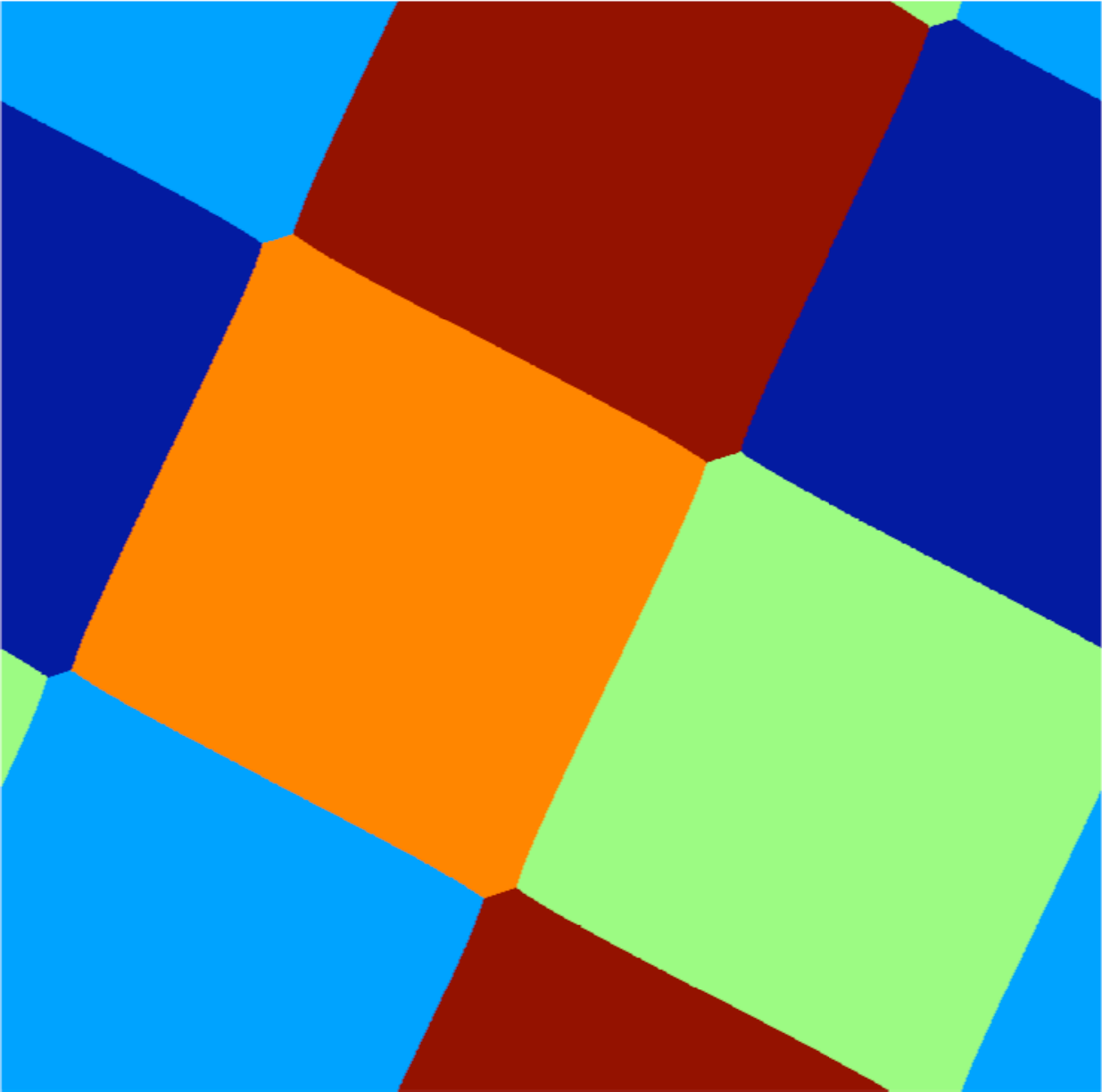}}\hfill
\subfigure[$b=1$\label{fig5PartTransCarre3}]{\includegraphics[width=3.4cm]{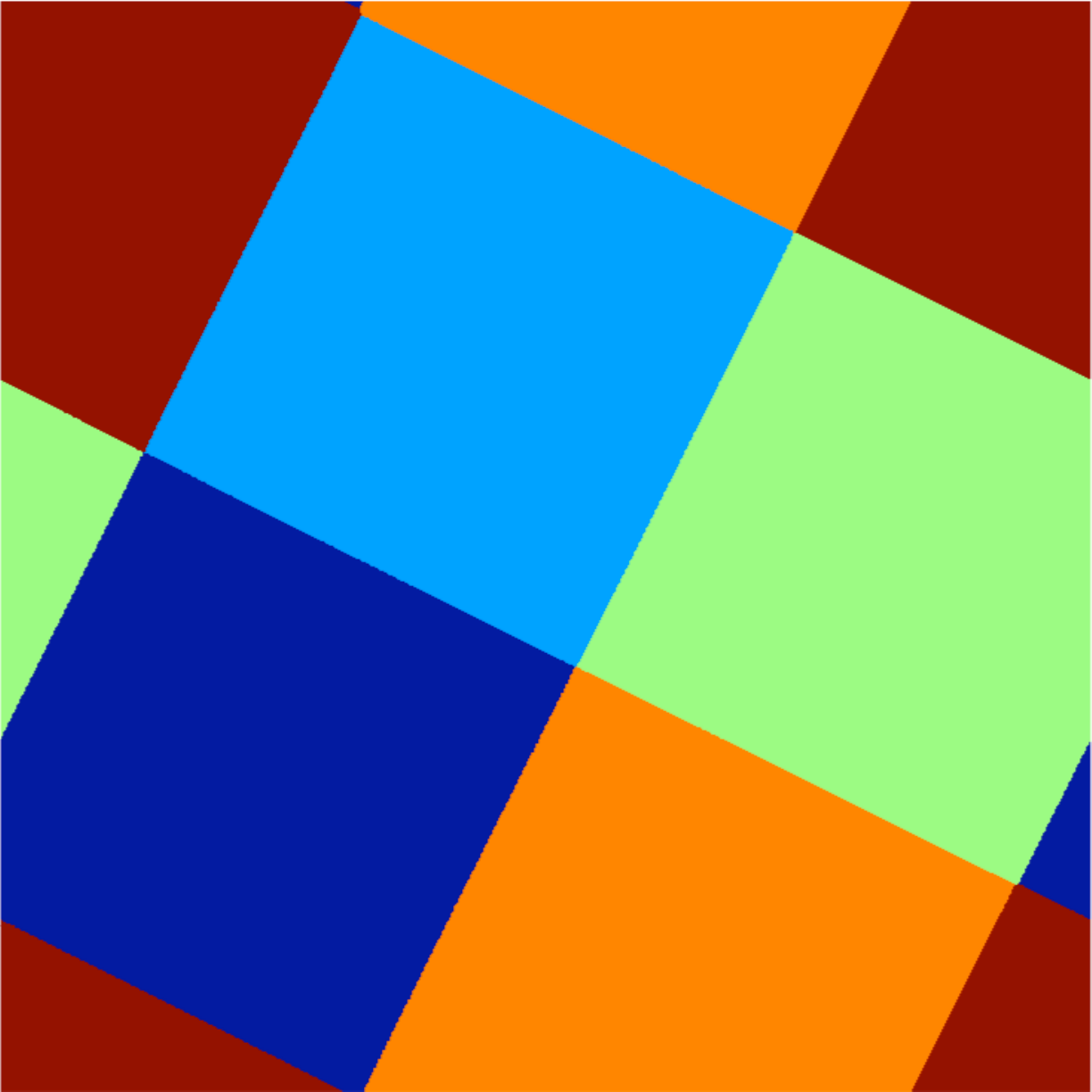}}
\caption{$5$-partitions for some values of $b\,$.
\label{fig5PartTrans}}
\end{center}
\end{figure}

We have conjectured that $b_5$ equals $\bC_5=1/\sqrt{6}\simeq0.408\,$. The first images of Figure \ref{fig5PartTrans} present the best candidates obtained numerically when $b$ is close to $\bC_{5}\,$. They seem to support the conjecture. Furthermore, for $b$ slightly larger than $\bC_5\,$, minimal partitions seem to be hexagonal, with domains close to the rectangles appearing in the partition  on Figure \ref{figPartk10Nod}.

For $b$ between $1/\sqrt{6}$ and $1\,$, minimal partitions appear to be hexagonal. However, pairs of singular points in the boundary of these hexagonal partitions seem to merge when $b$ approaches $1\,$, and for $b=1\,$, the best candidate produced by the algorithm is a partition of $\Tor(1,1)$ into five equal squares. This process is shown on Figures \ref{fig5PartTransCarre1}--\ref{fig5PartTransCarre3}. We obtain again an upper bound for $\mathfrak{L}_5(\Tor(1,b))$ as a function of $b\,$, plotted on Figure \ref{figfigk5.figVP}.

\begin{figure}[h!]\begin{center}
\includegraphics[height=6cm]{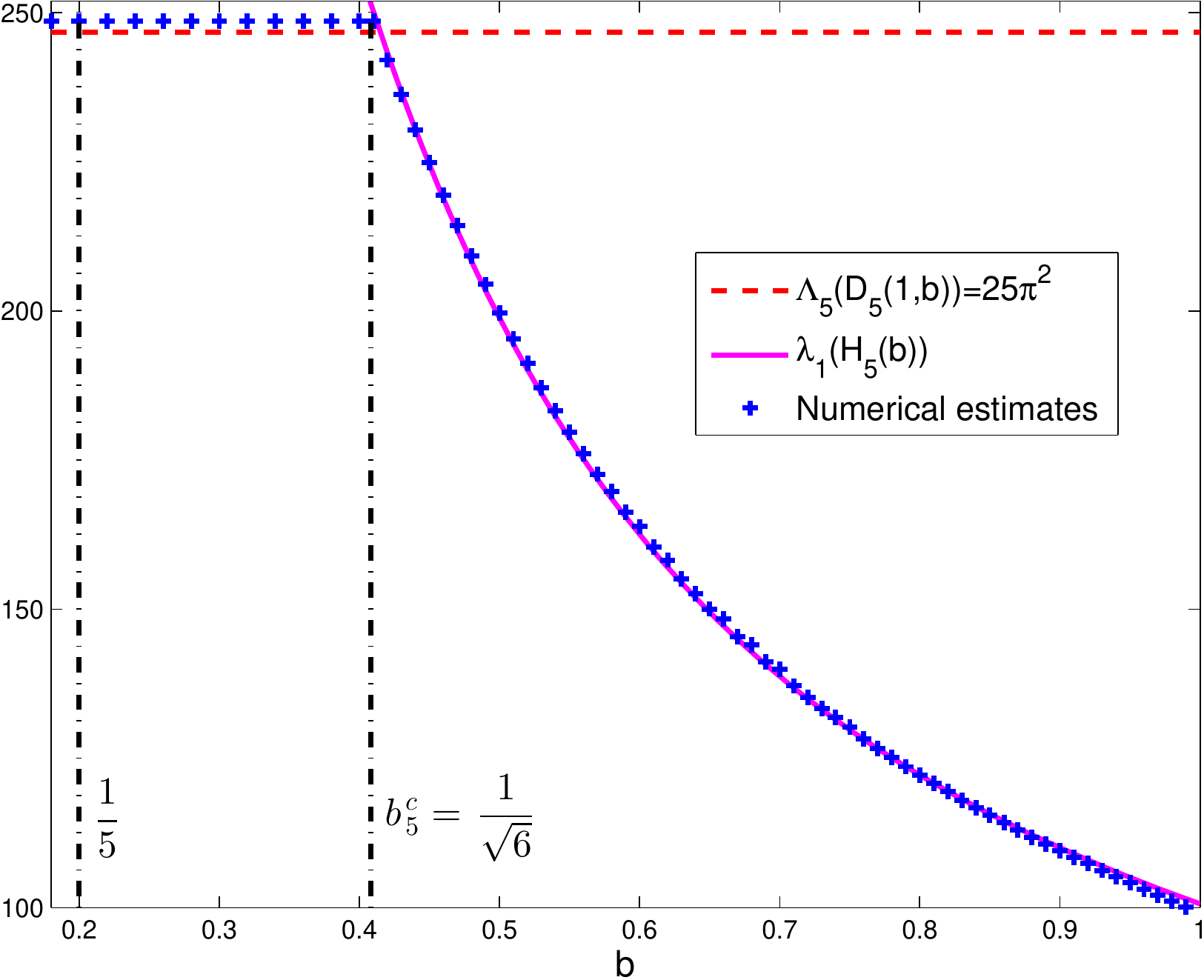} 
\caption{Upper bounds of $\mathfrak L_{5}(\Tor(1,b))$ for $b\in \{j/100\,;\, j=18\,,\,\dots\,,\,100\}\,$.\label{figfigk5.figVP}}
\end{center}\end{figure}

\section{Examples of partitions}
\label{secTilings}

\subsection{Tilings of $\Tor(1,b)$}

The results of Section \ref{secNumerics} suggest that, at least for some values of $k$ and $b\,$, the domains of minimal $k$-partitions of $\Tor(1,b)$ are isometric polygons. In fact, except when $k=5$ and $b=1\,$, these polygons seem to be hexagons. On the other hand, according to Theorem \ref{thmMinPartReg}, any minimal partition satisfies the equal angle meeting property.

This suggests the existence of partitions with a low energy, possibly minimal, that are tilings of  the torus $\Tor(1,b)$ by $k$ identical hexagons satisfying the equal angle meeting property. Let us note that this last property is equivalent to the fact that the interior angle at each vertex of the hexagon is $2\pi/3\,$. Finding these tilings is a purely geometrical problem. In this section, we prove Theorem \ref{thmExkPart} by constructing families of tilings, depending on $k$ and $b\,$, which seem close to the partitions obtained in Section \ref{secNumerics}. We then compute the energy of these tilings, and compare it to the numerical results of Section \ref{secNumerics}.

\subsection{Construction of the tilings}

	\subsubsection{Tilings of the plane}
	\label{subsubsecTilingPlane}
	
To study tilings of the torus $\Tor(1,b)$ by hexagons, it will be useful to consider tilings of the plane. As in Section \ref{subsecProofOdd}, we consider the natural projection map
\[\begin{array}{cccc}
	\Pi_{\infty}: & \RR^2& \rightarrow & \Tor(1,b)\\
	& (x,y) & \mapsto & (x\mbox{ mod }1,y\mbox{ mod }b)\,.\\
\end{array}\]
We denote by $(\mathbf{e}_1,\mathbf{e}_2)$ the canonical basis of $\RR^2\,$, i.e. $\mathbf{e}_1=(1,0)$ and $\mathbf{e}_2=(0,1)\,$.  
	
Let us consider a strong and regular $k$-partition $\mathcal{D}=\{D_1\,,\,\dots\,,\,D_k\}$ of the torus $\Tor(1,b)\,$, such that all the $D_i$'s are isometric to an hexagon that we denote by $\Hex\,$. Let us note that, since $\mathcal{D}$ is strong, the area of $\Hex$ is $b/k$ and, since $\mathcal{D}$ satisfies the equal angle meeting property, all the interior angles of $\Hex$ are $2\pi/3\,$. Let us then consider, for any $i\in\{1\,,\,\dots\,,\,k\}\,$, the open set $\Pi^{-1}(D_i)\,$. It has an infinite number of connected components, each one being isometric to $\Hex\,$. The family of all the connected components of all the sets $\Pi^{-1}(D_i)\,$, for $i\in\{1,\dots,k\}\,$, is a tiling of the plane $\RR^2$ by the hexagon $\Hex\,$. This tiling is invariant under the translations associated with the vectors $\mathbf{e}_1$ and $b\mathbf{e}_2\,$.
	
We can see that, conversely, the image by $\Pi$ of a tiling $\mathcal{T}$ of $\RR^2$ by an hexagon $\Hex$ is a regular $k$-partition of $\Tor(1,b)$ into domains isometric to $\Hex$ and satisfies the equal angle meeting property if the following conditions are verified by $\mathcal{T}\,$.
	\begin{enumerate}[i.]
		\item $\mathcal{T}$ is invariant under the translations associated with the vectors $\mathbf{e}_1$ and $b\mathbf{e}_2\,$.
		\item The area of $\Hex$ is $\frac{b}{k}\,$.
		\item All the interior angles of $\Hex$ are $2\pi/3\,$.
	\end{enumerate} 
	We have therefore reformulated the original problem. We now look for the tilings of $\RR^2$ that satisfy properties i--iii, and, if possible, for an algorithm to construct those tilings.
	
\subsubsection{Change of basis} \label{subsubsecBasis}
	
Let $\mathcal{T}$ be an hexagonal tiling of $\RR^2$ satisfying properties i--iii above. The following definition will be useful to describe $\mathcal{T}\,$.
	
\begin{defin} \label{defAdaptedBasis} We say that a basis of $\RR^2$ is \emph{$\mathcal{T}$-adapted} if its vectors connect the center of a tiling domain to the centers of two neighboring domains, with these two neighboring domains having a common edge (see Figure \ref{figTilingBasis}).
\end{defin}
Let us now denote by $(\mathbf{u}_1$,$\mathbf{u}_2)$ a $\mathcal{T}$-adapted basis.

\begin{figure}[h!t]
\begin{center}
\subfigure[$3$-partition of a torus with an adapted basis.\label{figTilingBasis}]{\includegraphics[width=4cm]{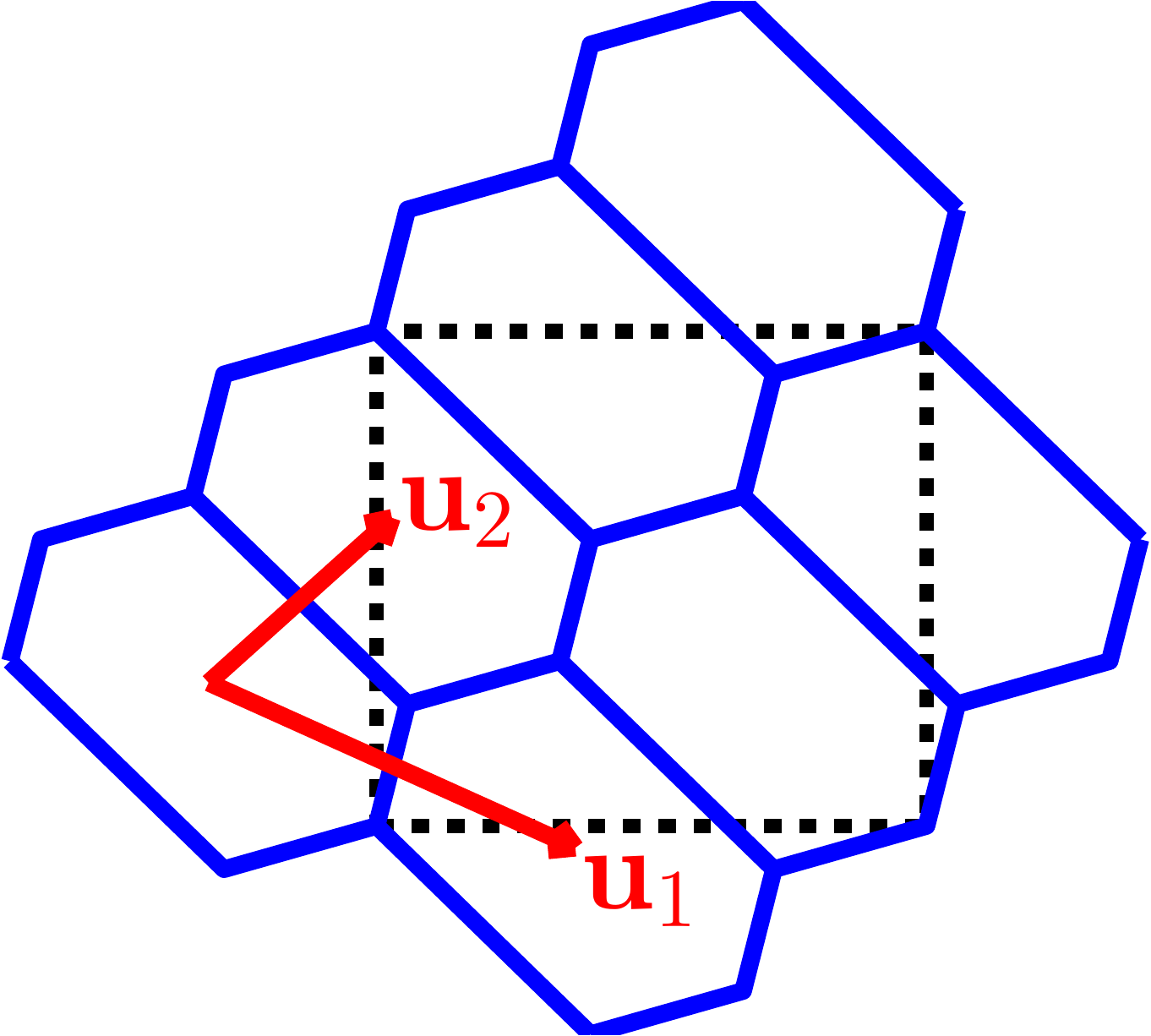}}
	\hfill
\subfigure[Area of the tiling domain.\label{figArea}]{\includegraphics[width=4cm]{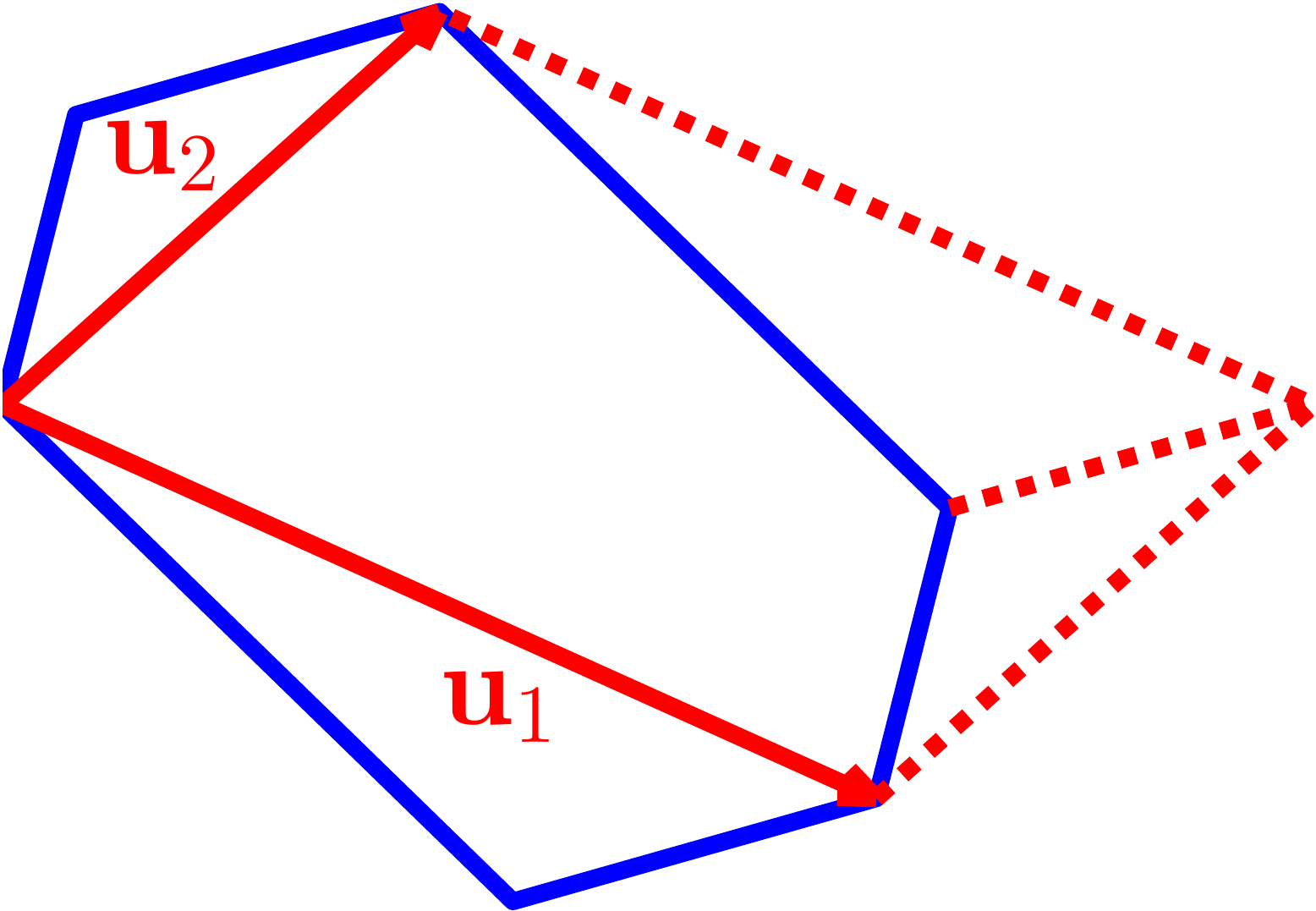}}
	\hfill
\subfigure[Construction of the tiling domain.\label{figTilingFermat}]{\includegraphics[width=4cm]{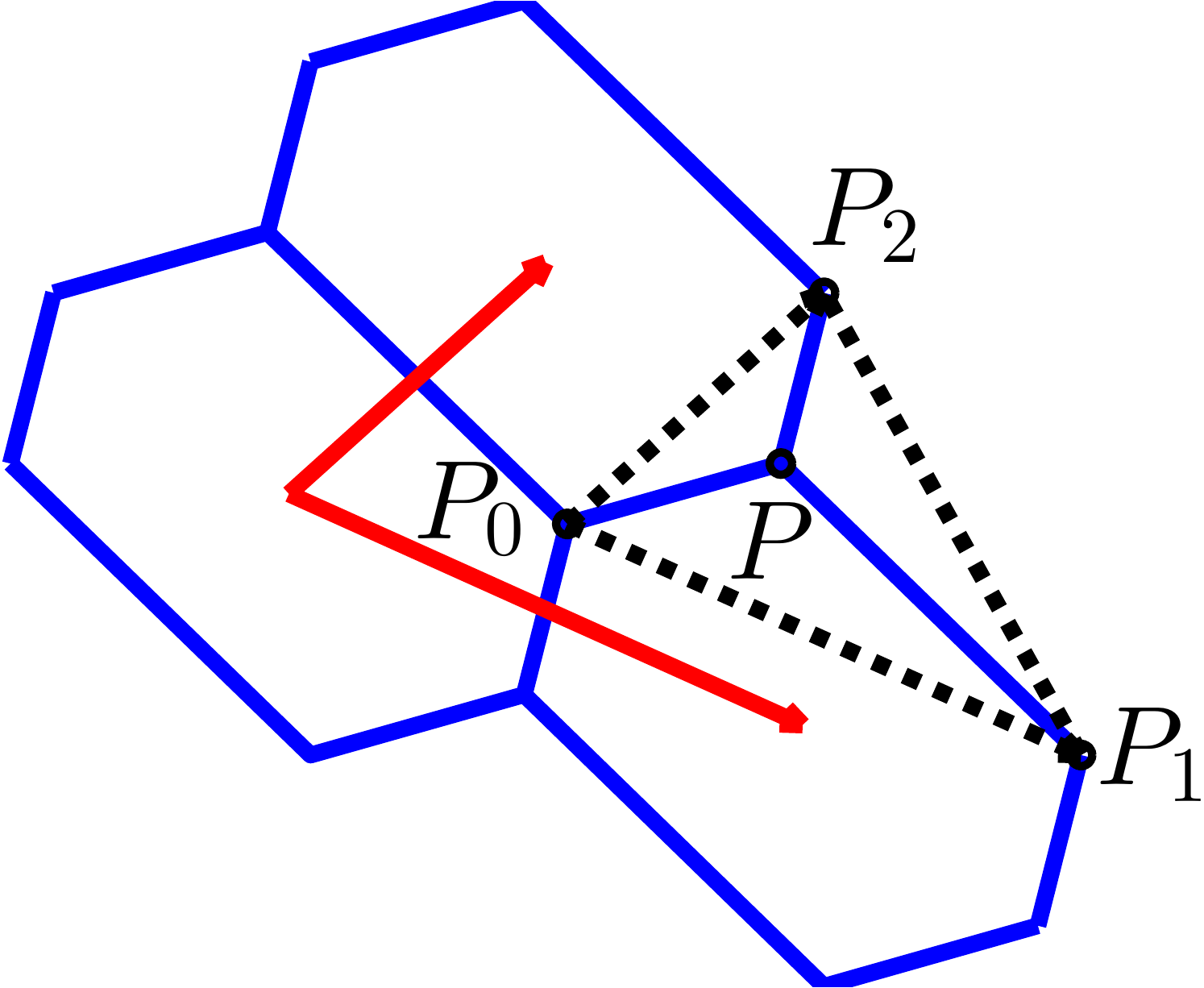}}
\caption{Hexagonal tiling.}
\end{center}
\end{figure}	

\begin{lem} \label{lemBasis}
There exists a $2\times 2$ matrix $V=(v_{i,j})$ with integer coefficients such that
\begin{equation*}
	\begin{cases}
		\mathbf{u}_1=\displaystyle\frac{v_{1,1}}{k}\mathbf{e}_1+\frac{v_{2,1}}{k}(b\mathbf{e}_2)\,,\\[8pt]
		\mathbf{u}_2=\displaystyle\frac{v_{1,2}}{k}\mathbf{e}_1+\frac{v_{2,2}}{k}(b\mathbf{e}_2)\,,
	\end{cases}
\end{equation*}
and $\det V=\pm k\,$. 
\end{lem}   
	
\begin{proof} 
Since $\mathcal{T}$ is invariant under the translations associated with the vectors $\mathbf{e}_1$ and $b\mathbf{e}_2\,$, there exist integers $s_{1,1}$, $s_{2,1}$, $s_{1,2}$ and $s_{2,2}$ such that
\begin{equation*}
	\begin{cases}
		\mathbf{e}_1=s_{1,1}\mathbf{u}_1+s_{2,1}\mathbf{u}_2\,;\\
		b\mathbf{e}_2=s_{1,2}\mathbf{u}_1+s_{2,2}\mathbf{u}_2\,.\\
	\end{cases}
\end{equation*} 
With the notation $S=(s_{i,j})\,$, $\mathbf{u}_1=(u_{1,1},u_{2,1})\,$, $\mathbf{u}_2=(u_{1,2},u_{2,2})\,$, and $U=(u_{i,j})\,$, we have
\[\left(\begin{array}{cc}
	1&0\\
	0&b\\
\end{array}\right)=US\,.\]
Thus $\det U\, \det S=b\,$. But $|\det U|$ is the area of the tiling domain $\Hex$ (see Figure \ref{figArea}). Therefore $\det U=\pm b/k$ and $\det S=\pm k\,$. This implies that the matrix $V=k S^{-1}$ has integer coefficients, and therefore has the desired properties.
\end{proof}

One can give a geometrical interpretation of the coefficients in the matrix $V\,$. Let us go back to the torus $\Tor(1,b)\,$. We assume that the matrix $V=(v_{i,j})$ is associated  to the $\mathcal{T}$-adapted basis $(\mathbf{u}_1,\mathbf{u}_2)$ (see Figure \ref{figTilingBasis} for an example). If we start from some hexagonal domain and translate it $k$ times in the $\mathbf{u}_1$ direction, it returns to its original position after turning $v_{1,1}$ times around the torus in the horizontal direction and $v_{2,1}$ times in the vertical direction. Similarly, if we translate the domain $k$ times in the $\mathbf{u}_2$ direction, it returns to its original position after turning $v_{1,2}$ times around the torus in the horizontal direction and $v_{2,2}$ times in the vertical direction. We can therefore say that the matrix $V$ describes how the hexagonal tiling $\mathcal{T}$ wraps around the torus $\Tor(1,b)\,$.
	
The following result tells us at which condition we can solve the converse problem, that is to say find a tiling associated with a given basis. 
\begin{lem}\label{lemExistTiling}
Let $(\mathbf{u}_1,\mathbf{u}_2)$ be a basis of $\RR^2\,$, such that there exists a matrix $V$ satisfying the properties of Lemma \ref{lemBasis}. Let $P_0$ be some point in $\RR^2$ and let us note $P_1=P_0+\mathbf{u}_1$ and $P_2=P_0+\mathbf{u}_2\,$. The two following statements are equivalent.
\begin{enumerate}[\upshape i.]
	\item There exists an hexagonal tiling $\mathcal{T}$ such that $(\mathbf{u}_1,\mathbf{u}_2)$ is $\mathcal{T}$-adapted.
	\item There exists a point $P$ in the interior of the triangle $P_0P_1P_2$ such that the segments $P_0P\,$, $P_1P\,$, and $P_2P$ meet with equal angles.
\end{enumerate}
\end{lem}
	
\begin{proof} Let us first consider the direct implication. We choose some domain of the tiling $\mathcal{T}$ and denote by $P_0$ the vertex that connects the sides of this domain intersected by $\mathbf{u}_1$ and $\mathbf{u}_2\,$ (see Figure \ref{figTilingFermat}). Then $P_1$ and $P_2$ are vertices of the tiling, and there is another vertex $P$ contained in the triangle $P_0P_1P_2$. The point $P$ has the desired property.\\		
Conversely, let us assume ii. After translating the three segments $P_0P\,$, $P_1P\,$, and $P_2P$ according to all the vectors of the lattice $\ZZ\mathbf{u}_1+\ZZ\mathbf{u}_2\,$, we obtain the boundary of a tiling $\mathcal{T}$ which satisfies the three properties i--iii stated at the beginning of this subsection, and the basis $(\mathbf{u}_1,\mathbf{u}_2)$ is $\mathcal{T}$-adapted. 
\end{proof}

\subsubsection{Reconstruction of the tiling domain}
We now try to determine at which condition Property ii. of Lemma \ref{lemExistTiling} is satisfied. We recall, without proof, a very classical geometrical result (see for instance \cite{Cox89}).
	
\begin{thm}\label{thmFermat}
Let $P_0\,$, $P_1\,$, and $P_2$ be three non-colinear points in $\RR^2$ (see Figure \ref{figFermat}). One of the two following situations occurs.
\begin{enumerate}[\upshape i.]
	\item If all three angles $\alpha_0\,$, $\alpha_1\,$, and $\alpha_2$ of the triangle $P_0P_1P_2$ are smaller than $\frac{2\pi}{3}\,$, there is a unique point $P$ belonging to the interior of the triangle $P_0P_1P_2$ such that the segments $P_0P\,$, $P_1P\,$ and $P_2P\,$ meet with equal angles at $P\,$. The point $P$ is called the \emph{Fermat point} of the triangle $P_0P_1P_2\,$. It is the point of minimum for the function $Q\mapsto QP_0+QP_1+QP_2\,$.
	\item If $\alpha_i \ge \frac{2\pi}{3}$ for some $i\in\{0\,,\,1\,,\,2\}$, then there is no point in the interior of  $P_0P_1P_2$ at which the segments from the vertices meet with equal angles. In that case, the function $Q\mapsto QP_0+QP_1+QP_2$ reaches its minimum at $P_i\,$.
\end{enumerate}
\end{thm}
	
\begin{figure}
\begin{center}
	\subfigure[Equal angle property.\label{figFermat}]{\includegraphics[width=5cm]{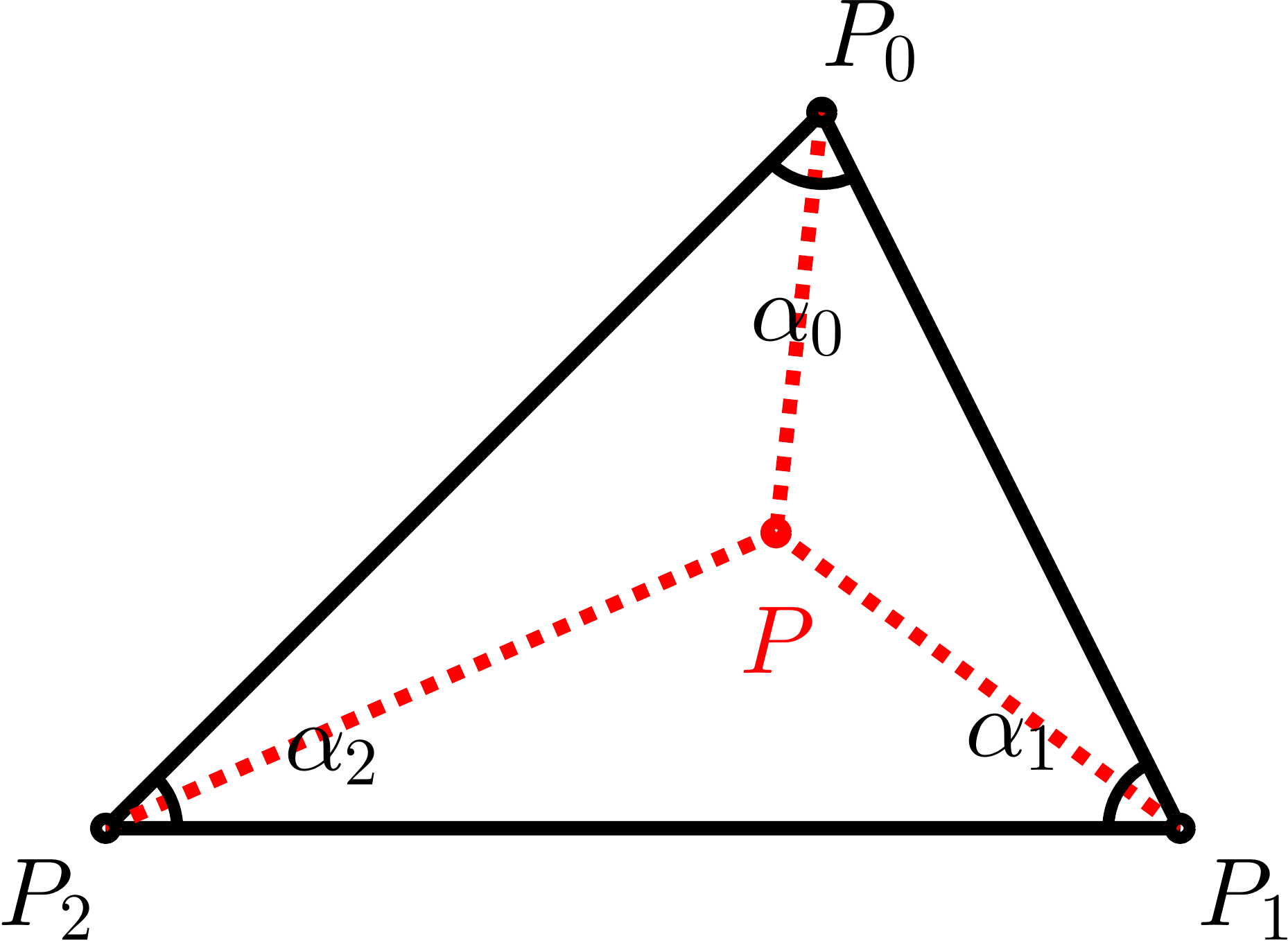}}
	\hspace{2cm}
	\subfigure[Construction of the Fermat point.\label{figConstFermat}]{\includegraphics[width=5cm]{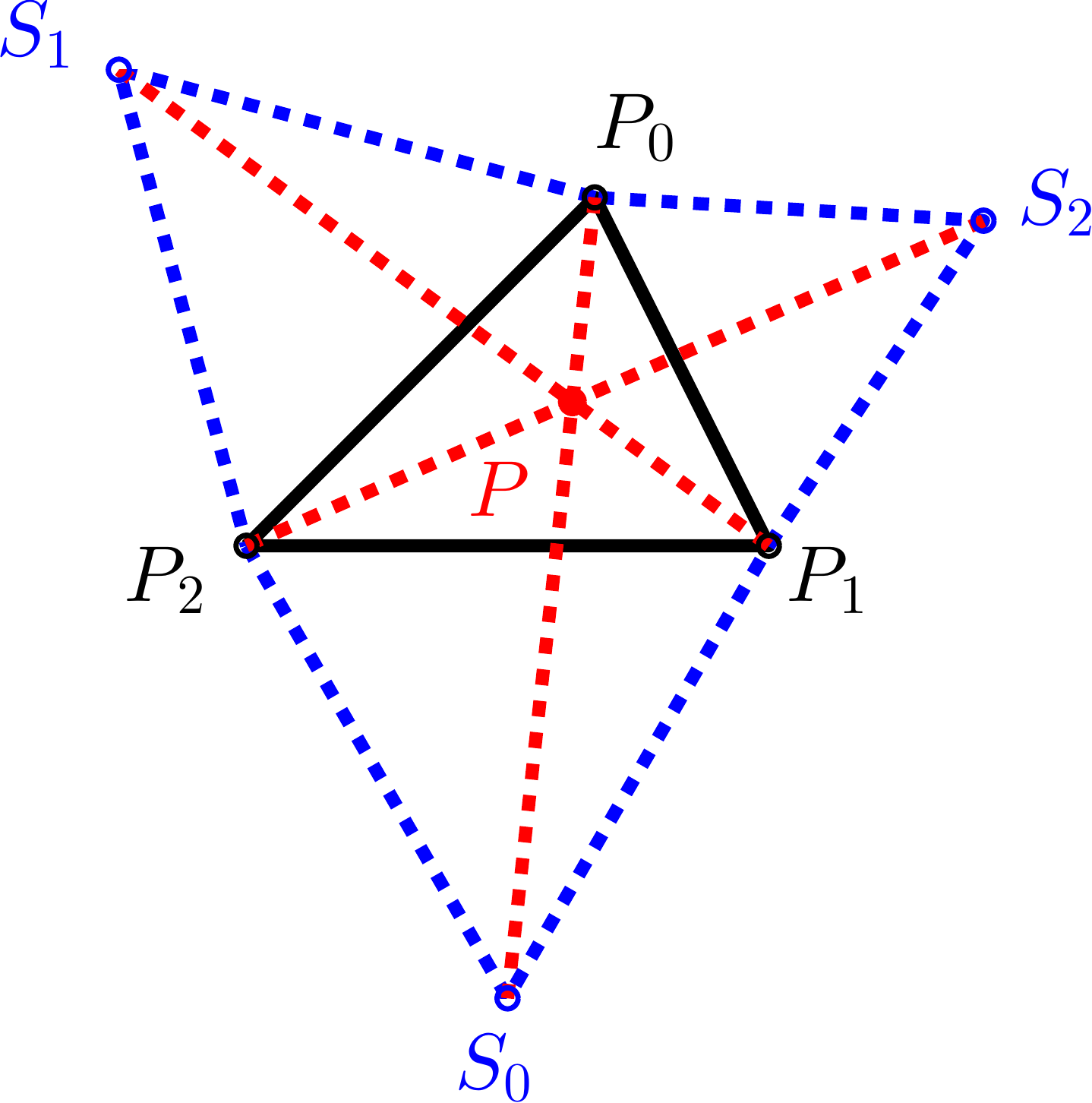}}
	\caption{Fermat point of a triangle.\label{figFermatPoint}}
\end{center}
\end{figure}
	
The following result gives an easy criterion for the existence of a Fermat point.
	
\begin{lem}\label{lemCondFermat} 
Let $\mathbf{u}_1$ and $\mathbf{u}_2$ be two non-zero vectors in $\RR^2\,$. Let $P_0$ be a point in $\RR^2\,$. We set $P_1=P_0+\mathbf{u}_1$ and $P_2=P_0+\mathbf{u}_2\,$. The triangle $P_0P_1P_2$ has a Fermat point if, and only if, 
\begin{equation}\label{eqCondFermat1}
	p\in\Big(-\frac{1}{2},\frac{1}{2}\Big]  
\end{equation}
or
\begin{equation}\label{eqCondFermat2}
	p\in \Big(\frac{1}{2},1\Big) \quad\mbox{ and }\quad p-\sqrt{\frac{1-p^2}{3}}<r<\frac{1}{p-\sqrt{\frac{1-p^2}{3}}}\,,
\end{equation}
with $p=\frac{\mathbf{u}_1\cdot\mathbf{u}_2}{\|\mathbf{u}_1\|\|\mathbf{u}_2\|}$ and $r=\frac{\|\mathbf{u}_1\|}{\|\mathbf{u}_2\|}\,$.
\end{lem}
\begin{proof}
The proof is a rather straightforward computation, and we merely indicate its steps. We express $\cos\alpha_i\,$, for $i \in \{0\,,\,1\,,\,2\}\,$, as a function of $p$ and $r$. Writing down the condition
\begin{equation*}
	\forall i \in\{0\,,\,1\,,\,2\}\,,\quad\cos\alpha_i\in (-\tfrac{1}{2},1)\,,
\end{equation*}
we show that it is equivalent to the alternative \eqref{eqCondFermat1} or \eqref{eqCondFermat2}.
\end{proof}
	
\subsubsection{Algorithm}	
The above results give us an algorithm to build a tiling of the torus $\Tor(1,b)$ by $k$ hexagons:
\begin{itemize}
	\item choose a $2\times 2$ matrix $V$ with integer coefficients such that $\det V=\pm k\,$;
	\item check whether the triangle generated by the vectors $\mathbf{u}_1$ and $\mathbf{u}_2$ (defined from $V$ as in Lemma  \ref{lemBasis}) has a Fermat point, using Lemma \ref{lemCondFermat};
	\item if the triangle has a Fermat point, compute its coordinates;
	\item use the coordinates of the Fermat point to build the tiling domain.
\end{itemize}
	
Let us describe in more details how we perform the last two steps. We first recall a geometric construction of the Fermat point (see for instance \cite{Cox89}).
\begin{thm}\label{thmConstFermat}
Let $P_0P_1P_2$ be a triangle in $\RR^2$ such that each of the angles $\alpha_i\,$, $i\in\{0\,,\,1\,,\,2\}\,$, is smaller than $\frac{2\pi}{3}\,$. Let us consider the three equilateral triangles lying outside of $P_0P_1P_2$ and having one edge in common with it. For each of these triangles, let us consider the line passing through the outer vertex and the vertex of $P_0P_1P_2$ that does not belong to it (see Figure \ref{figConstFermat}). The three lines meet at the Fermat point $P$.
\end{thm}
	
Let us assume that we have performed the first two steps in the algorithm. We now have two vectors $\mathbf{u}_1$ and $\mathbf{u}_2$ such that the basis $(\mathbf{u}_1,\mathbf{u}_2)$ is adapted to some tiling $\mathcal{T}$. We choose (arbitrarily) some point $P_0$ in $\RR^2\,$. We can then build the points $P_1$ and $P_2$ of Figure \ref{figTilingFermat}. Using the construction of Theorem \ref{thmConstFermat}, we can find the coordinates of the Fermat point $P$ of the triangle $P_0P_1P_2\,$. The segments $P_0P\,$, $PP_2\,$, and $P_1P$ then define three successive sides of the tiling domain, which is enough to construct the tiling domain itself.  		
	
\subsection{Examples}
We now look for examples of tilings with given matrices $V$, suggested by the numerical simulations in Section \ref{secNumerics}. By applying Lemma \ref{lemCondFermat}, we deduce the following result, which is a more precise version of Theorem \ref{thmExkPart}. 
	
\begin{prop} 
\label{propHexTilings}
For $k\in \{3\,,\,4\,,\,5\}\,$, there exists a tiling of the torus $\Tor(1,b)\,$, with an associated matrix $V_k\,$, if, and only if, $b\in \left(\bH_k\,,\,1\right]\,$, where
\begin{itemize}
	\item $V_3= \begin{pmatrix} 2 & 1\\ -1 & 1\end{pmatrix}$ and $\bH_3=\frac{\sqrt{11}-\sqrt3}{4}\,$;
	\item $V_4=\begin{pmatrix} 1&-1\\ 2&2\end{pmatrix}$ and $\bH_4=\frac1{2\sqrt{3}}\,$;
	\item $ V_5=\begin{pmatrix} 1&-1\\ 2&3\end{pmatrix}$ and $ \bH_5=\frac{\sqrt{291}-5\sqrt{3}}{36}\,$.
\end{itemize}
If it exists, this tiling is a $k$-partition of $\Tor(1,b)\,$, and we denote by $\Hex_k(b)$ the corresponding tiling domain. We have
\[\mathfrak{L}_k(\Tor(1,b))\le \min\left(k^2\pi^2, \lambda_1(\Hex_k(b))\right),\quad \forall b\in(\bH_{k},1] \,.\]
\end{prop}
We use the finite element library {\sc M\'elina} \cite{melina} to give an accurate upper bound of the first eigenvalue $\lambda_{1}(\Hex_{k}(b))$. This upper bound is represented in Figures~\ref{figfigk3.figVP}, \ref{figfigk4.figVP} and \ref{figfigk5.figVP}. We observe that the upper bound by $\lambda_{1}(\Hex_{k}(b))$ is very close to the numerical simulations for $b$ not too close to $\bC_{k}$ (and not too close to $1$ when $k=5$).

\subsection{The case $b=1$}
Let us consider the special case $b=1\,$, in which the tiling domains can be described simply. The hexagonal tilings give a majoration of $\mathfrak{L}_k(\Tor(1,b))\,$, for $k\in \{3\,,\,4\,,\,5\}\,$. Nevertheless, for $k=5\,$, numerical computations show that a tiling of $\Tor(1,1)$ by $5$ squares, represented on Figure \ref{figPart5T11}, has a lower energy.

\begin{figure}
  \begin{center}
  \subfigure[A hexagonal tiling domain for $k=3\,$.\label{figDomTilingk3T11}]{
    \includegraphics[height=3cm]{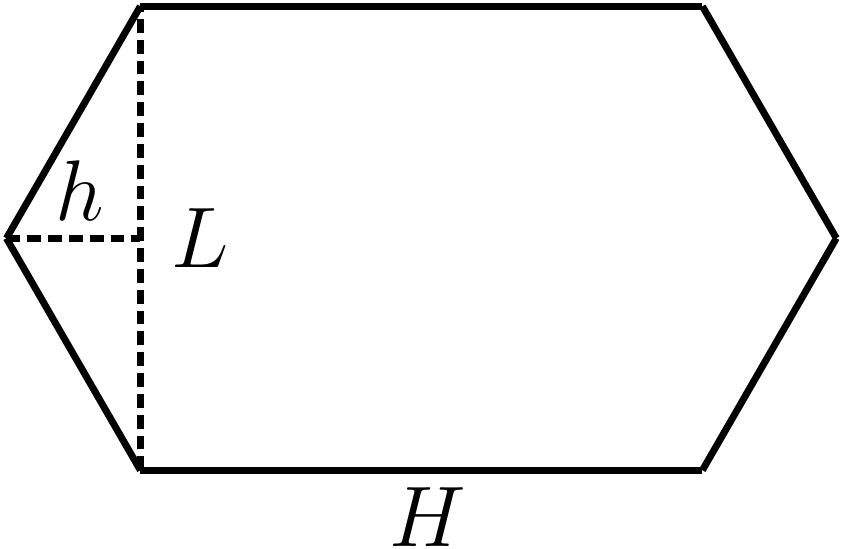}}
    \hfill
    \subfigure[A hexagonal tiling domain for $k=4\,$.\label{figDomTilingk4T11}]{\includegraphics[height=4cm]{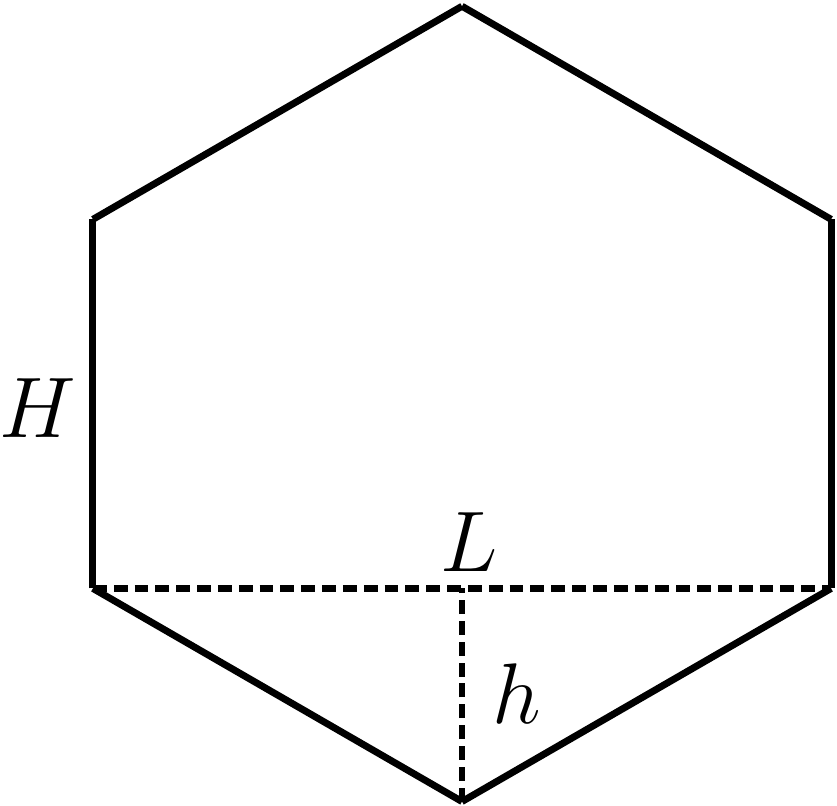}}
    \hfill
    \subfigure[A tiling by five squares.\label{figPart5T11}]{\includegraphics[height=4cm]{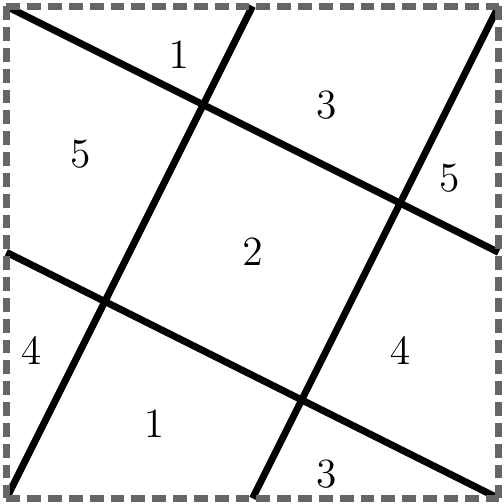}}
    \caption{The case of $\Tor(1,1)\,$.}
\end{center}
\end{figure}

\begin{prop}\label{propLkT11Upper}We have the following upper bound on the minimal energy.
\begin{enumerate}[i.]
	\item $\mathfrak{L}_3(\Tor(1,1))\le\lambda_1(\Hex_3(1))\simeq 62.8389\,$ where $\Hex_3(1)$ is the hexagon shown in Figure \ref{figDomTilingk3T11}, with $L=\frac{\sqrt2}3\,$, $h=\frac1{3\sqrt6}\,$, and $H=\frac1{\sqrt2}-\frac1{3\sqrt6}\,$.
	\item $\mathfrak{L}_4(\Tor(1,1))\le\lambda_1(\Hex_4(1))\simeq 74.9467\,$ where $\Hex_4(1)$ is the hexagon shown in Figure \ref{figDomTilingk4T11}, with  $L=\frac12\,$, $h=\frac1{4\sqrt3}\,$, and $H=\frac{1}{2}-\frac{1}{4\sqrt{3}}\,$. 
	\item $\mathfrak{L}_5(\Tor(1,1))\le \lambda_{1}(\mathsf Q)=10\pi^2\simeq 98.6960\,$ where $\mathsf Q$ is a square of side $\frac1{\sqrt5}\,$, and is the tiling domain of a $5$-partition of $\Tor(1,1)\,$, as seen in Figure \ref{figPart5T11}.
\end{enumerate} 
\end{prop}

\begin{conj}\label{conjLkT11Upper} The three inequalities in Proposition \ref{propLkT11Upper} are actually equalities.
\end{conj}

\subsection{Comparison with the numerical results}
Let us focus for a moment on the case $k=3$, for $b$ close to $\bC_3=1/\sqrt{2}\,$. Figure \ref{figEnergyTrans3} shows that there exists $b^*>1/\sqrt{2}$
such that, for $b \in \left[\bC_3,b^*\right)\,$, $\lambda_1(\Hex_3(b))>9\pi^2\,$. Therefore, for $b \in \left[\bC_3,b^*\right)\,$, the tiling of $\Tor(1,b)$ by three hexagons with straight edges that we have constructed is not minimal. The numerical method of Section \ref{secNumerics} generates better candidates.
\begin{figure}[!h]
  \begin{center}
  \subfigure[$k=3$ and $b$ close to $\bC_3$.\label{figEnergyTrans3}]{
    \includegraphics[height=5cm]{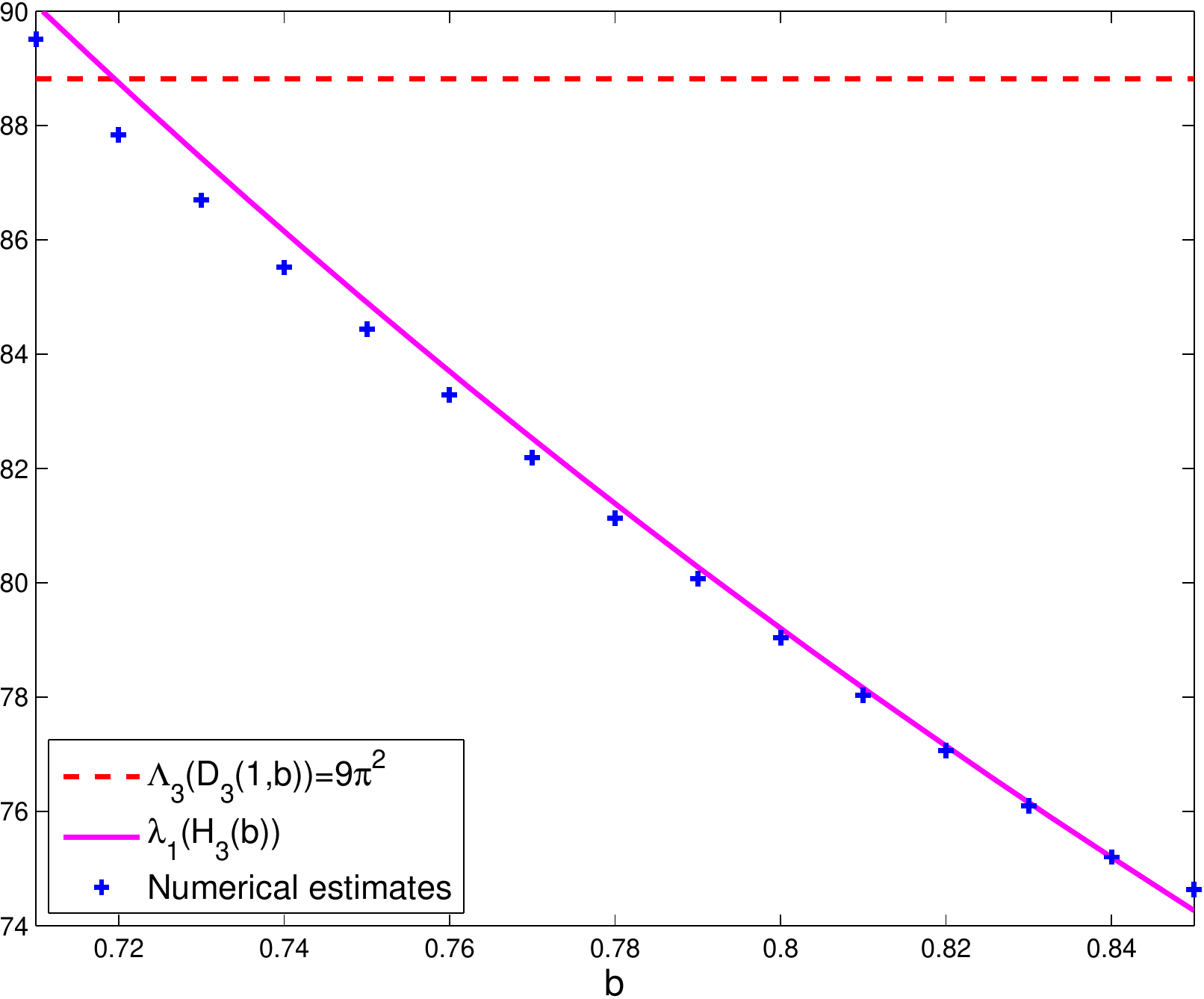}} 
    \hspace{2cm}
    \subfigure[$k=4$ and $b$ close to $b_4$.\label{figEnergyTrans4}]{\includegraphics[height=5cm]{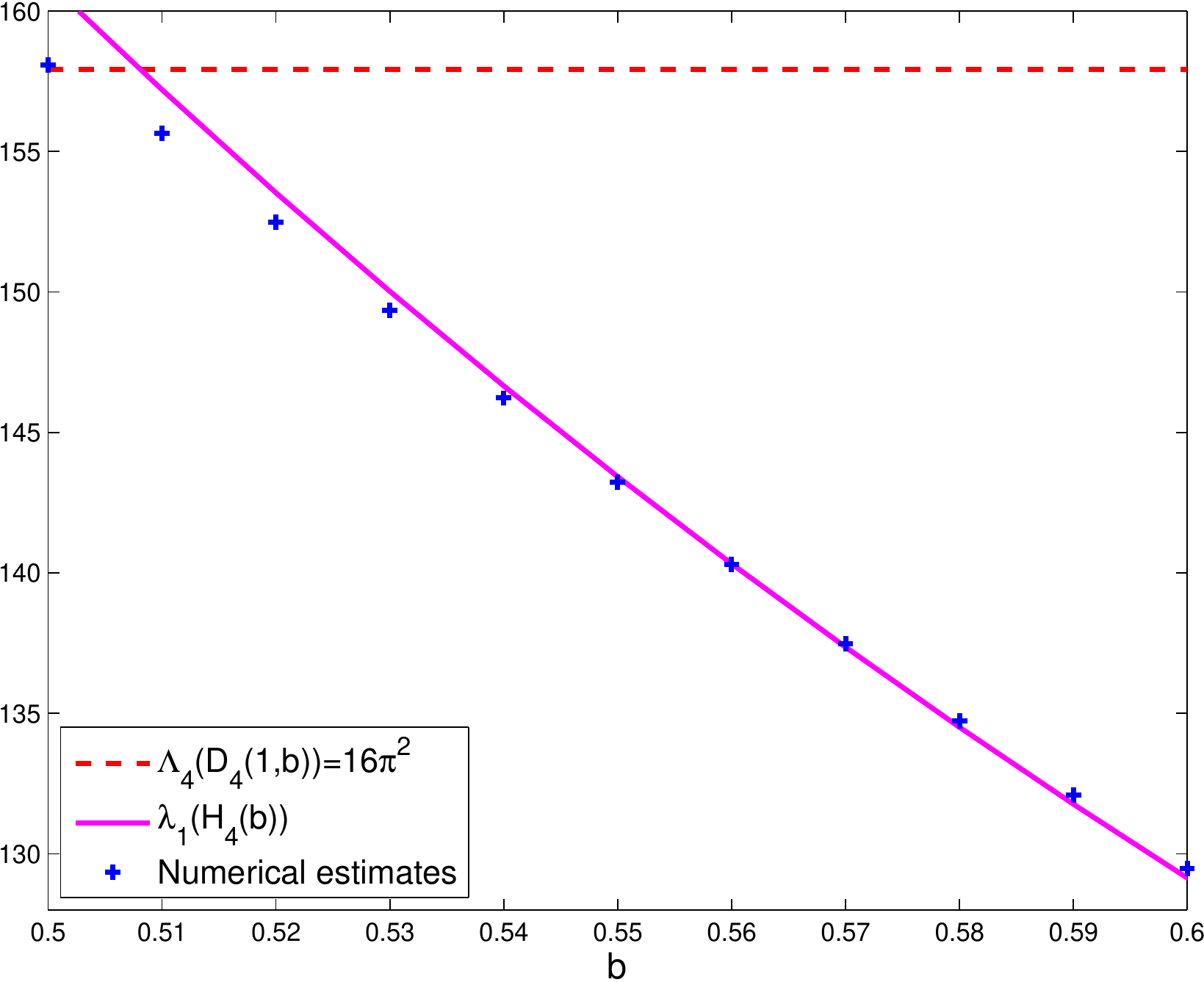}} 
    \\
    \subfigure[$k=5$ and $b$ close to $\bC_5$.\label{figEnergyTrans5}]{\includegraphics[height=5cm]{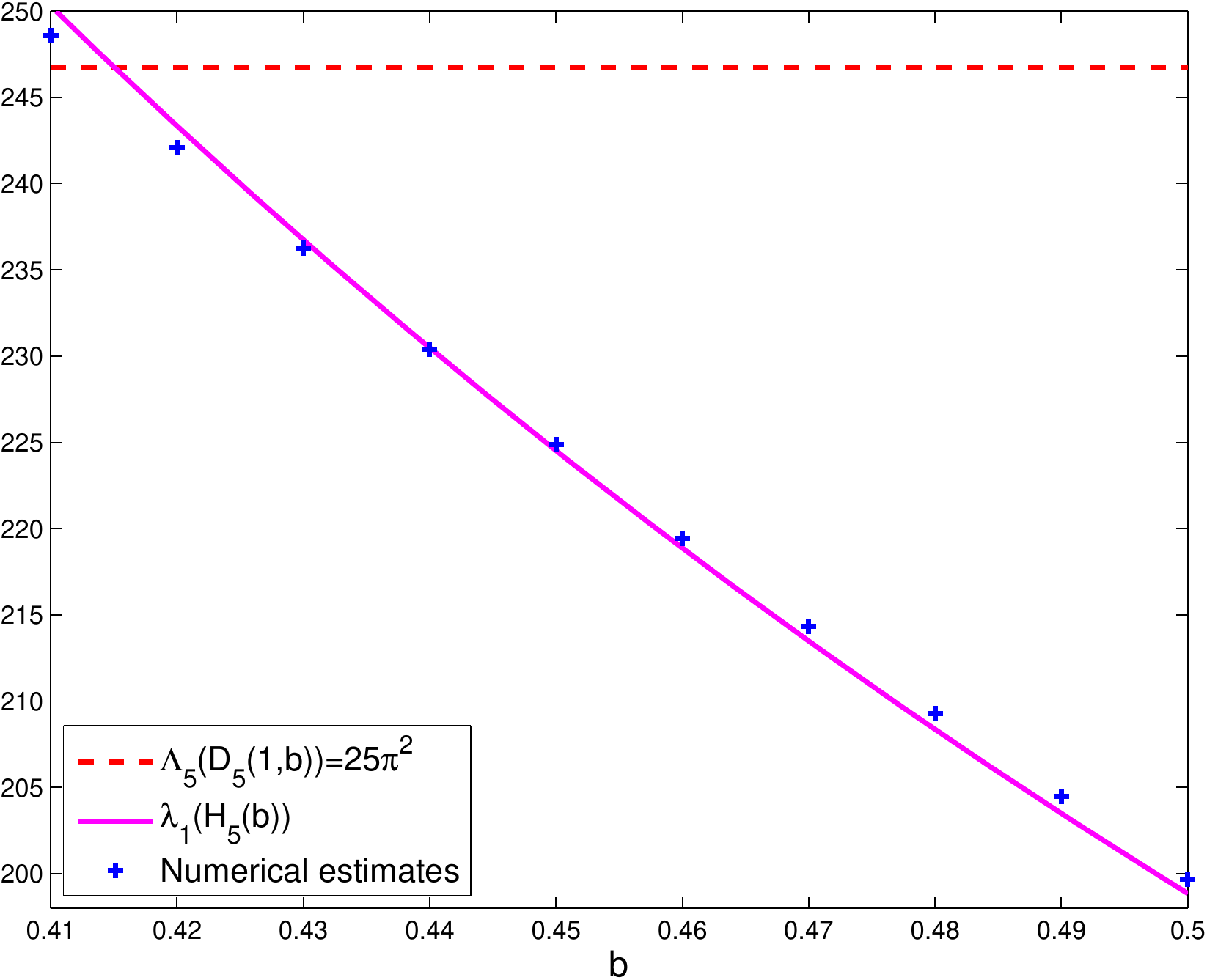}}
    \hspace{2cm}
    \subfigure[$k=5$ and $b$ close to $1$.\label{figEnergyZoom5}]{\includegraphics[height=5cm]{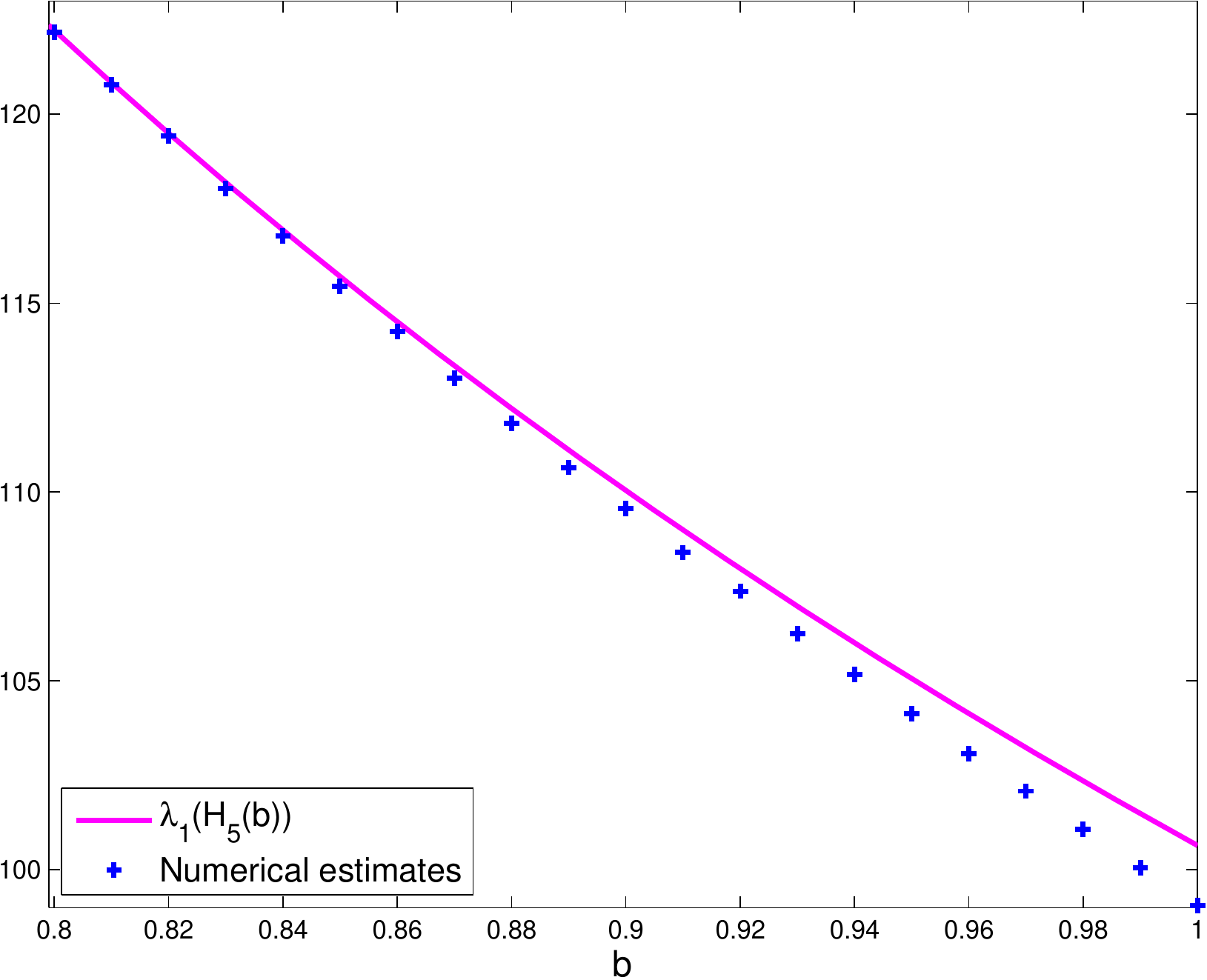}} 
    \caption{Zoom on the energy curves in Section \ref{secNumerics}.\label{figEnergyZoom}}
\end{center}
\end{figure}

This is consistent with the idea that there is some continuity of the minimal partitions with respect to $b\,$, and with the conjecture that the $3$-partition shown on Figure \ref{figNodal6Partb} is minimal. Indeed, if we try to deform this latest partition by splitting each singular point of order four into a pair of singular points of order three, while keeping each domain close the original rectangle, the resulting partition cannot satisfy the equal angle meeting property if all the regular parts of the boundary remain straight lines. This suggests that the boundary of the partition should be curved in the neighborhood of the singular points in such a way that the we keep the equal angle meeting property. This seems to be the case for the partitions represented on Figure \ref{figk3fig2}. This appears more clearly on Figure \ref{figZoom3}, obtained by zooming on the neighborhood of two singular points.  The same phenomenon occurs for $4$-partitions of $\Tor(1,b)$  when $b$ is close to $b_4=1/2\,$, as seen on Figures \ref{figEnergyTrans4} and \ref{figZoom4}, and for $5$-partitions of $\Tor(1,b)$  when $b$ is close to $\bC_5=1/\sqrt{6}$ or to $1\,$, as seen on Figures \ref{figEnergyTrans5}, \ref{figEnergyZoom5}, \ref{figZoom5}, and \ref{figZoom5bis}. As in Figure \ref{figEnergyTrans3}, we see that the numerical method produces partitions with lower energy.
\begin{figure}[h!] \begin{center}
\subfigure[$b=0.72$, $k=3$\label{figZoom3}]{\includegraphics[width=3.5cm]{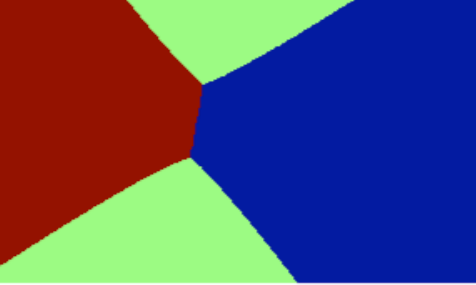} }\hfill
\subfigure[$b=0.51$, $k=4$\label{figZoom4}]{\includegraphics[width=3.5cm]{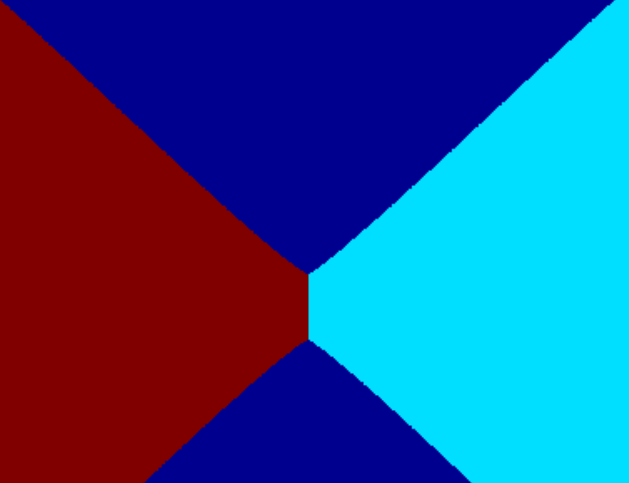} }\hfill
\subfigure[$b=0.42$, $k=5$\label{figZoom5}]{\includegraphics[width=3.5cm]{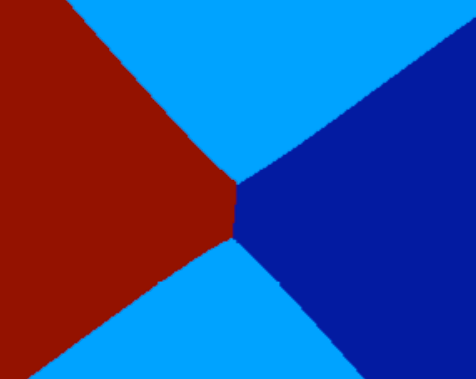} }\hfill
\subfigure[$b=0.99$, $k=5$\label{figZoom5bis}]{\includegraphics[width=3.5cm]{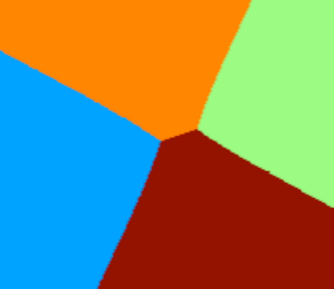} }
\caption{Zoom on the boundary of some of the partition in Section \ref{secNumerics}.
\label{figZoom}}
\end{center}
\end{figure}

For $k \in \{3,\,4,\,5\}$ and $b \in (\bH_k,1]$\,, let us denote by $\mathcal{T}_k(b)$ the hexagonal tiling of $T(1,b)$ given by Proposition \ref{propHexTilings}. According to \cite[Theorem 1.13]{HelHofTer09}, if $\mathcal{T}_k(b)$ is minimal, two neighboring domains of $\mathcal{T}_k(b)$ give a minimal $2$-partition of the doubly hexagonal domain formed by their reunion. More explicitly, if we denote by $2\Hex_k^j(b)\,$, $j\in \{1,\,2,\,3\}\,$, the three doubly hexagonal domains we can extract from $\mathcal{T}_k(b)$\,, we have $\mathfrak{L}_2(2\Hex_k^j(b))=\Lambda_k(\mathcal{T}_k(b))=\lambda_1(\Hex_k(b))$\,. On the other hand, since we are considering $2$-partitions, 
$\mathfrak{L}_2(2\Hex_k^j(b))=\lambda_2(2\Hex_k^j(b))\,$.
Therefore, if for some $b \in (\bH_k,1]$\,, $\mathcal{T}_k(b)$ is minimal, we have
\begin{equation}
\label{eqPCCs}
\lambda_1(\Hex_k(b))=\lambda_2(2\Hex_k^j(b))\quad\mbox{ for } j\in \{1,\,2,\,3\}\,
\end{equation}
(this is a special case of \cite[Proposition 8.3]{HelHofTer09}). Figure \ref{figDiffEV} gives, for $k=3,4,5$, the difference
\[\delta_k^j(b)=\lambda_1(\Hex_k(b))-\lambda_2(2\Hex_k^j(b))\quad \mbox{ for }j \in\{1,\,2,\,3\}\quad\mbox{ and }\quad b\in (\bH_k,1]\,,\]
computed thanks to the finite element library {\sc M\'elina} \cite{melina}.  It shows that in general, Equality \eqref{eqPCCs} is not satisfied, and then the hexagonal tiling $\mathcal{T}_k(b)$ is not minimal. 
Let us note that when $k=4$ and $b=\sqrt3/2$\,, the hexagonal tiling $\mathcal{T}_4(b)$ is regular, and condition \eqref{eqPCCs} is satisfied by symmetry. 
 \begin{figure}[h!] \begin{center}
\subfigure[$k=3$\label{figZoom3}]{\includegraphics[width=5cm]{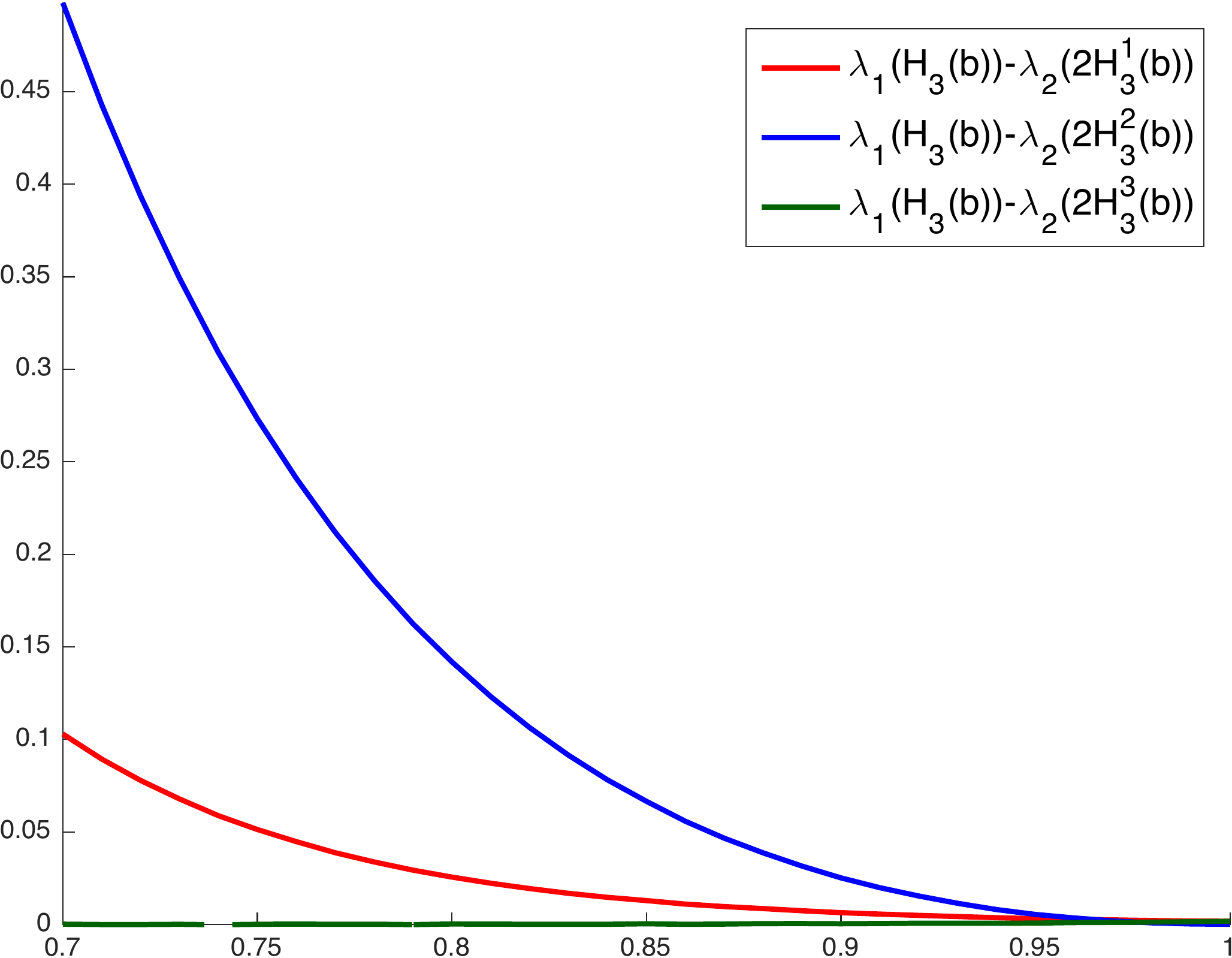} }\hfill
\subfigure[$k=4$\label{figZoom4}]{\includegraphics[width=5cm]{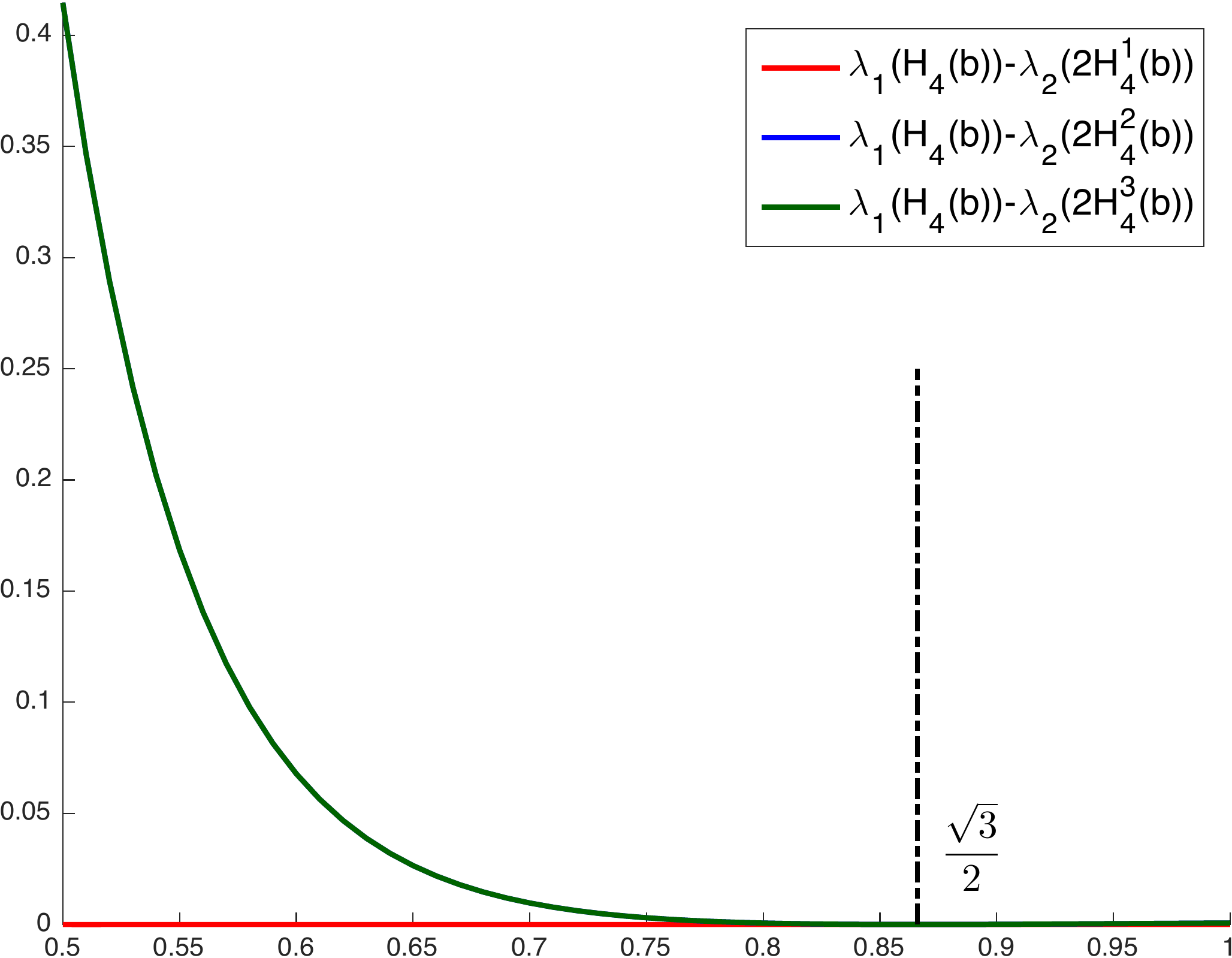} }\hfill
\subfigure[$k=5$\label{figZoom5}]{\includegraphics[width=5cm]{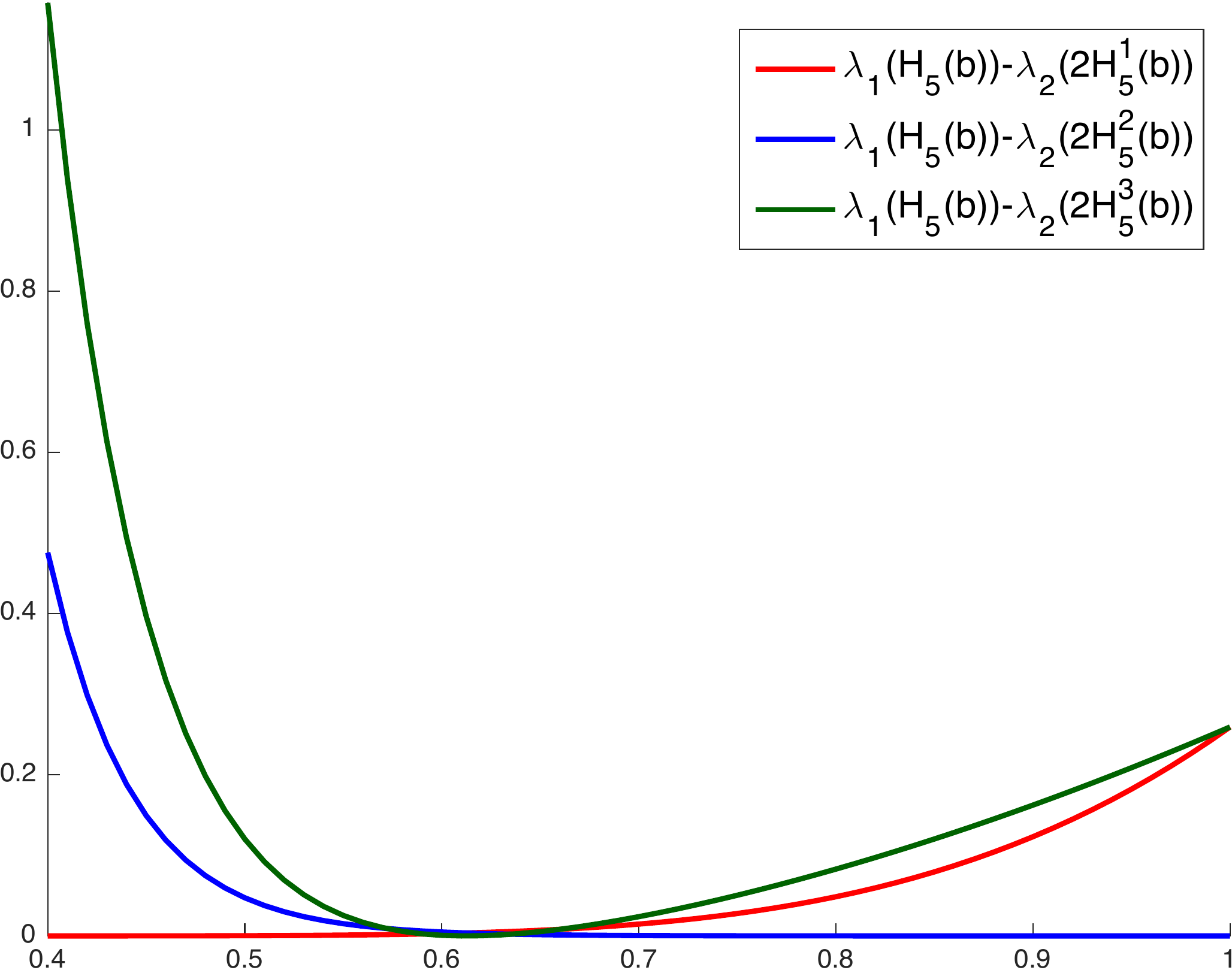} }
\caption{$b\mapsto \delta_k^j(b)$, $j \in\{1,\,2,\,3\}$, $b\in (\bH_k,1]$ for $k=3,4,5$. \label{figDiffEV}}
\end{center}
\end{figure}

\subsection*{Acknowledgments}
 We thank B.~Helffer for his suggestions concerning  Section \ref{secTransitions}, and for his help and interest with this project. We thank  \'E.~Oudet for his help in implementing the optimization algorithm. This work was partially supported by the ANR (Agence Nationale de la Recherche), project OPTIFORM, n$^\circ$ANR-12-BS01-0007-02, and by the ERC, project COMPAT, ERC-2013-ADG  n$^\circ$339958.

{\small
\bibliographystyle{plain} 
\bibliography{BiblioTorus}}
\end{document}